\documentclass[UTF8,10pt]{article}
\usepackage[left=1.91cm,right=1.91cm,top=2.54cm,bottom=2.54cm]{geometry}
\usepackage{amsfonts}
\usepackage{amsmath}
\usepackage{amssymb}
\usepackage{pgfplots}
\usepackage{tikz}
\usetikzlibrary{arrows}
\usepackage{titlesec}
\usepackage{bm}
\usepackage{appendix}
\usepackage{tikz-cd}
\usepackage{mathtools}
\usepackage{amsthm}
\usepackage{enumitem}
\setenumerate[1]{itemsep=0pt,partopsep=0pt,parsep=\parskip,topsep=5pt}
\setitemize[1]{itemsep=0pt,partopsep=0pt,parsep=\parskip,topsep=5pt}
\setdescription{itemsep=0pt,partopsep=0pt,parsep=\parskip,topsep=5pt}
\usepackage[colorlinks,linkcolor=black,anchorcolor=black,citecolor=black]{hyperref}
\usepackage{setspace}
\usepackage[numbers]{natbib}
\usepackage{cases}
\usepackage{fancyhdr}
\usepackage{txfonts}

\newtheorem{Theorem}{Theorem}[section]
\newtheorem{Proposition}{Proposition}[section]
\newtheorem{Lemma}{Lemma}[section]
\newtheorem{Corollary}{Corollary}[section]
\newtheorem{remark}{Remark}[section]
\newtheorem{Definition}{Definition}[section]

\numberwithin{equation}{section}

\numberwithin{equation}{section}
\usepackage{xcolor}

\begin{document}

\bibliographystyle{plain}
\title{\textbf{
On long time dynamics of 1D Schr\"odinger map flows  }}

\author{Ze Li\thanks{School of Mathematics and Statistics, Ningbo University, Ningbo, 315211, Zhejiang, P.R. China. Email: \texttt{rikudosennin@163.com}} }
\date{ }
\maketitle



\begin{abstract}
In this paper, we study the long time dynamics  of small solutions to Schr\"odinger map flows from $\Bbb R$ to Riemannian surfaces. The results are threefold. (i)  We prove that  for  general   Riemannian surface targets the points with some geometric condition  can be completely divided into two categories  according to the sectional curvature so that  the long time dynamics of small solutions of 1D Schr\"odinger map flow near them are described by modified scattering and scattering respectively  for the two categories. (ii)  If the  geometric condition fails, we prove that solutions with slow time growth in frequency space and sharp time decay in physical space,  which scatter  or scatter by a phrase correction, must be trivial.
(iii) We also prove the asymptotic completeness in $L^2$ spaces for 1D SMF into  general Riemannian surface near points  without any geometric assumptions.
Compared with our previous works \cite{li2,li1} on higher dimensional Schr\"odinger map flows where resolution  to finite numbers of radiation terms in energy space was proved for small solutions,  the  results of this work reveal the essentially different and diverse dynamical behaviors of 1D  Schr\"odinger map flows.
\end{abstract}




\setcounter{tocdepth}{1}

\section{Introduction }

Let $(\mathcal{N}, J, h)$ be a K\"ahler manifold. A map $u(t,x):\Bbb R\times \Bbb R^d\to \mathcal{N}$ is called Schr\"odinger map flow (SMF) if $u$ satisfies
\begin{align}\label{hia3}
\begin{cases}
u_t =J\tau(u)\\
u\upharpoonright_{t=0} = u_0.
\end{cases}
\end{align}
Here, the tension field $\tau(u)$ is defined by
\begin{align*}
\tau(u)=\sum^{d}_{j=1}\nabla_{x_j}\partial_{x_j}u,
\end{align*}
where $\nabla$ denotes the induced covariant derivative on the pullback bundle $u^*T\mathcal{N}$.
The equation (\ref{hia3}) with 2D sphere target (introduced by Landau-Lifshitz \cite{ll})  describes the evolution of spin fields in continuum ferromagnets and plays a fundamental role in magnetization dynamics. The 1D SMF with sphere target is closely related to the vortex filament equation, and is a completely integrable system, see e.g. \cite{bv,UT,zgt}. In the late 1990s, geometers and mathematical physicists introduced (\ref{hia3}) with general targets from different views.

We briefly recall the following non-exhaustive list of works on Cauchy problems and near soliton dynamics of SMF.  The local Cauchy theory of SMF was developed by Sulem-Sulem-Bardos \cite{ss}, Ding-Wang \cite{dw}, McGahagan \cite{m}. The global theory for Cauchy problem was pioneered by Chang-Shatah-Uhlenbeck \cite{csu}, Nahmod-Stefanov-Uhlenbeck \cite{nsu},  Bejenaru \cite{b}, Ionescu-Kenig \cite{ik1,ik2}. The global well-posedness theory in critical Sobolev spaces  for target $\Bbb S^2$ was completed by Bejenaru-Ionescu-Kenig \cite{bik}($d\ge 4$) and Bejenaru-Ionescu-Kenig-Tataru \cite{bikt} ($d\ge 2$). See also Smith \cite{smith} for the conditional regularity in critical spaces. In the other direction, the stability/instability of ground state harmonic maps and threshold scattering in equivariant class were studied by works of Gustafson, Kang, Tsai, Nakanish \cite{gnt,gkt},  Bejenaru, Ionescu, Kenig, Tataru \cite{bikt1,bt}.  Self-similar solutions were studied by Ding-Tang-Zeng \cite{dtz}, Germain-Shatah-Zeng \cite{gsz}, Banica-Vega \cite{bv}.  The type II singularity formulation was achieved by  Merle-Raphael-Rodnianski \cite{mrr} and Perelman \cite{p}.

For the global theory, additional efforts are needed for general targets than the constant curvature case. In $d=1$,  Rodnianski-Rubinstein-Staffilani \cite{rrs} proved global regularity for (\ref{hia3}) and general  K\"ahler manifolds $\mathcal{N}$.
In our previous works \cite{li1,li2}, we proved the global well-posedness and long time dynamics of  SMF from $\Bbb R^d$ with $d\ge 2$ into compact K\"ahler manifolds with small  initial data in critical Sobolev spaces. In this work, we aim to study the long time dynamics of small solutions to (\ref{hia3}) with $d=1$.

 It is well-known that the 1D SMF with $\Bbb S^2$ target under the Hasimoto transform  turns out to be the 1D cubic nonlinear Schr\"odinger equation (NLS), i.e.
\begin{align}\label{7800}
i\partial_t v+\Delta v+|v|^2v=0.
\end{align}
The long dynamics of 1D cubic NLS with small data were  proved by Hayashi-Naumkin \cite{hn} to be modified scattering, i.e. scatters under a phrase correction.  But the long time dynamics of the original  map itself were not clear.  In  this work, we prove that the long time dynamics of small solutions 1D Schr\"odinger map flow depend on both the  curvature and an intrinsic geometric vanishing condition.

Before stating the main theorem and illustrating the  main ideas of this paper, we give following notations.
For $\psi\in \mathcal{S}(\Bbb R)$, denote its Fourier transform by
\begin{align*}
 \widehat{\psi}(\xi)=\mathcal{F}[\psi](\xi)=\frac{1}{(2\pi)^{\frac{1}{2}}}\int_{\Bbb R} e^{-ix\xi} \psi(x)dx.
 \end{align*}
The inverse Fourier transform is given by
\begin{align*}
 \mathcal{F}^{-1}[\psi](x)=\frac{1}{(2\pi)^{\frac{1}{2}}}\int_{\Bbb R} e^{ix\xi} \psi(\xi)d\xi.
 \end{align*}
For $l,m\in \Bbb R$ we define  the weighted Sobolev space  $H^{l,m}(\Bbb R)$ by
 \begin{align*}
 H^{l,m}:=\{f\in \mathcal{S}'(\Bbb R):\|f\|_{H^{l,m}}:=\|\langle x\rangle^{m}\langle D\rangle^{l}f\|_{L^2(\Bbb R)}<\infty\},
 \end{align*}
where $\langle x\rangle:=(1+|x|^2)^{\frac{1}{2}}$, $\langle D\rangle:=\mathcal{F}^{-1}\langle \xi\rangle \mathcal{F}f$.

\begin{Definition}
Assume that $\mathcal{N}$ is a Riemannian surface with metric $h(z,\bar{z})d\bar{z}dz$,  and $Q\in \mathcal{N}$ is a given point which corresponds to $z=0$.
We say $Q$ is an intrinsic vanishing  point if
\begin{align}\label{KEY}
[\ln   {h}]_{z}(0)[\ln {h}]_{z\bar{z}}(0)-[\ln  {h}]_{z\bar{z}z}(0)=0.
\end{align}
\end{Definition}

We will prove whether $Q$ is an intrinsic vanishing  point or not is independent of the choice of   local complex coordinates near $Q$. To be precise, we have
\begin{Lemma}\label{ADs}
Assume that $\mathcal{N}$ is a Riemannian surface with metric $h(z,\bar{z})d\bar{z}dz$,  and $Q\in \mathcal{N}$ is an intrinsic vanishing  point in the sense of complex coordinate $z$.
Let $\eta$ be another local complex coordinate near $Q$ such that $\eta=f(z)$, where $f$ is a holomorphic function with $f(0)=0$, $f_z(0) \neq 0$. Then $Q\in \mathcal{N}$ is also an intrinsic vanishing  point in the sense of coordinate $\eta$.
\end{Lemma}
In other words, the notion  {\it intrinsic vanishing  point} is truly an intrinsic notion which is free of choice of  local complex coordinates.

The following result implies that  the existence of intrinsic vanishing  point.
\begin{Proposition}\label{AYs}
Assume that $\mathcal{N}$ is a closed  oriented  Riemannian surface  with non-zero Eluer-Poincare characteristic. Then there exists at least one intrinsic vanishing  point.
\end{Proposition}

We also have many other concrete examples of {\it intrinsic vanishing  point}.  In fact, if the metric near $Q$ is locally rotationally invariant, i.e.  $h=h(|z|^2)dz d\bar{z}$ with $z(Q)=0$, then $Q$ is an  intrinsic vanishing  point.  Particularly, all points of
$\Bbb S^2$ and $\Bbb H^2$  are {\it intrinsic vanishing  points}.  The proof of these facts and other examples are presented in Section 5.

Our main theorems are as follows:
\begin{Theorem}\label{XS2}
Let $d=1$. Assume that $\mathcal{N}$ is a Riemannian surface with metric $h(z,\bar{z})d\bar{z}dz$,  and $Q\in \mathcal{N}$ is a  given intrinsic   vanishing  point  in Def. 1.1. Let $K(Q)$ be the sectional curvature at $Q$.
\begin{itemize}
  \item Assume that $K(Q)\neq 0$. Then there exist a universal constant $n_*\in \Bbb N$ and a sufficiently small constant  $\epsilon_* >0$  depending only on $\mathcal{N}$, and a well chosen local complex coordinate $w$ of $\mathcal{N}$ near $Q$ with $w(Q)=0$,  such that if  $u_0:\Bbb R\to \mathcal{N}$  satisfies
\begin{align}\label{suption}
\|w_0\|_{{H}^{3,1}_{x}\cap H^{n_*}_x}\le \epsilon_*,
\end{align}
then (\ref{hia3}) with initial data $u_0$ evolves into a unique global solution $u(t,x)$, and  there exists a $\Bbb C$ valued function $U\in L^{\infty}(\Bbb R,\langle \zeta\rangle^2d\zeta)$ such that as $t\to\infty$ there holds
\begin{align}\label{asy}
\lim\limits_{t\to\infty}\|  e^{ic\int^t_{1}\frac{1}{2\tau}|\zeta\widehat{w}(\zeta)|^2d\tau}\widehat{w}(t,\zeta) -e^{-it\zeta^2}  U(\zeta) \|_{L^{\infty}_{\zeta}(\Bbb R,\langle \zeta\rangle^2d\zeta)\cap L^2_{\zeta}(\Bbb R, d\zeta)}=0.
\end{align}
Here, writing the  metric of $\mathcal{N}$ near $Q$ under the local complex coordinate $w$ as $h(w,\bar{w})dw d\bar{w}$, the constant $c$ in (\ref{asy}) is given by $c=-\frac{1}{2}K(Q)h_0$, $h_0=h(0,0)$. Moreover,  $w_0(x)$ and $w(t,x)$ are the local complex coordinates of $u_0(x)$ and $u(t,x)$ respectively.

Moreover, for any local complex coordinate $z$ of $\mathcal{N}$ near $Q$ with $z(Q)=0$, we have the sharp decay estimate
\begin{align}\label{zGGG}
\| z(t,x)\|_{W^{2,\infty}_x} +\|\nabla_x \partial_xu\|_{L^{\infty}_x}+\| \partial_xu\|_{L^{\infty}_x}\lesssim t^{-\frac{1}{2}}.
\end{align}
Furthermore, there exist some function  $\psi\in  \langle x\rangle^{-2}L^{\infty}_x(\Bbb R) $,  and some constant $\varsigma>0$ such that $z(t,x)$ has the refined asymptotic expansion for  $t\ge 1$
\begin{align}\label{GGG}
z(t,x)=\frac{1}{(2it)^{\frac{1}{2}}}\psi(\frac{x}{2t})\exp\left(i\frac{|x|^2}{4t}+ \frac{ic}{2}| \frac{x}{2t} \psi(\frac{x}{2t})|^2\ln (2t) \right)+\mathbf{R}(x,t),
\end{align}
where the remainder term $\mathbf{R}(x,t)$ satisfies
\begin{align*}
 \|\mathbf{R}(x,t)\|_{L^{\infty}_x}\lesssim t^{-\frac{1}{2}-\varsigma}, \mbox{ } \|\mathbf{R}(x,t)\|_{L^{2}_x}\lesssim  t^{-\varsigma}.
\end{align*}
  \item  Assume that  $K(Q)=0$. Let $z$ be any local complex coordinate  of $\mathcal{N}$ near $Q$ with $z(Q)=0$. Then there exist  a sufficiently small constant  $\epsilon_* >0$  and a sufficiently large constant $n_*>0$ depending only on $\mathcal{N}$ such that if  $u_0:\Bbb R\to \mathcal{N}$  satisfies (\ref{suption}),
then (\ref{hia3}) with initial data $u_0$ evolves into a unique global solution $u(t,x)$, and  there exists a $\Bbb C$ valued function $z_{\infty}\in { H}^1(\Bbb R)$ such that as $t\to\infty$ there holds
\begin{align}\label{asy2}
\lim\limits_{t\to\infty}\|z(t,x)-e^{it\Delta}z_{\infty}\|_{{H}^{1}_{x}}=0.
\end{align}
Here $z_0(x)$ and $z(t,x)$ are the local complex coordinates of $u_0(x)$ and $u(t,x)$ respectively.
Moreover,  the above sharp decay estimates (\ref{zGGG}) hold as well.
\end{itemize}
\end{Theorem}

If  $Q\in \mathcal{N}$ is  not an intrinsic vanishing  point  in Def. 1.1, generically neither modified  scattering  nor scattering could be expected. To be precise, we have
\begin{Theorem}\label{XS3}
Let $d=1$. Assume that $\mathcal{N}$ is a Riemannian surface, and $Q\in \mathcal{N}$ is not an intrinsic vanishing point in the sense of Def. 1.1.
Then there exists  a sufficiently small constant  $\epsilon_* >0$  depending only on $\mathcal{N}$,  such that,  if there exists a local complex coordinate $z$ of $\mathcal{N}$ near $Q$ with $z(Q)=0$ under which  the solution of (\ref{hia3})   satisfies
\begin{align}
\|z_0\|_{{H}^{3,1}_{x}}+\sup_{t\ge 0} t^{\frac{1}{2}}\| {z}(t)\|_{W^{1,\infty}_{x}}&\le \epsilon_*\label{suption2py}\\
\sup_{t\ge 0}[\|  \widehat{z}(t,\zeta) \|_{L^{\infty}_{ \zeta}}+ +t^{-\beta}\sum^{1}_{j=0}\|(x+2it\partial_x)\partial^j_xz\|_{L^2_{x}}]&\le \epsilon_*\label{suption2p}
\end{align}
for some $\beta\in (0,\frac{1}{12})$, and  there exist  a $\Bbb C$ valued function $U(\xi)$, a real valued function $\mathcal{D}$
for which there holds
\begin{align}
 \lim\limits_{t\to\infty}[\|\langle \zeta\rangle [ e^{i\mathcal{D}(\zeta,t)} \widehat{z}(t,\zeta) -e^{-it\zeta^2} U(\zeta) ]\|_{L^{2}_{\zeta}}+\| \langle \zeta\rangle^{\frac{1}{2}}[e^{i\mathcal{D}(\zeta,t)} \widehat{z}(t,\zeta) -e^{-it\zeta^2} U(\zeta)]\|_{L^{4}_{\zeta}}]=0,\label{suption2pp}
\end{align}
then $z(t)\equiv 0$.
Here $z_0(x)$ and $z(t,x)$ are the local complex coordinates of $u_0(x)$ and $u(t,x)$ respectively.
\end{Theorem}

For the asymptotic completeness problem, we have the following theorem for general Riemannian surface without assumption (\ref{KEY}).
\begin{Theorem}\label{XS4}
Let $d=1$. Assume that $\mathcal{N}$ is a  Riemannian surface with metric $h(z,\bar{z})d\bar{z}dz$,  and $Q\in \mathcal{N}$ is a given point. Let $z$ be any given local complex coordinate of $\mathcal{N}$ near $Q$ with $z(Q)=0$.
Then there exist a universal constant $m\in \Bbb N$ and a sufficiently small constant  $\epsilon_* >0$  depending only on $\mathcal{N}$, such that if  $\psi:\Bbb R\to \Bbb C$  satisfies
\begin{align}\label{assx}
\sum^{m}_{j=0}\|\langle x\rangle \partial^j_x \psi\|_{L^{2}_x\cap L^{\infty}_x}\le \epsilon_*,
\end{align}
then there exists an initial data $u_0$ which   evolves into a  global solution $u(t,x)$ of (\ref{hia3})  so that as $t\to\infty$   there holds
\begin{align}
&\lim\limits_{t\to\infty}\|z(t,x)-\frac{1}{(2it)^{\frac{1}{2}}}\psi(\frac{x}{2t})e^{i\frac{|x|^2}{4t}}e^{ \frac{ic}{2}| \frac{x}{2t} \psi(\frac{x}{2t})|^2\ln (2t)}\|_{ H^1_x}=0.\label{asy3}
\end{align}
Here  $z(t,x)$ denotes the local complex coordinate  of $u(t,x)$, $c=-\frac{1}{2}K(Q)h_0$, $h_0=h(0,0)$.
\end{Theorem}

 We make a few remarks on the above three theorems. \\
 (i) All the assumptions (\ref{suption2py}), (\ref{suption2p}), (\ref{suption2pp}) in Theorem \ref{XS3} are fulfilled by solutions in Theorem \ref{XS2}. \\
 (ii) The intrinsic vanishing condition (\ref{KEY}) indeed corresponds to the vanishing of  4 order terms $(\nu_2\bar{w}w+\nu_3\bar{w}^2)(\partial_x w)^2$ in the equation (\ref{3equation1}).  Theorem \ref{XS3} implies that solutions of slow growth in frequency space and sharp decay in physical space, which scatter   or scatter by a phrase correction, must be trivial, if the 4   order terms $(\nu_2\bar{w}w+\nu_3\bar{w}^2)(\partial_x w)^2$ do not vanish. This fact is not  obvious since the leading part of equation (\ref{3equation1})  is the cubic term $\nu_1\bar{w}(\partial_x w)^2$, and it seems that the 4 order terms would not have a big effect on the long time dynamics of small solutions at the first glance. In this view, Theorem \ref{XS3} indeed reveals that the 4 order terms also play a vital role in determining the whole dynamics.\\
  (iii) However, Theorem \ref{XS4}  shows the existence of wave operators is independent of the intrinsic vanishing condition. A further calculation implies that the key condition the  solutions constructed in Theorem \ref{XS4} fail to fulfill is the slow growth assumption in  (\ref{suption2p}). This will be clearly seen in the proof of  Theorem \ref{XS3}. \\
(iv) As far as we know, it  is a new phenomenon revealed  here that high  order terms also have a dominant effect on the long time dynamics of small solutions.

\begin{remark}
Note that Theorem 1.1 implies the dynamical behavior is determined by the curvature at the given point. For example, if $\mathcal{N}$ has two intrinsic vanishing  points $Q_1,Q_2$ with $K(Q_1)=0$ and $K(Q_2)\neq 0$, then the 1D SMF near $Q_1$ scatters while the 1D SMF near $Q_2$ scatters with a phrase correction. Here we have a more concrete example, say $\mathcal{N}$  has a metric $h(z,\bar{z})=e^{\lambda(z,\bar{z})}$, $\lambda(z,\bar{z})=c_1(z+\bar{z})+c_2(z^2+\bar{z}^2)+c_3(z^3+\bar{z}^3)+c_4|z|^4$ with $z\in \Bbb C$. Then both $z=0$ and $z=1$ are  intrinsic vanishing  points  if $c_1+2c_2+3c_3+2c_4=1$. And $K(0)=0$, $K(1)\neq0$ if $c_4\neq 0$.
\end{remark}

\begin{remark}
For the classical 2d sphere target, i.e. the 1D Landau-Lifshitz equation,  Theorem \ref{XS2} is still new, since the previous method depending on Hasimoto transform did not give long time dynamics of the map itself.  When $(\mathcal{N},h dz d\bar{z})$ is the 2d sphere, the metric writes as  $h(|z|^2)=({1+|z|^2})^{-2}$ under the stereographic projection. If $(\mathcal{N},h dz d\bar{z})$ is the hyperbolic plane, using the Poincare disk model, the metric is given by  $h(|z|^2)=4({1-|z|^2})^{-2}$. Hence, every point of standard  $\Bbb S^2$ or $\Bbb H^2$ is an intrinsic vanishing point, for which  Theorem \ref{XS2} applies. Moreover, for the $\Bbb S^2$ target and $\Bbb H^2$ target, the local complex coordinate  $w$ in  Theorem \ref{XS2} can be directly chosen to be the  stereographic projection coordinate and the coordinate of Poincare disk model respectively.
\end{remark}

\begin{remark}
For long time dynamics of small solutions to (\ref{hia3}) in $d\ge 2$, our previous works \cite{li1,li2} proved that the map splits into finite numbers of radiation terms  plus an asymptotic vanishing error in the energy space, and converges to the constant map in uniform distance.  However, the dichotomous result of  Theorem 1.1 in this work reveals that both scattering and modified scattering occur for 1D SMF, and it depends on the curvature.
\end{remark}

\begin{remark}
Modifications of arguments here can give better regularity index than that stated in  (\ref{suption}).   The choice of $n_*$ is largely casual. Besides (\ref{asy3}), we also have
\begin{align*}
 \|z(t,x)-\frac{1}{(2it)^{\frac{1}{2}}}\psi(\frac{x}{2t})e^{i\frac{|x|^2}{4t}}e^{ \frac{ic}{2}| \frac{x}{2t} \psi(\frac{x}{2t})|^2\ln (2t)}\|_{ L^2_x}\lesssim t^{-\frac{1}{2}}.
\end{align*}
\end{remark}

\begin{remark}
If $h(z,\bar{z})=1$ in Theorem 1.1, then $\mathcal{N}$ is the complex plane with the standard metric. In this case, the solution of (\ref{hia3}) is given by
$u= e^{it\Delta} u_0$.
\end{remark}

\subsection {Main ideas of Theorem \ref{XS2}}

Let's begin with some known facts on 1D NLS. Recall that for 1D NLS of the form $i\partial_t v+\Delta v=|v|^{p-1}v$, the solution with small initial data scatters for $p>3$, and for $1<p<3$ the only solution that scatters in $L^2(\Bbb R)$ is zero, see e.g. \cite{tao2}. The case $p=3$ is the critical value.  Consider the general NLS with cubic nonlinearity:
\begin{align}\label{z1}
i\partial_t v+\Delta v=\mathcal{N}( \partial^{j_1}_x {v}^{\pm}, \partial^{j_2}_x {v}^{\pm}, \partial^{j_3}_x {v}^{\pm}), \mbox{ }j_1,j_2,j_3\in \{0,1\},
\end{align}
where $v^+=v, v^{-}=\bar{v}$. The typical candidate of  (\ref{z1})  is
\begin{align}\label{z2}
i\partial_t v+\Delta v=a_1  \bar{v}v^2+a_2  v^3+a_3  {\bar v}^3+a_4  {\bar v}^2v.
\end{align}
In dimension 1, except for the case $a_2=a_3=a_4=0$, the long time dynamics of (\ref{z2}) are not completely clear even in small data regime, see e.g. \cite{mp} and reference therein. The known advantage of  $\bar{v}v^2$ compared with other three combinations is that  $\bar{v}v^2$ enjoys gauge invariance. And we also remark that the derivative nonlinear terms in (\ref{z1}) sometimes give better behaviors in the low frequency regime. For more detailed discussions on model (\ref{z1})-(\ref{z2}) see \cite{go,gms1,gms2,hn,IT,K,ip2,o,LS,tao2,sh}.

Let's recall the intrinsic view of Hasimoto transform. Let $\mathcal{N}$ be a Riemannian surface, and let $\{e,Je\}$ be a moving frame for $u^*T\mathcal{N}$. The Hasimoto transform indeed can be viewed as transforming
the map $u$ into scalar fields $\psi_{t,x}=\langle \partial_{t,x}u,e\rangle+i\langle \partial_{t,x}u,Je\rangle$ on the trivial vector bundle over $\Bbb R^2$ with fiber $\Bbb C$ under the gauge fixing assumption that $\nabla_xe=0$. If the target $\mathcal{N}$ is a 2d sphere, then $\psi_x$ solves the equation (\ref{7800}), see e.g. \cite{nsvz}. By \cite{rrs} and additional observations, one can check  $\psi_x$ solves the general cubic NLS system
\begin{align}\label{7800}
i\partial_t v+\Delta v=C_{\pm,\pm,\pm}  {v}^{\pm} {v}^{\pm} {v}^{\pm},
\end{align}
if $\mathcal{N}$ is a  Hermitian symmetric space. Here, we denote ${v}^{+}=v$, ${v}^{-}=\bar{v}$, $C_{\pm,\pm,\pm}$ are constant matrices.  If $\mathcal{N}$ is a  general  K\"ahler manifold without symmetry,  \cite{rrs} deduces that $\psi_x$ satisfies
\begin{align}\label{7801}
i\partial_t v+\Delta v=P+Q,
\end{align}
where $|P|\lesssim |v|^3$, $|Q|\lesssim (\int_{\Bbb R}|v|^3dx)|v|$. From the conservation of energy, $Q$ is indeed a quadratic term in the view of time decay.
According to previous discussions on model (\ref{z2}), the gauged equations (\ref{7800}), (\ref{7801}) shall be very tough to handle in the study of long time dynamics.

Another way to fix  the gauge of SMF is the caloric gauge introduced by Tao \cite{tao1,tao}. It is powerful in the study of critical regularity problem and long time dynamics, see for instance \cite{bikt,li1,li2}. Let $U(s,t,x)$ be the solution to heat flow equation with initial data $u(t,x)$, and $\{e_{\alpha},Je_{\alpha}\}$ be the moving frame of $U^*T\mathcal{N}$. The caloric gauge assumes $\nabla_s e_{\alpha}=0 $. Denote $\psi_{t}, \psi_x$ the differential fields associated with the frame $\{e_{\alpha},Je_{\alpha}\}$ and $A_x,A_t$ the corresponding connection coefficients. In dimension 1, when $s=0$, $\psi_x$ roughly  solves
\begin{align}\label{7802}
i(\partial_t+A_t) v+(\partial_x+A_x)^2v=0.
\end{align}
(Note that the cubic curvature term vanishes since $d=1$.) Here, the  connection coefficient $A_{t,x}$ is approximately $\int^{\infty}_s O(\psi_{t,x})(D_x\psi_x)ds'$.
(\ref{7802}) suffers from the same structure problem as (\ref{7801}). On the other hand,
the nonlinearity terms of (\ref{7802}) seem to be  cubic, but in the study of time decay of $\|\psi_x\|_{L^{\infty}_x}$, one will find
there is no enough decay in  $s$ to make the integral  $\int^{\infty}_s O(\psi_{x})(D_x\psi_x)ds'$ and $\int^{\infty}_s \psi_s ds'$ converge in the desired $L^p$ space.
See Appendix A
for a concrete discussion.

From the above discussions, assuming neither $\nabla_xe=0$ nor $\nabla_se=0$ is a good choice of gauge fixing for our problem.

Another way to write (\ref{hia3}) is to use the local complex coordinate.
In fact,  assume that  $\mathcal{N}$ is a Riemannian surface with metric $h(z,\bar{z})dzd\bar{z}$. Then (\ref{hia3}) is written as
\begin{align}\label{equation1}
\begin{cases}
i\partial_t z+\Delta z=\frac{h_{z}(z,\bar{z})}{h(z,\bar{z})}\partial_xz\partial_xz\\
z\upharpoonright_{t=0} = z_0.
\end{cases}
\end{align}
We observe that  (\ref{equation1}), which is SMF under the complex coordinates (or in the frame  $\{\frac{\partial}{\partial z},\frac{\partial}{\partial \bar{z}}\}$), has several key good structures. Expanding $h_{z}/h$ at $z=0$ gives
\begin{align}\label{777}
\begin{cases}
i\partial_t z+\Delta z=&c_0 \partial_xz\partial_xz+c_1 \bar{z}\partial_xz\partial_xz+c_2 {z}\partial_xz\partial_xz+  c_3 {z}^2\partial_xz\partial_xz+  c_4 {z}\bar{z}\partial_xz\partial_xz\\
&+   c_5 {\bar{z}}^2\partial_xz\partial_xz+O(|z|^3)\partial_xz\partial_xz\\
z\upharpoonright_{t=0}  = z_0.&
\end{cases}
\end{align}
The quadratic term $c_0 \partial_xz\partial_xz$, the cubic term $c_2 {z}(\partial_xz)^2$ and the 4 order term $c_3 {z}^2(\partial_xz)^2$ can be removed by letting
\begin{align*}
w:=z+\gamma_1 z^2+\gamma_2 z^3+\gamma_3z^4,
\end{align*}
for well chosen constants $\gamma_1,\gamma_2,\gamma_3$.
Indeed $w$ fulfills
\begin{align}
i\partial_t w+\Delta w&= \nu_1 \overline{w}\partial_x w\partial_xw +  \nu_2 {w}^2\partial_xw\partial_xw
 +  \nu_3 {w}\overline{w}\partial_xw\partial_xw\nonumber\\
 &+  \nu_4  \overline{w}^2\partial_xw\partial_xw+O(|w|^3)\partial_xw\partial_xw.
\label{3equation1}
\end{align}
Note that (\ref{equation1}) is invariant under the holomorphic transformation. This is analogous to the aforementioned freedom of gauge fixing.

The leading cubic nonlinearity of (\ref{3equation1}) is $\overline{w}(\partial_xw)^2$.  In the physical space,  we find that this nonlinearity   commutes not that badly with the operator $L:=ix-2t\partial_x$ ( the generator of Galilean transformation),  although (\ref{equation1}) is not invariant  under the Galilean transformation.
This observation inspires us to choose  (\ref{3equation1}) as the master equation.

However, (\ref{3equation1}) suffers from two key problems. One is the serious derivative loss compared with (\ref{7800}) or (\ref{7801}).  In fact, to obtain global bounds of solution, we will mainly depend on   mass and two vector fields, one is
$S:=2t\partial_t+ x\partial_x$, the infinitesimal generator of scaling, and the other is $L:= ix-2t\partial_x$ arising from the Galilean transformation. Directly calculating
$\|w\|_{L^2_x}$, $\|Sw\|_{H^{k}_x} $ , $\|L w\|_{L^{2}_x}$ will face serious derivative loss and time growth. The solution consists of two observations. One observation is  that $Su$ is an intrinsic geometric quantity valued in $u^*T\mathcal{N}$. So, instead of estimating $\|Sz\|_{H^{k}_x} $, we dominate the geometric quantity $\|\nabla^{l}Su\|_{L^2_x}$, where $\nabla$ denotes the induced covariant derivative. The other observation is  that
for a well chosen weight $H(w,\bar{w})$ the new functional
\begin{align*}
\mathcal{H}(w):=\int_{\Bbb R}H(w,\bar{w}) |w|^2dx
\end{align*}
fulfills a closed energy estimate without any loss of derivatives. In fact, it
satisfies an improved energy estimate
\begin{align}\label{67bn}
\frac{d}{dt} \mathcal{H}(w)\lesssim \|w\|^3_{W^{1,\infty}_x}\mathcal{H}(w).
\end{align}
Meanwhile,  $\mathcal{H}(w)$ is comparable to $\|w\|^2_{L^2_x}$.
Note that  the power of  $\|w\|_{W^{1,\infty}_x}$ in (\ref{67bn}) is 3 rather than 2. So it provides a time uniform bound of $\|w\|_{L^2_x}$ rather than logarithm upper-bound.
The same technique also  provides a slow time growth upper-bound  of $\|Lw\|_{L^2_x}$ with additionally assuming the intrinsic vanishing condition.  Indeed for the well-chosen weight $H(w,\bar{w})$, the functional, which is comparable to $\|Lw\|_{L^2_x}$, defined by
\begin{align*}
\mathcal{L}(w)=\int_{\Bbb R} |Lw|^2 H(w,\bar{w}) dx
\end{align*}
satisfies a closed energy estimate without loss of derivative provided that the intrinsic vanishing condition holds.

The other  key problem  (\ref{3equation1}) suffering  from  is the 4 order terms $(\nu_3w\bar{w}+\nu_4\bar{w}^2)(\partial_x w)^2$.
These two terms have a relatively very large time growth in estimating $\|Lw\|_{L^2_x}$. And we will see in the proof of Theorem \ref{XS3}, that generally one can not expect modified scatting or scattering if these two are not vanishing. Fortunately, we find that the geometric condition (\ref{KEY}), i.e., intrinsic vanishing point condition, can kill these two terms.

We make several additional remarks on the proof to Theorem \ref{XS2}. The proof relies on the framework of  space-time resonance analysis  developed by  Germain-Masmoudi-Shatah\cite{gms1,gms2,gms3}. The way of proving modified scattering by stationary phrase method is inspired by  Ionescu-Pusateri \cite{ip} and Kato-Pusateri \cite{ip2}.

\subsection{Main idea of Theorem \ref{XS3}}

Suppose that  $Q$ is not an intrinsic vanishing point, i.e.
\begin{align}\label{Gyi}
[\ln   {h}]_{z}(0)[\ln {h}]_{z\bar{z}}(0)-[\ln  {h}]_{z\bar{z}z}(0)\neq 0.
\end{align}
The proof of Theorem \ref{XS3} will be divided into two cases.\\
{\bf Case 1.} Assume that $c_5\neq 0$ in (\ref{777}). In this case,
letting $f=e^{-it\Delta}z$, its Fourier transform  $\widehat{f}$  solves the equation (\ref{ccJ3}). When bounding $\|\widehat{f}\|_{L^{\infty}_{\sigma}}$, we observe that the 4 order term associated with $\phi_{++}$, i.e.  $c_5\bar{z}^2(\partial_xz)^2$ part,   implicitly contains a model of logarithmic growth in time.
In fact, delicate analysis via stationary phrase method shows
\begin{align*}
\|\widehat{f}(t,\zeta)\|_{L^{\infty}}\ge  C |c_5|\int^{t}_1\frac{1}{\tau}\zeta^2|\widehat{f}|^4d\zeta d\tau-O_{L^{\infty}}(1),
\end{align*}
where $C>0$ is a universal constant. This, together with the convergence claimed in Theorem \ref{XS3} and the almost conservation of mass, implies $z\equiv 0$.

{\bf Case 2.} Assume that $c_5=0$ in (\ref{777}). In this case, (\ref{Gyi}) shows $c_0\neq 0$.
For well chosen constants $\{\kappa_{i}\}^3_{i=1}$, letting $w=z+\kappa_1z^2+\kappa_2 z^3+\kappa_3 z^4$,   ${w}$  solves the equation (\ref{fff}), i.e. the terms $(\partial_x z)^2$, $z(\partial_x z)^2$, $z^2(\partial_x z)^2$, ${\bar z}^2(\partial_x z)^2$ vanish together. And particularly, $\kappa_1\neq 0$ because of $c_0\neq 0$.  Let $g=e^{-it\Delta }w$. Its Fourier transform  $\widehat{g}$ solves  (\ref{final1}).   Then  delicate analysis via stationary phrase method implies
\begin{align*}
\|\partial_{\sigma} \widehat{g}(t,\zeta)\|_{L^{2}_{\sigma}}\lesssim t^{\frac{1}{2}-\nu+3\beta}
\end{align*}
with $\nu\in(3\beta,\frac{1}{4})$. This further gives
\begin{align*}
|\kappa_1|\|tz\partial_x z\|_{L^{2}_{x}}\lesssim t^{\frac{1}{2}-\nu+3\beta}.
\end{align*}
Meanwhile, the
convergence of Theorem \ref{XS3} implies as $t\to \infty$
\begin{align*}
\|tz\partial_x z\|_{L^{2}_{x}}\sim t^{\frac{1}{2}} \|\zeta^{\frac{1}{2}} U\|^2_{L^4_{\zeta}}.
\end{align*}
Then we conclude $U=0$, and further $z\equiv 0$ by  the almost conservation of mass.

\subsection {Main ideas of Theorem \ref{XS4}}
Theorem \ref{XS4} aims to prove the asymptotic completeness, i.e. to find an appropriate  initial data to match the given final state. The whole proof is divided into two parts, the first part is to prove  asymptotic completeness for model (\ref{3equation1}), and the second is to derive  asymptotic completeness of (\ref{equation1}) from  (\ref{3equation1}).  For model (\ref{3equation1}), we first  construct an approximate solution $w_{ap}$, and then solve the perturbation problem around
$w_{ap}$.   To construct the approximate solution $w_{ap}$, it suffices to choose $w_{ap}$ to make
\begin{align*}
 i\partial_tw_{ap}+\Delta w_{ap}-c \overline{w_{ap}}(\partial_x w_{ap})^2-\nu_2\overline{w_{ap}}w_{ap}(\partial_x w_{ap})^2-\nu_3\overline{w_{ap}^2}(\partial_x w_{ap})^2
\end{align*}
decay  faster than $t^{-\frac{3}{2}}$ in $L^2_x\cap L^{\infty}_x$. We remark that the idea of constructing approximate solutions dates back to Ozawa \cite{o}, and generalized by many authors to other models, e.g. \cite{mtt,shi}.
The key difficulty to solve the perturbation problem around $w_{ap}$  is again   derivative loss. In fact, because of  the derivative loss in $(\partial_xw)^2$, the perturbation problem around $w_{ap}$ can not be solved by standard methods such as fixed point argument or bootstrap argument combined with linear estimates of Schr\"odinger equations (e.g. smoothing estimates, maximal function estimates).  The ingredients to  solve this problem are the following:\\
 (i) Construct a series of solutions to SMF which almost converge to the desired solution in some naive sense;\\
  (ii) Use  geometric Sobolev norms  obtained in the Cauchy problem  to compensate derivative loss in  $O(\partial^j_xw_{ap} \partial^k_x (w-w_{ap} ) \partial^{l}_x(w-w_{ap}) )$, $j,k,l\in \{0,1\}$;\\
   (iii) Construct a new functional with well chosen  weight to compensate derivative loss in  $O((\partial^k_xw_{ap})^2\partial_x (w-w_{ap}) )$, $k\in \{0,1\}$.

  In fact, the most serious derivative loss  term is the type in (iii), for which  we find a  weight $W(\cdot)$ and define the functional
\begin{align*}
\mathcal{W}(t):=\int_{\Bbb R} W(w_{ap},\overline{w_{ap}})|w-w_{ap}|^2dx,
\end{align*}
to simultaneously satisfy that $\mathcal{W}$ is  comparable to $\|w-w_{ap}\|^2_{L^2_x}$ and  $\frac{d}{dt} \mathcal{W}(t)$ enjoys closed energy estimate. If  $\mathcal{W}$ is found,   troublesome terms of  (iii) cancel  with each other, and Theorem 1.3 follows by the above strategy.

{\it Organization.} The paper is organized as follows. In Section 2, we bound $\|Sz\|_{H^k_x}$ and $\|\partial_xz\|_{H^{k}_x}$. In Section 3, we introduce the holomorphic transformation and prove some important algebraic facts. In Section 4, we prove the global bound of mass. In Section 5, we prove several  geometric facts on intrinsic vanishing points. In Section 6, we bound $\|Lz\|_{L^2_x}$. In Section 7, we set up the bootstrap assumption and recall some preliminaries on decay estimates of free Schr\"odinger equations. In Section 8, we derive bounds in Fourier space.
In Section 9, we close bootstrap, prove decay estimate and modified scattering/scattering. In Section 10, we prove Theorem 1.2. In Section 11, we verify Theorem \ref{XS4}.

{\bf{Notations.}} The notation $A\lesssim B$ means there exists some $C>0$ such that $A\le CB$.

Let $\varrho$ be a function supported in the annulus $\{\xi\in \Bbb R: \frac{3}{4}\le |\xi|\le \frac{8}{3}\}$ such that
$$\sum_{k\in \Bbb Z} \varrho(\frac{\xi}{2^{k}})=1, \mbox{ }\forall \xi\neq 0.$$
Define the Fourier multipliers
$$P_{k}=\varrho(\frac{D}{2^{k}}),\mbox{ } P_{<k}=\sum_{j<k}\varrho(\frac{D}{2^{k}}), \mbox{ }P_{\ge k}=I-P_{<k}.$$

Let $(\mathcal{N},h,J)$ denote the target K\"ahler manifold.
The connections of $T\mathcal{N}$  and  $u^{*}T\mathcal{N}$ are all denoted by ${\nabla}$ for simplicity.  Let $\bf{R}$ denote the curvature tensor of $\mathcal{N}$. We make the convention that
\begin{align*}
{\bf R}(X,Y)Z=\nabla_{X}\nabla_{Y}Z-\nabla_{Y}\nabla_{X}Z-\nabla_{[X,Y]}Z,
\end{align*}
and
\begin{align*}
{\bf R}(X,Y,Z,W)=h( {\bf R}(X,Y)Z,W).
\end{align*}

In the following, we fix three small constants $\epsilon_*,\varepsilon,\epsilon$ to fulfill
\begin{align*}
0<\epsilon_*\ll \varepsilon\ll \epsilon\ll 1.
\end{align*}

\section{Estimates of Sobolev norms and  $Su$}

In this section, we prove the slow growth of intrinsic Sobolev norms to $u$ and $Su$.
Throughout  Section 2,  we essentially only assume $\mathcal{N}$ is a complete K\"ahler manifold with bounded geometry.

\subsection{Preliminaries on Cauchy problem}

We recall the  global regularity theorem of 1D SMF due to Rodnianski-Rubinstein-Staffilani \cite{rrs}.
\begin{Theorem}[\cite{rrs}]\label{rrs}
Let $\mathcal{N}$ be a complete K\"ahler manifold with bounded geometry. Given $l\ge 2$, assume that $u_0\in W^{l,2}(\Bbb R;\mathcal{N})$. Then there exists a unique global solution $u\in C(\Bbb R; W^{l,2}(\Bbb R;\mathcal{N}))$ to (\ref{hia3}).
\end{Theorem}

Recall also that the energy defined by
\begin{align*}
E(u(t))=\frac{1}{2}\int_{\Bbb R}|\partial_x u(t)|^2dx
\end{align*}
conserves along the SMF.

Let $z$ be any given local complex coordinate of $\mathcal{N}$ near $Q$ with $z(Q)=0$.
Assume that
\begin{align}\label{zzsuption}
\|z_0\|_{L^2_x} +\|  u_0\|_{ W^{3,2}(\Bbb R;\mathcal{N}) } \ll 1.
\end{align}
From Theorem \ref{rrs}, initial data $u_0$ satisfying (\ref{zzsuption}) evolves to a global solution of SMF. Moreover, the proof of \cite{rrs} and assumption (\ref{zzsuption}) imply
\begin{align}\label{90o}
\| u\|_{C([0,1];W^{3,2}(\Bbb R;\mathcal{N}))}\lesssim \|u_0\|_{W^{3,2}(\Bbb R;\mathcal{N})}\ll 1.
\end{align}
Let $\omega\in (0,1)$ be the radius such that $\{z\in\Bbb C: |z|\le \omega\}$ lies in the local coordinate chart of $\mathcal{N}$ near $Q$ which corresponds to $z=0$. Assume that $t_*\in [0,1]$ be largest time such that
\begin{align*}
\sup_{t\in [0,t_*]}\|z(t)\|_{L^{\infty}_x}\le \omega.
\end{align*}
Then changing the intrinsic bounds given by  (\ref{90o}) to the extrinsic function $z(t,x)$ yields
\begin{align}
\sup_{t\in [0,t_*]}\|\partial_x z(t)\|_{L^{2}_x}&\lesssim \|\partial_xu_0\|_{L^2_x}  \nonumber\\
\sup_{t\in [0,t_*]}\|\partial_xz\|_{H^2_x}&\lesssim \|u_0\|_{W^{3,2}(\Bbb R;\mathcal{N})}. \label{k009}
\end{align}
Moreover, by the SMF equation, one has
\begin{align*}
\|z(t)\|^2_{L^2_x}\lesssim \|z_0\|^2_{L^2_x}+\int^{t}_0\|\partial_xz\|^2_{L^{\infty}_x}\|z\|^2_{L^2_x}ds,
\end{align*}
which  with (\ref{k009}) further shows
\begin{align*}
\sup_{t\in [0,t_*]}\|z\|_{L^2_x} \lesssim \|z_0\|_{H^3_x}.
\end{align*}
Thus Sobolev embedding yields
\begin{align*}
\|z\|_{L^{\infty}_{t,x}([0,t_*]\times \Bbb R)}\lesssim \|z\|_{L^{\infty}_t[0,t_*]H^1_x}\lesssim  \|z_0\|_{H^3_x}\ll 1.
\end{align*}
Therefore, $t_*=1$, i.e. $u([0,1]\times \Bbb R)$ lies in the local chart near $z=0$.

Define $\tilde{T}\in [0,\infty]$ to be the largest time  such that
\begin{align*}
\sup_{t\in [0,\tilde {T}]}\langle t\rangle^{\frac{1}{2}}\|z\|_{W^{2,\infty}_x}\le  \varepsilon.
\end{align*}
Then by Sobloev embedding and (\ref{k009}),  $\tilde {T}\ge 1$.

\subsection{Control of Sobolev norms}

\begin{Lemma}\label{energy}
Let $u$  be a sufficient regular solution to 1D SMF. Then for each $k\in\Bbb N$ one has
\begin{align}
\frac{d}{dt}\int_{\Bbb R}|\nabla^k_x\partial_xu|^2dx\le  C_{k}( \|\partial_xu\|^{2}_{L^{\infty}_x}+\|\partial_xu\|^{k+3}_{L^{\infty}_x})[\int_{\Bbb R}|\nabla^k_x \partial_xu|^2dx+
(\int_{\Bbb R}|\nabla^k_x \partial_xu|^2dx)^{\frac{k}{2k-1}} ]
\end{align}
\end{Lemma}
\begin{proof}
We compute
\begin{align*}
\frac{1}{2}\frac{d}{dt}\int_{\Bbb R}|\nabla^k_x\partial_xu|^2dx= \int_{\Bbb R}\langle\nabla^k_x\partial_xu,\nabla_t\nabla^k_x\partial_xu\rangle dx.
\end{align*}
By the commutating inequality
\begin{align*}
|\nabla_t\nabla^k_x\partial_xu-\nabla^k_x\partial_tu|\lesssim \sum_{j\ge 2}\sum_{l_0+(l_1+1)+...+(l_j+1)=k+1,l_i\ge 0}|\nabla^{l_0}_x\partial_t u||\nabla^{l_1}_x\partial_x u|...|\nabla^{l_j}_x\partial_xu|,
\end{align*}
and the equation $\partial_tu=J\nabla_x\partial_xu$, we get
\begin{align}
& \int_{\Bbb R}\langle\nabla^k_x\partial_xu,\nabla_t\nabla^k_x\partial_xu\rangle dx\nonumber\\
&\lesssim \int_{\Bbb R}\langle\nabla^k_x\partial_xu,J\nabla^{k+1}_x\nabla_x\partial_xu\rangle dx\label{n1}\\
&+\|\nabla^k_x\partial_xu\|_{L^2_x}\sum_{  2\le j\le k+1}\sum_{l_0+(l_1+1)+...+(l_j+1)=k+1,l_i\ge 0}\||\nabla^{l_0}_x\partial_t u||\nabla^{l_1}_x\partial_x u|...|\nabla^{l_j}_x\partial_x u|\|_{L^2_x}.\label{n3}
\end{align}
By integration by parts and  $\langle JX,X\rangle=0$, (\ref{n1}) vanishes. Let $p_0,p_1,...,p_j\in[1,\infty]$ and $\theta_0,...,\theta_j\in[0,1]$ be
\begin{align*}
&\frac{1}{p_0}+...+\frac{1}{p_j}=\frac{1}{2}\\
&\theta_{0}=\frac{k+\frac{1}{2}+\frac{1}{p_0}-l_0-2}{k-\frac{1}{2}}\\
&\theta_{b}=\frac{k+\frac{1}{2}+\frac{1}{p_b}-l_b-1}{k-\frac{1}{2}}, \mbox{ }b=1,...,j.
\end{align*}
Then H\"older inequality and Gagliardo-Nirenberg inequality show (\ref{n3}) is dominated by
\begin{align*}
&\|\nabla^k_x\partial_xu\|_{L^2_x}\sum_{ 2\le j\le k+1}\sum_{l_0+(l_1+1)+...+(l_j+1)=k+1,l_i\ge 0}\|\nabla^{l_0+1}_x\partial_x u\|_{L^{p_0}_x}\|\nabla^{l_1}_x\partial_x u\|_{L^{p_1}_x}...\|\nabla^{l_j}_x\partial_x u\|_{L^{p_j}_x}\\
&\lesssim \|\nabla^k_x\partial_xu\|_{L^2_x}\sum_{2\le j\le k+1}\sum_{l_0+(l_1+1)+...+(l_j+1)=k+1,l_i\ge 0}\|\partial_x u\|^{\sum^{j}_{b=0}\theta_b}_{L^{\infty}_x} \|\nabla^{k}_x\partial_x u\|^{\sum^{j}_{b=0}1-\theta_{b}}_{L^{2}_x}.
\end{align*}
It is direct to check
\begin{align*}
{\sum^{j}_{b=0}\theta_b}&=j+\frac{j-2}{k-\frac{1}{2}}\in [2,k+3]\\
{\sum^{j}_{b=0}(1-\theta_b)}&=1-\frac{j-2}{k-\frac{1}{2}}\in  [\frac{1}{2k-1},1].
\end{align*}
Putting all these together gives our lemma.
\end{proof}

For convenience, let
\begin{align*}
E_{k}(u)=\int_{\Bbb R}|\nabla^k_x\partial_xu|^2dx.
\end{align*}
Using Lemma \ref{energy} and Gronwall inequality, we have
\begin{Corollary}\label{067}
Let $u$   be a sufficient regular solution to 1D SMF satisfying
\begin{align*}
\sup_{0\le t\le T}\langle t\rangle^{\frac{1}{2}}\|\partial_xu(t)\|_{L^{\infty}_x}\le \epsilon.
\end{align*}
Then for each given $k\in\Bbb  N$,
\begin{align*}
\sup_{t\in [0,T]}\langle t\rangle^{-2\epsilon}  E_{k}(u(t)) \le E_{k}(u(0)).
\end{align*}
\end{Corollary}

\subsection{Control of $\|Su\|_{H^{k}_x}$}

Recall that the vector field corresponding to scaling is $S:=2t\partial_t+x\partial_x$.

Let's begin with the evolution of $\|Su\|_{L^2_x}$ along the SMF.

\begin{Lemma}\label{j1}
If $u$ solves 1D SMF, then
\begin{align}\label{p1}
\frac{d}{dt}\|Su\|^2_{L^2_x}\lesssim \|\partial_xu\|^2_{L^{\infty}_x}\|Su\|^2_{L^2_x}.
\end{align}
\end{Lemma}
\begin{proof}
In the following, denote
$$\langle X,Y\rangle=\int_{\Bbb R} h(X,Y)dx,$$
for $X, Y\in u^*T{\mathcal N}$.
Compute
\begin{align}\label{Eq1}
\frac{1}{2}\frac{d}{dt}\|Su\|^2_{L^2_x}=\langle 2\partial_t u, Su\rangle+\langle 2t\nabla_t\partial_t u, Su\rangle+\langle x\nabla_t\partial_xu,Su\rangle.
\end{align}
By the SMF equation and $\nabla J=0$, the second term of the RHS becomes
\begin{align}
&\langle 2t\nabla_t\partial_t u, Su\rangle=2t\langle J\nabla_t\nabla_x\partial_xu,Su\rangle\nonumber\\
&= \langle J{\bf R}(2t\partial_t u,\partial_x u)\partial_x u,Su\rangle+2t\langle J \nabla_x\nabla_x\partial_t u,Su\rangle\nonumber\\
&= \langle J{\bf R}(2t\partial_t u+x\partial_x u,\partial_x u)\partial_x u,Su\rangle-\langle J \nabla_x(2t\partial_t u),\nabla_xSu\rangle\label{p7}
\end{align}
where in the second line we used $[\nabla_t,\nabla_x]={\bf R}(\partial_tu,\partial_xu)$, and in the third line we applied ${\bf R}(\partial_x u,\partial_xu)=0$,  and integration by parts.
Again by integration by parts, $\nabla J=0$, and the SMF equation,
the third term on the RHS of (\ref{Eq1}) reduces to
\begin{align}
 \langle x\nabla_x\partial_t u, Su\rangle&=- \langle \partial_t u,Su\rangle- \langle x \partial_t u,\nabla_xSu\rangle\nonumber\\
&= - \langle \partial_t u,Su\rangle- \langle x J \nabla_x  \partial_x u,\nabla_xSu\rangle\nonumber\\
&= - \langle \partial_t u,Su\rangle+ \langle J   \partial_x u,\nabla_xSu\rangle-\langle J  \nabla_x(x \partial_x u),\nabla_xSu\rangle \nonumber\\
&= - \langle \partial_t u,Su\rangle-\langle J \nabla_x  \partial_x u,Su\rangle-\langle J  \nabla_x(x \partial_x u),\nabla_xSu\rangle\nonumber\\
&= - 2\langle \partial_t u,Su\rangle-\langle J  \nabla_x(x \partial_x u),\nabla_xSu\rangle.\label{p8}
\end{align}
Therefore, (\ref{Eq1}), (\ref{p7}), (\ref{p8}) together give
 \begin{align}
 \frac{1}{2}\frac{d}{dt}\|Su\|^2_{L^2_x}&= \langle J{\bf R}(2t\partial_t u+ x\partial_x u,\partial_x u)\partial_x u,Su\rangle-\langle J  \nabla_x(x \partial_x u+2t\partial_t u),\nabla_xSu\rangle\nonumber\\
&= \langle J{\bf R}(Su,\partial_x u)\partial_x u,Su\rangle-\langle J  \nabla_xSu,\nabla_xSu\rangle.\label{p5}
\end{align}
Since $\langle JX,X\rangle=0$ for any $X\in u^*T\mathcal{N}$,   the second term of the RHS of (\ref{p5}) vanishes. So we conclude
 \begin{align*}
&\frac{1}{2}\frac{d}{dt}\|Su\|^2_{L^2_x}=\langle J{\bf R}(Su,\partial_x u)\partial_x u,Su\rangle,
\end{align*}
from which (\ref{p1}) follows.
\end{proof}

\begin{Lemma}\label{a1}
If $u$ solves 1D SMF, then
\begin{align}
\frac{d}{dt}\|\nabla_x Su\|^2_{L^2_x}&\lesssim \|\partial_xu\|_{L^{\infty}_x}\|\nabla_x\partial_x u\|_{L^{\infty}_x}\|\nabla Su\|_{L^2_x}\|Su\|_{L^2_x}+\|\partial_xu\|^2_{L^{\infty}_x}\|\nabla Su\|^2_{L^2_x}\nonumber\\
&+\|\partial_xu\|^3_{L^{\infty}_x}\|\nabla Su\|_{L^2_x}\|Su\|_{L^2_x}.\label{p20}
\end{align}
\end{Lemma}
\begin{proof}
The computation is similar to that of Lemma \ref{j1}. For sake of completeness, we give a detailed proof. As before, in the following, denote
$$\langle X,Y\rangle=\int_{\Bbb R} h(X,Y)dx,$$
for $X, Y\in u^*T\mathcal N$.
Then one has
 \begin{align}
\frac{1}{2}\frac{d}{dt}\|\nabla_x Su\|^2_{L^2_x}&=\langle \nabla_t\nabla_xSu, \nabla_xSu\rangle \nonumber\\
&=\langle {\bf R}(\partial_tu,\partial_xu)Su, \nabla_xSu\rangle +\langle \nabla_x\nabla_tSu, \nabla_xSu\rangle.\label{Eq2}
\end{align}
Since
$$\nabla_t Su=2\partial_tu +2t\nabla_t\partial_tu+x\nabla_x \partial_tu,$$
using integration by parts yields
\begin{align}\label{j2}
\langle \nabla_x\nabla_tSu, \nabla_xSu\rangle=2\langle \nabla_x\partial_t u, \nabla_xSu\rangle+2t\langle \nabla_x\nabla_t\partial_tu,\nabla_xSu\rangle-\langle x\nabla_x\partial_tu,\nabla^2_x Su\rangle.
\end{align}
By the SMF equation and $\nabla J=0$, the second term of the RHS of (\ref{j2}) becomes
\begin{align}
2t\langle \nabla_x\nabla_t\partial_tu,\nabla_xSu\rangle&=2t\langle J\nabla_x\nabla_t\nabla_x\partial_xu,\nabla_xSu\rangle\nonumber\\
&=-2t\langle J\nabla_t \nabla_x\partial_xu,\nabla^2_x Su\rangle\nonumber\\
&=- \langle J{\bf R}(2t\partial_tu,\partial_xu)\partial_xu,\nabla^2_xSu\rangle-2t\langle J\nabla_x\nabla_t\partial_xu,\nabla^2_xSu\rangle\nonumber\\
&=- \langle J{\bf R}(Su,\partial_xu)\partial_xu,\nabla^2_xSu\rangle- \langle J\nabla^2_x(2t\partial_tu),\nabla^2_xSu\rangle,\label{j3}
\end{align}
where we  used ${\bf R}(X,X)=0$ again, and  applied $\nabla_x\partial_tu=\nabla_t\partial_xu$.
Meanwhile, by the SMF equation and $\nabla J=0$, the third term of the RHS of (\ref{j2})  can be rewritten as
\begin{align}
-\langle x\nabla_x\partial_t u,\nabla^2_xSu\rangle&= -\langle  \nabla_x(x\partial_tu),\nabla^2_xSu\rangle+\langle  \partial_t u,\nabla^2_x Su\rangle\nonumber\\
&=- \langle J\nabla_x(\nabla_x(x\partial_xu)),\nabla^2_xSu\rangle+\langle \partial_tu,\nabla^2_xSu\rangle+\langle J \nabla_x\partial_xu,\nabla^2_xSu\rangle\nonumber\\
&=- \langle J\nabla^2_x(x\partial_xu),\nabla^2_xSu\rangle+2\langle \partial_tu,\nabla^2_xSu\rangle.\label{j4}
\end{align}
Combining   (\ref{j2}), (\ref{j3}), (\ref{j4}) and (\ref{Eq2}) gives
 \begin{align*}
\frac{1}{2}\frac{d}{dt}\|\nabla_x Su\|^2_{L^2_x}&=- \langle J{\bf R}(Su,\partial_xu)\partial_xu,\nabla^2_xSu\rangle+ \langle  {\bf R}(\partial_tu,\partial_xu)Su,\nabla_xSu\rangle\\
&- \langle J\nabla^2_x(Su),\nabla^2_xSu\rangle
+2\langle \partial_tu,\nabla^2_xSu\rangle+2\langle \nabla_x\partial_t u, \nabla_xSu\rangle.
\end{align*}
The second line of the RHS vanishes by integration by parts and $\langle JX,X\rangle=0$. Hence one has
 \begin{align*}
\frac{1}{2}\frac{d}{dt}\|\nabla_x Su\|^2_{L^2_x}=- \langle J{\bf R}(Su,\partial_xu)\partial_xu,\nabla^2_xSu\rangle+ \langle  {\bf R}(\partial_tu,\partial_xu)Su,\nabla_xSu\rangle.
\end{align*}
For the first term on the RHS, integration by parts shows
 \begin{align*}
- \langle J{\bf R}(Su,\partial_xu)\partial_xu,\nabla^2_xSu\rangle&= \langle J{\bf R}(\nabla_xSu,\partial_xu)\partial_xu,\nabla_xSu\rangle+\langle J{\bf R}(Su,\nabla_x\partial_xu)\partial_xu,\nabla_xSu\rangle\\
&+\langle J{\bf R}( Su,\partial_xu)\nabla_x\partial_xu,\nabla_xSu\rangle-(\nabla{\bf R})(Su,\partial_xu,\partial_xu,J\nabla_xSu; \partial_xu).
\end{align*}
Then (\ref{p20}) follows.
\end{proof}

A longer but essentially the same proof of Lemma \ref{a1} gives  the corresponding result of $\nabla^2_x Su$ as follows.
\begin{Lemma}\label{p2}
If $u$ solves 1D SMF, then we have
\begin{align*}
&\frac{d}{dt}\|\nabla^2_x Su\|^2_{L^2_x}\\
&\lesssim (\|\partial_xu\|_{L^{\infty}_x}\|\nabla_x\partial_x u\|_{L^{\infty}_x}+\|\partial_xu\|^3_{L^{\infty}_x})\|\nabla^2_x Su\|_{L^2_x}\|\nabla_x Su\|_{L^2_x}+\|\partial_xu\|^2_{L^{\infty}_x}\|\nabla^2_x Su\|^2_{L^2_x}\nonumber\\
&+(\|\partial_xu\|^2_{L^{\infty}_x}\|\nabla_x\partial_x u\|_{L^{\infty}_x}+\|\partial_xu\|^4_{L^{\infty}_x}+\|\nabla^2_x\partial_xu\|_{L^{\infty}_x}\|\partial_xu\|_{L^{\infty}_x}+\|\nabla_x\partial_x u\|^2_{L^{\infty}_x})\|\nabla^2_x Su\|_{L^2_x}\| Su\|_{L^2_x}
\end{align*}
and
\begin{align*}
&\frac{d}{dt}\|\nabla^3_x Su\|^2_{L^2_x}\lesssim \|\partial_xu\|^2_{L^{\infty}_x}\|\nabla^3_x Su\|^2_{L^2_x}\\
&+(\|\partial_xu\|_{L^{\infty}_x}\|\nabla_x\partial_x u\|_{L^{\infty}_x}+\|\partial_xu\|^3_{L^{\infty}_x}) \|\nabla^3_x Su\|_{L^2_x}\| \nabla^2_xSu\|_{L^2_x}\\
&+(\|\partial_xu\|_{L^{\infty}_x}\|\nabla^2_x\partial_x u\|_{L^{\infty}_x}+\|\nabla_x\partial_xu\|^2_{L^{\infty}_x}
 +\|\partial_xu\|^4_{L^{\infty}_x})\|\nabla^3_x Su\|_{L^2_x}\|\nabla_x Su\|_{L^2_x}\\
&+(\|\partial_xu\|^3_{L^{\infty}_x}\|\nabla_x\partial_x u\|_{L^{\infty}_x}+\|\nabla_x\partial_x u\|^2_{L^{\infty}_x}\|\partial_xu\|_{L^{\infty}_x} +\|\partial_xu\|^5_{L^{\infty}_x})\|\nabla^3_x Su\|_{L^2_x}\| Su\|_{L^2_x}\\
&+(\|\nabla^3_x\partial_xu\|_{L^{\infty}_x}\|\nabla_x\partial_x u\|_{L^{\infty}_x}+\|\nabla^3_x\partial_xu\|_{L^{\infty}_x}\|\partial_x u\|^2_{L^{\infty}_x}+\|\nabla^3_x\partial_xu\|_{L^{\infty}_x}\|\partial_xu\|_{L^{\infty}_x})\|\nabla^3_xSu\|_{L^2_x}\| Su\|_{L^2_x}.
\end{align*}
\end{Lemma}

Then by Lemma \ref{j1}, Lemma \ref{a1}, Lemma \ref{p2} and Gronwall inequality, we have
\begin{Corollary}\label{Bnnhp}
If $u$ solves 1D SMF with  sufficient regular initial data, and assume that for all $t\in [0,T]$
\begin{align}\label{78}
\|\nabla_x\partial_xu\|_{L^{\infty}_x}+\|\partial_x u\|_{L^{\infty}_x}\le \epsilon\langle t\rangle^{-\frac{1}{2}},
\end{align}
then for $n_*$ sufficiently large there exists a sufficiently small constant $\delta$ depending only on $n_*,\epsilon$ such that
\begin{align}
\sum_{l=0,1,2,3}\sup_{t\in [0,T]} (1+t)^{-\delta}\|\nabla^l_x Su(t)\|_{L^2_x} \lesssim  \sum^{4}_{j=1}\|\langle x\rangle \nabla^{j}_xu_0\|_{L^2_x})\label{79}.
\end{align}
\end{Corollary}
\begin{proof}
By Gagliardo-Nirenberg inequality and Kato inequality, one has
\begin{align*}
  \|\nabla^2_x \partial_xu\|_{L^{\infty}_x}&\lesssim \|\nabla_x\partial_xu\|^{\theta }_{L^{\infty}_x}\|\nabla^{n_*-1}_x\partial_xu\|^{1- \theta}_{L^2_x} \\
   \|\nabla^3_x \partial_xu\|_{L^{\infty}_x}&\lesssim \|\nabla_x\partial_xu\|^{\gamma}_{L^{\infty}_x}\|\nabla^{n_*-1}_x\partial_xu\|^{1-\gamma}_{L^2_x},
\end{align*}
where $\theta=(n_*-3.5)/(n_*-2.5)$, and  $\gamma=(n_*-4.5)/(n_*-2.5)$. Using (\ref{78}) and Lemma \ref{j1}, we see if $n_*$ is large,
then $\|\nabla^2_x \partial_xu\|_{L^{\infty}_x}$ and  $\|\nabla^3_x \partial_xu\|_{L^{\infty}_x}$  decay almost as $t^{-\frac{1}{2}}$. Then (\ref{79})  follows by Lemma \ref{a1}, Lemma \ref{p2} and Gronwall inequality.
\end{proof}

\section{Holomorphic transformation and Some algebraic facts}
In this section, the assumption (\ref{KEY}) is not needed. The results hold for general Riemannian targets.

The holomorphic transformation
$$
w:=z+\gamma_1z^2+\gamma_2z^3+\gamma_3z^4$$
plays a vital role in the whole proof. We need to  carefully calculate the corresponding  coefficients to this transform.

Expanding (\ref{equation1}) gives
\begin{align}
(i\partial_t z+\Delta z)&=c_0(\partial_x z)^2+c_1\bar{z}(\partial_x z)^2+c_2z(\partial_x z)^2+c_3z^2(\partial_x z)^2+c_4\bar{z}^2(\partial_x z)^2\nonumber\\
&+c_5 \bar{z}z(\partial_x z)^2+O(|z|^3)(\partial_x z)^2.\label{xG78}
\end{align}
Then one has
\begin{align*}
&(i\partial_t w+\Delta w) = i\partial_tz+2i\gamma_1z\partial_tz+3i\gamma_2z^2\partial_tz+\Delta z+\gamma_1\Delta z^2+\gamma_2 \Delta z^3\\
&=(i\partial_tz+\Delta z)+2\gamma_1zi\partial_t z+2\gamma_1z\Delta z+2\gamma_1(\partial_x z)^2\\
&+3\gamma_2z^2i\partial_t z+3\gamma_2z^2\Delta z+6\gamma_2z(\partial_x z)^2\\
&+4\gamma_3z^2 i\partial_t z+4\gamma_3z^2\Delta z+12\gamma_3z^2(\partial_x z)^2\\
&=(i\partial_tz+\Delta z)+2\gamma_1z(i\partial_tz+\Delta z)+2\gamma_1(\partial_x z)^2+3\gamma_2(i\partial_tz+\Delta z)z^2\\
&+4\gamma_3(i\partial_tz+\Delta z)z^2+6\gamma_2z(\partial_x z)^2+12\gamma_3z^2(\partial_x z)^2.
\end{align*}
Therefore,
\begin{align*}
 (i\partial_t w+\Delta w) &= (c_0 +c_1\bar{z} +c_2z +c_3z^2 +c_4\bar{z}^2 +c_5|z|^2)(\partial_x z)^2\\
&+2\gamma_1z(c_0+c_1\bar{z}+c_2z)(\partial_x z)^2+2\gamma_1(\partial_x z)^2+3c_0\gamma_2z^2(\partial_x z)^2\\
&+6\gamma_2z(\partial_x z)^2+12\gamma_3z^2(\partial_x z)^2+O(|z|^3|\partial_x z|^2).
\end{align*}
In order to cancel quadratic terms and the $z(\partial_x z)^2$ term, we take
\begin{align}\label{YY}
\left\{
  \begin{array}{ll}
    2\gamma_1+c_0=0, & \hbox{ } \\
    c_2+2\gamma_1c_0+6\gamma_2=0, & \hbox{ } \\
    12\gamma_3+c_3+2\gamma_1c_2+3\gamma_2c_0=0. & \hbox{ }
  \end{array}
\right.
\end{align}
Now, for the chosen $\gamma_1,\gamma_2,\gamma_3$, we conclude
\begin{align*}
 (i\partial_t w+\Delta w) &=  c_1\bar{z}(\partial_x z)^2+(c_5+2\gamma_1c_1)|z|^2(\partial_x z)^2+c_4\bar{z}^2(\partial_x z)^2+O(|z|^3|\partial_x z|^2).
\end{align*}
Since $w=z+\gamma_1z^2+\gamma_2z^3+\gamma_3z^4$, at $w=0$ one gets
\begin{align*}
\frac{d z}{dw}=1,\mbox{  }\frac{d^2 z}{dw^2}=-2\gamma_1.
\end{align*}
Thus
\begin{align*}
z=w-\gamma_1w^2+O(w^3).
\end{align*}
Then we see
\begin{align*}
c_1\bar{z}(\partial_xz)^2=c_1\overline{w}(\partial_x w)^2-c_1\overline{\gamma_1}\overline{w}^2(\partial_x w)^2-4c_1\gamma_1|w|^2(\partial_x w)^2+O(|w|^3|\partial_x w|^2).
\end{align*}
So we have
\begin{align}
 (i\partial_t w+\Delta w) &=  c_1\overline{w}(\partial_x w)^2+(c_5-2\gamma_1c_1)|w|^2(\partial_x w)^2+(c_4-c_1\overline{\gamma_1})\overline{w}^2(\partial_x w)^2\nonumber\\
 &+O(|w|^3|\partial_x w|^2).\label{zgbn}
\end{align}

\begin{Lemma}\label{Bn}
Let $\nu_2=c_5-2\gamma_1c_1$, $\nu_3=c_4-c_1\overline{\gamma_1}$. Then
\begin{align*}
\nu_2=2\overline{\nu_3}.
\end{align*}
\end{Lemma}
\begin{proof}
Note that $c_4,c_5$ defined by (\ref{xG78}) are indeed given by
\begin{align}
c_5=[\ln h(z)]_{z\bar{z}z}(0,0);\mbox{  }c_4=\frac{1}{2}[\ln h(z)]_{z\bar{z}\bar{z}}(0,0).
\end{align}
Hence, there holds
\begin{align*}
c_5=2\overline{c_4}.
\end{align*}
And note that
\begin{align*}
c_1=[\ln h(z)]_{z\bar{z}}(0,0)=\frac{1}{4}(\partial^2_x+\partial^2_y) (\ln h )(0,0)
\end{align*}
is a real constant.
Using these two facts we see
\begin{align*}
c_5-2\gamma_1c_1=2\overline{c_4}-2\overline{c_1}\gamma_1,
\end{align*}
i.e. $\nu_2=2\overline{\nu_3}$.
\end{proof}

\begin{Lemma}\label{Bn2}
Let $\mathcal{N}$ be a Riemannian surface and $Q\in \mathcal{N}$. Assume that $z$ is a local complex coordinate of $\mathcal{N}$ near $Q$ with $z(Q)=0$, and the metric of $\mathcal{N}$ writes as $h(z,\bar{z})dzd\bar{z}$ near $Q$.  Then the constant $c_1$ in the RHS of
(\ref{zgbn}) satisfies
\begin{align*}
c_1=-\frac{1}{2}K(Q)h_0,
\end{align*}
where $K(Q)$ denotes the sectional curvature at $Q$, $h_0=h(0,0)$ denotes the metric at $Q$ under the coordinate $z$.
\end{Lemma}
\begin{proof}
The Laplace-Beltrami operator with respect to the metric $hdzd\bar{z}$ is given by
\begin{align*}
\Delta_{\mathcal{N}}=\frac{4}{h}\frac{\partial}{\partial z}\frac{\partial}{\partial\bar {z}}.
\end{align*}
The sectional curvature at $z$ is
\begin{align*}
K(z)=-\Delta_{\mathcal{N}}(\ln h^{\frac{1}{2}}).
\end{align*}
Note that
\begin{align*}
c_1=[\ln h(z)]_{z\bar{z}}(0,0).
\end{align*}
So one has
\begin{align*}
c_1=[\ln h(z)]_{z\bar{z}}(0,0)=-\frac{1}{2}K(Q)h_0.
\end{align*}
\end{proof}

\begin{remark}
Note that $h_0$ in Lemma \ref{Bn2} is the same under coordinate $z$ and coordinate $w=z+\gamma_1 z^2+\gamma_2 z^3+\gamma_3 z^4$, since $\frac{\partial {w}}{\partial z}(0)=1$.
\end{remark}

\section{Mass conservation}

In this section, the assumption (\ref{KEY}) is not needed. The results hold for general Riemannian targets.

 Recall  that  $w(t,x):=z+\gamma_1z^2+\gamma_2 z^3+\gamma_3z^4$. For simplicity, write (\ref{zgbn}) as
\begin{align}
 (i\partial_t w+\Delta w) &=  (\nu_1\overline{w}+\nu_2|w|^2+\nu_3\overline{w}^2)(\partial_x w)^2+O(|w|^3)(\partial_x w)^2\nonumber\\
 &:=[\nu_1 \bar{w}+g(w,\bar{w})](\partial_x w)^2\label{Gvbn}
\end{align}
Let
\begin{align}\label{HJ}
H(w,\overline{w})=1-\nu_1|w|^2
\end{align}
\begin{Lemma}
$H(w,\overline{w})$ defined by (\ref{HJ}) is real valued and satisfies
\begin{align}\label{Re}
H(w,\overline{w})( \nu_1 \bar{w}+g(w,\bar{w}))+H_{w}(w,\overline{w})=O(|w|^2).
\end{align}
\end{Lemma}
\begin{proof}
This is a direct calculation.
\end{proof}

For the real valued function $H:\Bbb C\to \Bbb R$  define in (\ref{HJ}), consider the functional
\begin{align}\label{g2}
\mathcal{H}(w)=\int_{\Bbb R}H(w,\overline{w})|w|^2dx.
\end{align}
We will prove for this well chosen $H$, $\mathcal{H}(w)$ behaves well along the SMF.

\begin{Lemma}\label{sK1}
Let $H(w,\overline{w})$ be defined as above.
Assume that $\|w\|_{L^{\infty}_{t,x}([0,T]\times \Bbb R)}\le \omega\ll 1$.  Then for all $t\in [0,T]$
\begin{align*}
\frac{d}{dt}\mathcal{H}(w)&\lesssim  \|w\|^3_{W^{1,\infty}_x}\mathcal{H}(w).
\end{align*}
\end{Lemma}
\begin{proof}
By definition of $\mathcal{H}(w)$,
\begin{align}\label{ccga5}
\frac{d}{dt}\mathcal{H}(w)=\int_{\Bbb R} (H_{w} \partial_tz+H_{\overline{w}}\partial_t \overline{w})|w|^2dx+\int_{\Bbb R}H(w,\overline{w})[ ( \partial_t w)\overline{w}+ w \partial_t\overline{ w}]dx.
\end{align}
(\ref{Gvbn})  implies
\begin{align}
 &H_{w}\partial_tw+H_{\overline{w}}\partial_t \overline{w}\nonumber\\
 &=iH_{w}\Delta w-iH_{\overline{w}}\Delta\overline{w}\label{gau}\\
 &-iH_{w}(\nu_1\overline{w}+g(w,\overline{w}))(\partial_xw)^2+ iH_{\overline{w}}(\nu_1 {w}+\bar{g}(w,\overline{w}))(\partial_x\overline{w})^2.\label{ga1}
\end{align}
and
\begin{align}
 &\overline{w}\partial_t w+  w \partial_t\overline{ w}\nonumber\\
 &=i\overline{ w}\Delta  w -i  w\Delta\overline{w}\label{ga2}\\
&-i({g}(w,\overline{w})+ \nu_1\bar{w})   (\partial_x   {  w}) ^2\overline{w}+i(\bar{g}(w,\overline{w})+ \nu_1w)  (\partial_x \overline{  w})^2 w.
  \label{ga}
\end{align}
The  contributions of the  two terms of (\ref{ga1}) are easy to bound, in fact, one has
\begin{align*}
&\int_{\Bbb R} |H_{w}(\nu_1\overline{w}+g(w,\overline{w}))(\partial_xw)^2|| w|^2 dx \lesssim \| w\|^3_{W^{1,\infty}_x}  \| w\|^2_{L^2_x}.
\end{align*}
For the first two terms of RHS of (\ref{gau}), integration by parts gives
\begin{align}
&\int_{\Bbb R} [iH_{w}  \Delta w-iH_{\overline{w}}\Delta\overline{w}  ]|w|^2 dx\nonumber\\
&=\int_{\Bbb R} (-iH_{w}  \partial_x{w}+iH_{\overline{w}} \partial_x \overline{w} )[(\partial_x w) \overline{ w}+ w\partial_x\overline{ w}] dx\nonumber\\
&-\int_{\Bbb R}[i(\partial_xH_{w})\overline{w}\partial_x{w} -i(\partial_xH_{\overline{w}})w\partial_x\overline{w}  ]| w|^2dx.\label{xxGh}
\end{align}
Denote the term (\ref{xxGh}) by ${\rm II}$. Then   ${\rm II}$ is dominated by
\begin{align*}
|{\rm II}|\lesssim \|w\|^3_{W^{1,\infty}_x}\| w\|^2_{L^2_x}.
\end{align*}
Again by integration by parts, the integral associated with the RHS of (\ref{ga2}) becomes
\begin{align*}
&\int_{\Bbb R} H(w,\overline{w})[i\overline{ w}\Delta  w -i  w \Delta \overline{ w}]dx\\
&=\int_{\Bbb R} (H_{w}\partial_x{w}+H_{\overline{w}}\partial_x \overline{w})[-i(\partial_x w) \overline{ w}+i w\partial_x\overline{ w}]dx.
\end{align*}
Thus the RHS of (\ref{ccga5}) equals
\begin{align*}
&\int_{\Bbb R}(H_{w} \partial_tw+H_{\overline{w}}\partial_t\overline{ w}) w\overline{ w}dx+\int_{\Bbb R}H(w,\overline{w})[ ( \partial_t w)\overline{w}+  w \partial_t\overline{ w}]dx\\
&= \int_{\Bbb R} (-iH_{w}  \partial_x{w}+iH_{\overline{w}} \partial_x \overline{w} )[(\partial_x w) \overline{ w}+ w\partial_x\overline{ w}] dx\\
&+\int_{\Bbb R} (H_{w}\partial_x{w}+H_{\overline{w}}\partial_x \overline{w})[-i(\partial_x w) \overline{ w}+i w\partial_x\overline{ w}]dx \\
&+\int_{\Bbb R} H(w,\overline{w})  [ i(\bar{g}(w,\overline{w})+ \nu_1w) (\partial_x \overline{  w})^2 w- i( g(w,\overline{w})+ \nu_1\overline{w})   \overline{ w} (\partial_x  w)^2]dx\\
&+\mathrm{I}+\mathrm{II},
\end{align*}
where $\mathrm{I}$ and $\mathrm{II}$ denote  the last   terms of (\ref{ga1}) and  (\ref{xxGh}) respectively.
Notice that (\ref{Re}) shows
\begin{align*}
H(w,\overline{w})(g(w,\overline{w})+\nu_1\overline{w})+H_w(w,\overline{w})=O(|w|^2).
\end{align*}
Then,   we arrive at
\begin{align*}
&\int_{\Bbb R}(H_{w} \partial_tw+H_{ \overline{w}}\partial_t \overline{w}) w\overline{ w}dx+\int_{\Bbb R}H(w,\overline{w})[ ( \partial_t w)\overline{ w}+   w \partial_t\overline{ w}]dx\\
&=\mathrm{I}+\mathrm{II}+\int_{\Bbb R}O(|w|^3)|\partial_x w|^2 dx.
\end{align*}
Therefore, one obtains
\begin{align}
 \frac{d}{dt}\mathcal{H}(w)
 &\lesssim   \| w\|^3_{W^{1,\infty}_x}  \| w\|^2_{L^2_x}.\label{ccfgh}
\end{align}
Since $\|w\|_{L^{\infty}([0,T]\times \Bbb R)}\le \omega$, one has
$$ \mathcal{H}(w) \sim \| w\|^2_{L^2_x}.
$$
Then the lemma follows by (\ref{ccfgh}).
\end{proof}

Corollary \ref{067} and
Corollary \ref{Bnnhp} are for the intrinsic quantities. We need to transfer  them to the extrinsic function $w(t,x)$. The bridge is the global bound of mass.
\begin{Corollary}\label{K3}
If $u$ solves SMF with initial data fulfilling  (\ref{suption}) and
\begin{align}\label{67}
\sup_{t\in [0,T]}(1+t)^{\frac{1}{2}}(\|\partial_x u\|_{L^{\infty}_x}+\|\nabla_x\partial_xu\|_{L^{\infty}_x})\le \epsilon.
\end{align}
Then for $n_*$ sufficiently large there exists a sufficiently small $\delta$ depending only on $n_*,\epsilon$ such that for all $t\in [0,T]$
\begin{align}
\sup_{t\in [0,T]}\langle t\rangle^{-\delta} \|Sw(t)\|_{H^2_x}&\lesssim   \|w_0\|_{H^{3,1}_x}\label{66o}\\
\sup_{t\in [0,T]} \langle t\rangle^{-\delta} \|w(t)\|_{H^{n_*}_x}&\lesssim   \|w_0\|_{H^{n_*}_x}.\label{P7}
\end{align}
\end{Corollary}
\begin{proof}
Set
$$0<\epsilon_*\ll \epsilon\ll \omega\ll 1.
$$
Let $T'_*\in [0,T]$ be the maximal time such that
\begin{align*}
\|w(t,x)\|_{L^{\infty}_{t,x}([0,T'_*]\times \Bbb R)}\le \omega.
\end{align*}
By (\ref{suption}) and the local Cauchy theorem, $T'_*>0$. By   Lemma \ref{sK1},  assumption (\ref{67}), and Gronwall inequality, one has
\begin{align*}
\|w\|_{L^{\infty}_tL^2_x([0,T'_*]\times \Bbb R)}\lesssim \|w_0\|_{L^2_x}\le \epsilon_*,
\end{align*}
which together with (\ref{67})  gives
\begin{align*}
\|w\|_{L^{\infty}_{t,x}([0,T'_*]\times \Bbb R)}\lesssim \|\partial_xw\|^{\frac{1}{3}}_{L^{\infty}_{t,x}}\|w\|^{\frac{2}{3}}_{L^{\infty}_tL^2_x}\lesssim \epsilon\ll \omega.
\end{align*}
So $T'_*=T$, i.e. $u([0,T]\times \Bbb R)$ lies in a local chart of $\mathcal{N}$ around $w=0$. And thus  all the Christoffel symbols and their k-th order derivatives with respect to $x$ are uniformly  bounded
in $(t,x)\in [0,T]\times \Bbb R$. Then writing $|\nabla_x\partial_xu|$ and $|\partial_x u|$ in the local coordinate $z$ shows that  (\ref{67}) implies (\ref{78}). Then
Corollary \ref{Bnnhp} yields bounds of  $\{\|\nabla^j_x Su(t)\|_{L^2_x}\}^{2}_{j=0}$. Again expressing  $|\nabla^j_xSu|$  in the local coordinate $w$ gives the desired result  (\ref{66o}). Similarly (\ref{P7}) follows by Corollary \ref{067} and the mass bound.
\end{proof}

We will also need the upper-bound of $\|\langle x\rangle w\|_{H^1_x}$. The proof is a simple energy argument.

\begin{Lemma} \label{NP}
Assume that
\begin{align}
\|w_0\|_{H^2_x}+\sup_{t\ge 0}\langle t\rangle^{\frac{1}{2}}\|   w\|_{W^{2,\infty}_x}&\lesssim \epsilon.\label{Hj}
\end{align}
Then we have
\begin{align}
\| x w\|_{L^2_x}&\lesssim \epsilon t\label{J1}\\
\| x\partial_x w\|_{L^2_x}&\lesssim \epsilon t^{1+\epsilon}.\label{J2}
\end{align}
\end{Lemma}
\begin{proof}
By (\ref{Hj}) and integration by parts,
\begin{align*}
\frac{d}{dt}\| x w\|^2_{L^2_x}\lesssim \|\partial_x w\|_{L^2_x}\|x w\|_{L^2_x}+ \|\partial_x w\|^2_{L^{\infty}_x}\|x w\|^2_{L^2_x}.
\end{align*}
Then (\ref{J1}) follows by Gronwall inequality, (\ref{Hj}) and energy conservation.

And we also have  by (\ref{Hj}) and integration by parts that
\begin{align*}
\frac{d}{d t}\| x\partial_x w\|^2_{L^2_x}\lesssim \|\partial^2_x w\|_{L^2_x}\|x\partial_x w\|_{L^2_x}+ \|\partial_x w\|^2_{L^{\infty}_x}\|x\partial_x w\|^2_{L^2_x}
+ \|\partial^2_x w\|_{L^{\infty}_x}\|w\|_{L^{\infty}_x}\|x\partial_x w\|^2_{L^2_x}.
\end{align*}
Thus (\ref{J2}) follows by Gronwall inequality, (\ref{Hj}) and the bound $\|\partial^2_x  w \|_{L^2_x}\lesssim  \epsilon t^{\epsilon}$.
\end{proof}

\section{Intrinsic Vanishing condition}

First, we prove Lemma \ref{ADs}, i.e. the assumption (\ref{KEY}) is invariant under holomorphic coordinate transformation. Hence, (\ref{KEY}) is an intrinsic geometric  assumption rather an analytic extrinsic assumption.

We first prove the following result.
\begin{Lemma}\label{HjVnn}
Suppose that $\mathcal{N}$ is a Riemannian surface and $Q$ is a given point in $\mathcal{N}$.
Let $\eta$ be a local coordinate of $\mathcal{N}$ near $Q$ with  $\eta(Q)=\eta_0$. The metric in coordinate $\eta$ writes as $h(\eta,\bar{\eta})d\eta d\bar{\eta}$.  Assume that $\eta=f(z)$ is a local holomorphic transformation which maps $z_0$ to $\eta_0$ and $f_{z}(z_0)\neq 0$. Let $\tilde{h}(z,\bar{z}) dz d\bar{z}$ be the metric under the $z$ coordinate.
Then near $z_0$ one has
\begin{align*}
\left([\ln \tilde{h}]_{z}[\ln\tilde{h}]_{z\bar{z}}-[\ln \tilde{h}]_{z\bar{z}z}\right)(z)= \overline{f_z}(f_z)^2\left([\ln h]_{\eta}[\ln h]_{\eta\bar{\eta}}-[\ln h]_{\eta\bar{\eta}\eta}\right)(f(z)).
\end{align*}
\end{Lemma}
\begin{proof}
The metric under under the $z$ coordinate shall be $h(f(z),\bar{f}(z))f_z\overline{f_z}dz d\bar{z}$. Hence,
\begin{align*}
  \tilde{h}=h(f(z),\bar{f}(z))f_z\overline{f_z}.
\end{align*}
By computation, we get
\begin{align*}
[\ln \tilde{h}]_{z}&=\frac{f_{zz}}{f_z}+\frac{h_{\eta} f_{z}}{h}\\
[\ln \tilde{h}]_{z\bar{z}}&=  \frac{h_{\eta\bar{\eta}} \overline{f_{z}}f_{z} h-h_{\bar{\eta}}\overline{f_{z}}h_{\eta}f_{z}}{h^2},
\end{align*}
and
\begin{align*}
&[\ln \tilde{h}]_{z\bar{z}z}\\
&= h^{-2}\left(   h_{\eta\bar{\eta}} \overline{f_{z}}(f_{z})^2 h+h_{\eta\bar{\eta}} \overline{f_{z}} f_{z} h_{\eta}f_{z}+ h_{\eta\bar{\eta}} \overline{f_{z}} f_{zz} h-h_{\bar{\eta}}h_{\eta}\overline{f_{z}} f_{zz}-h_{\eta\bar{\eta}}h_{\eta}\overline{f_{z}} (f_{z})^2-h_{\bar{\eta}\eta}h_{\bar{\eta}}\overline{f_{z}} (f_{z})^2\right)\\
&-2h^{-3}h_{\eta}f_{z}\left(h_{\eta\bar{\eta}}\overline{f_{z}}f_zh-h_{\eta}h_{\bar{\eta}}f_z\overline{f_z}\right).
\end{align*}
Therefore, we see
\begin{align*}
&[\ln \tilde{h}]_{z}[\ln\tilde{h}]_{z\bar{z}}-[\ln \tilde{h}]_{z\bar{z}z}\\
&= \overline{f_{z}}(f_{z})^2\left([\ln h]_{\eta}[\ln h]_{\eta\bar{\eta}}-[\ln h]_{\eta\bar{\eta}\eta}\right)
 +\frac{f_{zz}}{f_z}[\ln \tilde{h}]_{z\bar{z}}
-\frac{ h_{\eta\bar{\eta}} \overline{f_{z}} f_{zz} h-h_{\bar{\eta}}h_{\eta}\overline{f_{z}} f_{zz}}{h^2}\\
&=\overline{f_{z}}(f_{z})^2\left([\ln h]_{\eta}[\ln h]_{\eta\bar{\eta}}-[\ln h]_{\eta\bar{\eta}\eta}\right).
\end{align*}
\end{proof}

\underline{{\it Proof of Lemma \ref{ADs}.}} Lemma \ref{ADs} follows directly by Lemma \ref{HjVnn}.

The {Proposition} \ref{AYs} implies   the existence of intrinsic vanishing  point. Now, we are ready to give a proof. \\
 \underline{{\it Proof of Proposition  \ref{AYs}.}}
 Let $\{(U_i,\varphi_i);1\le i\le n\}$ be a coordinate chart  of $\mathcal{N}$ for which each  $\varphi_i:U_i\to  \Bbb C$ is a homeomorphism  from the open subset $U_i$ of $\mathcal{N}$ to its image in $\Bbb C$, so that $\cup_{i}U_i$ is a covering of $\mathcal{N}$, and for each $i,j\in \{1,...,n\}$ and $U_i\cap U_j \neq  \emptyset$, the map $\varphi_i\varphi^{-1}_j$ is a holomorphic map from $\varphi_j(U_i\cap U_j)$ to $\varphi_i(U_i\cap U_j)$.
In each given open set $U_i$, define the   1-form
$$
\frac{[\ln  {h}]_{z}[\ln {h}]_{z\bar{z}}-[\ln  {h}]_{z\bar{z}z}}{h(\frac{\partial }{\partial z}, \frac{\partial}{\partial \bar{z}})}dz
+\frac{[\ln  {h}]_{\bar{z}}[\ln {h}]_{z\bar{z}}-[\ln {h}]_{\bar{z}\bar{z}z}}{h(\frac{\partial }{\partial z}, \frac{\partial}{\partial \bar{z}})}d\bar{z}.
$$
{Lemma} \ref{HjVnn} implies that these (real valued) 1-form  coincide with each other in $U_i\cap U_j$. So, one gets a globally defined   1-form $\omega$ on $\mathcal{N}$. By duality, we have a globally defined real valued  vector field $X$ on $\mathcal{N}$. Since the Euler-Poincare characteristic of $\mathcal{N}$ is non-zero, Poincare-Hopf index theorem implies that there exists at least one point $Q\in \mathcal{N}$ such that $X(Q)=0$. And thus dy duality, $\omega(Q)=0$.
Let
$$
\frac{[\ln  {h}]_{z}[\ln {h}]_{z\bar{z}}-[\ln {h}]_{z\bar{z}z}}{h(\frac{\partial }{\partial z}, \frac{\partial}{\partial \bar{z}})}=f+ig,
$$
where $f$ and $g$ denote the real and respectively imaginary parts.
Denote also $z=x+iy$. Then the 1-form  $\omega$ writes as
$$
2(fdx-gdy).
$$
Thus we infer from $\omega(Q)=0$ that $f(Q)=g(Q)=0$, i.e.
\begin{align*}
 \left({[\ln {h}]_{z}[\ln {h}]_{z\bar{z}}-[\ln {h}]_{z\bar{z}z}}\right)(Q)=0.
\end{align*}
\mbox{ }\mbox{ }\mbox{ }\mbox{ }\mbox{ }$\square$

\begin{Corollary}\label{00}
Suppose that $\mathcal{N}$ is a Riemannian surface and $Q$ is a given point in $\mathcal{N}$.
Let $z$ be a local coordinate of $\mathcal{N}$ near $Q$ which corresponds to $z=0$. Assume further that (\ref{KEY}) holds. Then the cosnatnts $\nu_2,\nu_3$ in  (\ref{Gvbn}) fulfills
\begin{align*}
 \nu_2=0=\nu_3.
\end{align*}
\end{Corollary}
\begin{proof}
By {Lemma} \ref{Bn}, it suffices to prove
$$\nu_2=c_5-2\gamma_1c_1=0.$$
The first line of equation (\ref{YY}) shows
\begin{align*}
c_5-2\gamma_1c_1=c_5-c_0c_1.
\end{align*}
Note that
\begin{align*}
c_5-c_0c_1=[\ln h]_{z\bar{z}z}(0)-[\ln h]_{z\bar{z}}(0)[\ln h]_{z}(0).
\end{align*}
So  (\ref{KEY}) implies $c_5-c_0c_1=0$, and thus $\nu_2=0=\nu_3$.
\end{proof}

The following lemma collects some other examples verifying the assumption (\ref{KEY}).

\begin{Lemma}
Assume that $\mathcal{N}$ is a Riemannian surface with metric $h(z,\bar{z})dz d\bar{z}$.  \\
(i) If $\mathcal{N}$ is rotationally symmetric near $Q$, i.e. $h(z,\bar{z})=g(|z|^2)$ for some function $g$ near $Q$ with coordinate $z=0$, then $Q$ is an intrinsic vanishing point.\\
(ii) Each point of  $\Bbb S^2$ or $\Bbb H^2$ is an intrinsic vanishing point.
\end{Lemma}
\begin{proof}
(i)  follows by direct calculation.  (ii) follows by   the symmetry of $\Bbb S^2$, $\Bbb H^2$ and (i).
\end{proof}

\section{Estimates of $Lw$.}

In this section, we assume $Q$ is an intrinsic vanishing point.
Recall  the  operator $L:=ix-2t\partial_x$.  Recall  also that $w(t,x):=z+\gamma_1z^2+\gamma_2 z^3+\gamma_3z^4$, and it solves  (\ref{Gvbn}).
Under the assumption (\ref{KEY}), {Corollary} \ref{00} shows the constants $\nu_2,\nu_3$ in  (\ref{Gvbn}) fulfill
$$\nu_2=0=\nu_3.$$
Thus $w$ satisfies
\begin{align}\label{gre1}
i\partial_t w+\Delta w=( \nu_1\bar{w}+g(w,\bar{w}))(\partial_xw)^2,
\end{align}
where  $ g(w,\bar{w})$ is of order 3 or higher in $w$.

And the function $H(w,\bar{w})$  defined by (\ref{HJ}) now
fulfills
\begin{align}\label{Re1}
H(w,\bar{w})( \nu_1 \bar{w}+g(w,\bar{w}))+H_{w}(w,\bar{w})=O(|w|^3).
\end{align}

Applying $L$ to both sides of (\ref{gre1}) yields
\begin{align}
&i\partial_t Lw+\Delta Lw\nonumber\\
&=\nu_1\overline{Lw}\partial_xw\partial_xw+2\nu_1\bar{w}\partial_x(Lw)\partial_xw+(g_w Lw+g_{\bar{w}}\overline{Lw})(\partial_x w)^2\nonumber\\
&+2g(w,\bar{w}) \partial_xw\partial_x (L w)-2i(\nu_1\bar{w}+g(w,\bar{w}))w \partial_x w+O(|x||w|^3|\partial_xw|^2).\label{g1}
\end{align}

For the real valued function $H:\Bbb C\to \Bbb R$  define in (\ref{HJ}), consider the functional
\begin{align}\label{g2}
\mathcal{L}(w)=\int_{\Bbb R}H(w,\bar{w})Lw\overline{Lw}dx.
\end{align}
We will prove  $\mathcal{L}(w)$ behaves well along the SMF.

\begin{Lemma}\label{K1}
Assume that $Q$ is an intrinsic vanishing point. Let $w$ and $H(w,\bar{w})$ be defined as above.
Assume that $\|w\|_{L^{\infty}_{t,x}([0,T]\times \Bbb R)}\le \omega\ll 1$.  Then for all $t\in [0,T]$
\begin{align*}
\frac{d}{dt}\mathcal{L}(w)&\lesssim  \|w\|^2_{W^{1,\infty}_x}  \mathcal{L}(w)+\|xw\|_{L^2_x}\|w\|^4_{w^{1,\infty}_x} \sqrt{ \mathcal{L}(w)}\\
 &+\|w\|^4_{W^{1,\infty}_x} (t\|\partial^2_x w\|_{L^2_x}+\|\langle x\rangle  w\|_{H^1_x})\sqrt{ \mathcal{L}(w)}.
\end{align*}
\end{Lemma}
\begin{proof}
By definition of $\mathcal{L}(w)$,
\begin{align}\label{ga5}
\frac{d}{dt}\mathcal{L}(w)=\int_{\Bbb R} (H_{w} \partial_tz+H_{\bar{w}}\partial_t \bar{w})Lw\overline{Lw}dx+\int_{\Bbb R}H(w,\bar{w})[ ( \partial_tLw)\overline{Lw}+  Lw \partial_t\overline{Lw}]dx.
\end{align}
(\ref{gre1}) and (\ref{g1}) imply
\begin{align}
 &H_{w}\partial_tw+H_{\bar{w}}\partial_t \bar{w}\nonumber\\
 &=iH_{w}\Delta w-iH_{\bar{w}}\Delta\bar{w}-iH_{w}(\nu_1\bar{w}+g(w,\bar{w}))(\partial_xw)^2+ iH_{\bar{w}}(\nu_1 {w}+\bar{g}(w,\bar{w}))(\partial_x\bar{w})^2.\label{ga12}\\
 &\overline{Lw}\partial_tLw+(Lw)\partial_t\overline{Lw}\nonumber\\
 &=2(\nu_1\bar {w}+ g(w,\bar{w}))w \partial_x w\overline{Lw}+2 (\nu_1 {w}+ \bar{g}(w,\bar{w}))\bar{w} \partial_x \bar{w}Lw   \label{yu00}\\
 &+i\overline{L w}\Delta Lw -i Lw\Delta\overline{Lw}\label{ga22}\\
&-2i( g(w,\bar{w})+ \nu_1\bar{w})   \partial_x w\overline{Lw} (\partial_x Lw)+2i(\bar{g}(w,\bar{w})+ \nu_1w)  \partial_x\bar{w}(\partial_x \overline{L w})Lw\nonumber\\
&-i(g_{w}Lw+g_{\bar{w}}\overline{Lw}) (\partial_x w)^2  \overline{Lw}+i(\bar{g}_{\bar{w}}\overline{Lw}+\overline{g}_{{w}} {Lw}) (\partial_x \bar{w})^2   {Lw} \label{ga32}\\
&+O(x|w|^3  |\partial_x w|^2|Lw|) \label{gag2}.
\end{align}
The  contributions of  (\ref{gag2}), (\ref{yu00}),  line (\ref{ga32}) and the last two terms of (\ref{ga12}) are easy to bound, in fact, one has
\begin{align*}
&\int_{\Bbb R} |H_{w}(\nu_1\bar{w}+g(w,\bar{w}))(\partial_xw)^2||Lw|^2 dx
 + \int_{\Bbb R} |H (w,\bar{w})(g_{w}Lw+g_{\bar{w}}\overline{Lw} ) (\partial_x w)^2  \overline{Lw}|dx
\\
& +\int_{\Bbb R}2|H(w,\bar{w})||(\nu_1\bar {w}+ g(w,\bar{w}))w \partial_x w\overline{Lw}|dx+\int_{\Bbb R}|x|| H (w,\bar{w})| |w|^3  |\partial_x w|^2|Lw| dx \\
&\lesssim \| w\|^2_{W^{1,\infty}_x}  \|Lw\|^2_{L^2_x}+ \|Lw\|_{L^2_x}\|w\|_{L^2_x}\|w\|^2_{w^{1,\infty}_x}+\|xw\|_{L^2_x}\|Lw\|_{L^2_x}\|w\|^4_{w^{1,\infty}_x}.
\end{align*}
For the first two terms of RHS of (\ref{ga12}), integration by parts gives
\begin{align*}
&\int_{\Bbb R} [iH_{w}  \Delta w-iH_{\bar{w}}\Delta\bar{w}  ]|Lw|^2 dx\\
&=\int_{\Bbb R} (-iH_{w}  \partial_x{w}+iH_{\bar{w}} \partial_x \bar{w} )[(\partial_xLw) \overline{Lw}+Lw\partial_x\overline{Lw}] dx\\
&-\int_{\Bbb R}[i(\partial_xH_{w})\bar{w}\partial_x{w} -i(\partial_xH_{\bar{w}})w\partial_x\bar{w}  ]|L w|^2dx\\
&-\int_{\Bbb R}|\partial_x{w}|^2[i H_{w}  -i H_{\bar{w}}  ]|L w|^2dx.
\end{align*}
And three further holds
\begin{align*}
&\int_{\Bbb R} [iH_{w}  \Delta w-iH_{\bar{w}}\Delta\bar{w}  ]|Lw|^2 dx\\
&=\int_{\Bbb R} (-iH_{w}  \partial_x{w}+iH_{\bar{w}} \partial_x \bar{w} )[(\partial_xLw) \overline{Lw}+Lw\partial_x\overline{Lw}] dx+{\rm I_1},
\end{align*}
where ${\rm I_1}$ is dominated by
\begin{align*}
|{\rm{I_1}}|\lesssim \|w\|^3_{W^{1,\infty}_x}\|Lw\|^2_{L^2_x}.
\end{align*}
Again by integration by parts, the integral associated with the RHS of (\ref{ga22}) becomes
\begin{align*}
&\int_{\Bbb R} H(w,\bar{w})[i\overline{Lw}\Delta Lw -i Lw \Delta \overline{Lw}]dx\\
&=\int_{\Bbb R} (H_{w}\partial_x{w}+H_{\bar{w}}\partial_x \bar{w})[-i(\partial_xLw) \overline{Lw}+iLw\partial_x\overline{Lw}]dx.
\end{align*}
Thus the RHS of (\ref{ga5}) equals
\begin{align*}
&\int_{\Bbb R}(H_{w} \partial_tw+H_{\bar{w}}\partial_t \bar{w})Lw\overline{Lw}dx+\int_{\Bbb R}H(w,\bar{w})[ ( \partial_tLw)\overline{Lw}+  Lw \partial_t\overline{Lw}]dx\\
&= \int_{\Bbb R} (-iH_{w}  \partial_x{w}+iH_{\bar{w}} \partial_x \bar{w} )[(\partial_xLw) \overline{Lw}+Lw\partial_x\overline{Lw}] dx\\
&+\int_{\Bbb R} (H_{w}\partial_x{w}+H_{\bar{w}}\partial_x \bar{w})[-i(\partial_xLw) \overline{Lw}+iLw\partial_x\overline{Lw}]dx \\
&+\int_{\Bbb R} H(w,\bar{w})  [2i(\bar{g}(w,\bar{w})+ \nu_1w)  \partial_x\bar{w}(\partial_x \overline{L w})Lw-2i( g(w,\bar{w})+ \nu_1\bar{w})   \partial_x w\overline{Lw} (\partial_x Lw)]dx\\
&+{\rm{I_1}+\rm{I_2}},
\end{align*}
where ${\rm I_2}$ denotes the sum of (\ref{gag2}), line (\ref{ga32}), (\ref{yu00}) and the last two terms of (\ref{ga12}).
Notice that (\ref{Re}) now becomes
\begin{align*}
H(w,\bar{w})(2g(w,\bar{w})+\nu_1\bar{w})+H_z(w,\bar{w})=O(|w|^3).
\end{align*}
Then   we arrive at
\begin{align*}
&\int_{\Bbb R}(H_{w} \partial_tw+H_{\bar{w}}\partial_t \bar{w})Lw\overline{Lw}dx+\int_{\Bbb R}H(w,\bar{w})[ ( \partial_tLw)\overline{Lw}+  Lw \partial_t\overline{Lw}]dx\\
&={\rm{I_1}+\rm{I_2}}+\int_{\Bbb R}O\left(|w|^3|\partial_x w|(t|\partial_x z|+|x\partial_xw|+|w|)|Lw|\right)dx.
\end{align*}
Therefore, one obtains
\begin{align}
 \frac{d}{dt}\mathcal{L}(w)
 &\lesssim   \|w\|^2_{W^{1,\infty}_x}\|Lw\|^2_{L^2_x}+ \|Lw\|_{L^2_x}\|w\|_{L^2_x}\|w\|^2_{w^{1,\infty}_x}
 +\|xw\|_{L^2_x}\|Lw\|_{L^2_x}\|w\|^4_{w^{1,\infty}_x}\nonumber\\
 &+ \|w\|^4_{W^{1,\infty}_x}\|Lw\|_{L^2_x}(t\|\partial^2_x w\|_{L^2_x}+\|\langle x\rangle w\|_{H^1_x})
 .\label{fgh}
\end{align}
Since $\|w\|_{L^{\infty}([0,T]\times \Bbb R)}\le \omega$, one has
\begin{align}
 \mathcal{L}(w) \sim \|Lw\|^2_{L^2_x}.\label{xxxxfgh}
\end{align}
Then the lemma follows by (\ref{fgh}).
\end{proof}

\begin{Corollary}\label{K3s}
Assume that
\begin{align}\label{67s}
\sup_{t\in[0,T]} \langle t\rangle^{\frac{1}{2}}\|w(t)\|_{W^{2,\infty}_x}\le \epsilon.
\end{align}
Then under the assumption (\ref{KEY}) we have
\begin{align}
 \|Lw(t)\|_{L^2_x}&\lesssim \epsilon \langle t\rangle^{\epsilon}. \label{P7s}
\end{align}
\end{Corollary}
\begin{proof}
By (\ref{67s}), {Lemma} \ref{NP} and  Lemma \ref{K1}, we get
\begin{align*}
 &\frac{d}{dt}\mathcal{L}(w)
 \lesssim  \epsilon^2 \langle t\rangle^{-1}\mathcal{L}(w)+ \epsilon^2 \langle t\rangle^{-1+\epsilon}\sqrt{\mathcal{L}(w)}.
\end{align*}
Then  (\ref{P7s}) follows by Gronwall inequality and (\ref{xxxxfgh}).
\end{proof}

\section{Setting of bootstrap}

Assume that $T\in (0,\infty]$ is the largest time for which
\begin{align}\label{2A}
\sup_{t\in[0,T]}  \langle t\rangle^{\frac{1}{2}}\| w\|_{W^{2,\infty}_{x}} \le \varepsilon.
\end{align}
Recall that $w:=z+\gamma_1z^2+\gamma_2z^3+\gamma_3z^4$. By Duhamel principle and (\ref{3equation1}), we obtain
\begin{align*}
w(t)&=e^{it\Delta}w_0-i\int^{t}_0e^{i(t-\tau)\Delta} G(w,\bar{w})d\tau\\
G(w,\bar{w})&:=\nu_1\overline{w}( \partial_x w)^2+
 \nu_2\overline{w}w( \partial_x w)^2 +\nu_3\overline{w}^2( \partial_x w)^2
 + O(|w|^3) ( \partial_x w)^2.
\end{align*}

We will use the following dispersive estimate.
\begin{Lemma}[\cite{hn}]\label{agh1}
Let $g(t,x)$ be any given $\Bbb C$ valued function with finite norm $\|\widehat{g}\|_{L^{\infty}_x(\Bbb R)}+\|g\|_{H^{0,\gamma}_x(\Bbb R)}$.  For any $t\in\Bbb R$, and any $\gamma>\frac{1}{2}+2\beta$, there holds
\begin{align*}
\|e^{it\Delta}g\|_{L^{\infty}_x(\Bbb R)}\lesssim {t^{-\frac{1}{2}}\|\widehat{g}\|_{L^{\infty}_x(\Bbb R)} +\frac{1}{t^{\frac{1}{2}+\beta}}}\|g\|_{H^{0,\gamma}_x(\Bbb R)}.
\end{align*}
\end{Lemma}

\section{Global bounds of the solution  }

Let $f=e^{-it\Delta}w$, then
 (\ref{Gvbn}) can be written as
\begin{align}\label{J3y}
f(t)=f(1)+\int^{t}_1e^{-i\tau\Delta}(c\bar{w}+\nu_2\overline{w}w+\nu_3\overline{w^2}+O(|w|^3))(\partial_x w)^2d\tau,
\end{align}
where $c=-\frac{1}{2}K(Q)h_0$, see Lemma \ref{Bn2}.
Recall $\widehat{f}$ denotes the Fourier transform of $f$. (\ref{J3y}) can be written as
\begin{align}
\widehat{f}(t,\sigma)&=\widehat{f}(\upharpoonright_{t=1})- \frac{i}{2\pi}\int^{t}_1 ce^{i\tau\phi_0}(\eta-\xi)(\sigma-\eta)\widehat{\overline{f}}(\xi)\widehat{f}(\eta-\xi)\widehat{f} (\sigma-\eta)d\xi d\eta d\tau+\int^t_1\mathcal{R}d\tau\\
&- \sum_{\imath_1,\imath_2\in\{+ ,- \}}\frac{i}{2\pi}\int^{t}_1\nu_{\imath_1,\imath_2}e^{i\tau\phi_{\imath_1,\imath_2}}(\zeta-\eta)(\sigma-\zeta)\widehat{{f}^{\imath_1}}(\xi)\widehat{f^{\imath_2}}(\eta-\xi)\widehat{f}(\zeta-\eta)\widehat{f} (\sigma-\zeta)d\xi d\eta d\zeta d\tau,\label{sJ3}
\end{align}
where we denote
\begin{align*}
\phi_0&=\sigma^2+\xi^2-(\eta-\xi)^2-(\sigma-\eta)^2.\\
\phi_{\imath_1,\imath_2}&=\sigma^2+\imath_1\xi^2+\imath_2(\xi-\eta)^2-(\zeta-\eta)^2-(\sigma-\zeta)^2, \mbox{  }\imath_1,\imath_2\in \{+,-\}\\
\mathcal{R}&=\mathcal{F}[e^{-i\tau \Delta}\left(O(|w|^3) (\partial_x w)^2\right)]\\
f^{\imath}&=f, \mbox{ }{\rm if} \mbox{ }\imath=+; \mbox{ }f^{\imath}=\overline{f}, \mbox{ }{\rm if} \mbox{ }\imath=-.
\end{align*}

\subsection{Estimates  of  $f$ in $H^{k,l}$}

Denote the RHS of   (\ref{Gvbn})   by $G(w,\bar{w})$.
Recall that $S:=x\partial_x+2t\partial_t$ and
\begin{align*}
i\partial_t w+\Delta w=G(w,\bar{w}).
\end{align*}
With $f(t)=e^{-it\Delta} w(t)$, one has
\begin{align*}
(x\partial_x f)(t)=e^{it\Delta} \left(S w-2t G(w,\bar{w})\right).
\end{align*}

Recall that Corollary \ref{K3} proves
\begin{align*}
\|Sw(t)\|_{H^2_x}\lesssim (1+t)^{\delta}\|w_0\|_{H^{3,1}_x}.
\end{align*}
Thus we have
\begin{align}
\|x\partial_xf (t)\|_{H^2_x}&\lesssim \|Sw\|_{H^2_x}+ t\|w\|_{H^3_x}\|w\|^2_{W^{2,\infty}_x}\lesssim (1+t)^{\delta} \|w_0\|_{H^{3,1}_x}.\label{p90o}
\end{align}

We summarize the estimates of $w$ and $f$ as the following lemma.
\begin{Lemma}\label{B1}
Assume that  $w_0$ satisfies (\ref{suption}), (\ref{2A}) and (\ref{KEY}) hold. Then for $n_*$ large,  there exists a small constant $\delta$ depending only on $\varepsilon,n_*$  such that
\begin{align}
\sup_{t\in [0,T]}\langle t \rangle^{-\delta}\|w(t)\|_{H^{n_*}_x}&\lesssim \|w_0\|_{H^{n_*}_x}\label{E1}\\
\sup_{t\in [0,T]}(1+t)^{-\delta}\|f(t)\|_{H^{3,1}_x}&\lesssim \|w_0\|_{H^{3,1}_x}.\label{p901}
\end{align}
\end{Lemma}
\begin{proof}
(\ref{E1})  follows  directly by Corollary \ref{K3}.
By computation, $xe^{it\Delta}g=-Lg$, so   Corollary \ref{K3s}   yields
\begin{align*}
\|f (t)\|_{H^{0,1}_x}\lesssim \|L w\|_{L^2_x}\lesssim  t^{\epsilon} \|w_0\|_{H^{0,1}_x}.
\end{align*}
And it is easy to check
\begin{align*}i
\sum^{3}_{j=1}\|x\partial^j_xf (t)\|_{L^2_x}\lesssim \|x\partial_xf\|_{H^2_x}+\|f\|_{H^3_x}.
\end{align*}
Then (\ref{p901}) follows by (\ref{p90o}) and bounds of $\|w\|_{H^3_x}$.
\end{proof}

\subsection{Estimates  in the Fourier Space with assumption (\ref{KEY})}

 Let $n_*$ be sufficiently large and $\varepsilon$ be sufficiently small such that $\delta<0.01$.

Under the assumption (\ref{KEY}), (\ref{sJ3}) can be written as
\begin{align}
\widehat{f}(t,\sigma)&=\widehat{f}(\upharpoonright_{t=1})- \frac{i}{2\pi}\int^{t}_1 ce^{i\tau\phi_0}(\eta-\xi)(\zeta-\eta)\widehat{\overline{f}}(\xi)\widehat{f}(\eta-\xi)\widehat{f} (\zeta-\eta)d\xi d\eta d\tau+\int^{t}_1\mathcal{R}d\tau \label{J3}
\end{align}
where we denote
\begin{align*}
\phi_0&=\sigma^2+\xi^2-(\eta-\xi)^2-(\sigma-\eta)^2.
\end{align*}

{\it Space-time resonance analysis of $\phi_0$. }
By computation, we have the time resonance set of $\phi_0$ is
\begin{align*}
\mathcal{R}^{\phi_0}_{t}=\{(\zeta,\xi,\eta):\phi_0=0\}.
\end{align*}
The space resonance sets of $\phi_0$ are
\begin{align*}
\mathcal{R}^{\phi_0}_{s,\xi}&=\{(\zeta,\xi,\eta):\partial_{\xi}\phi_0=0\}=\{(\zeta,\xi,\eta):\eta=0\}\\
\mathcal{R}^{\phi_0}_{s,\eta}&=\{(\zeta,\xi,\eta):\partial_{\eta}\phi_0=0\}=\{(\zeta,\xi,\eta):2\eta-\zeta-\xi=0\}.
\end{align*}
The space-time resonance set is
\begin{align*}
\mathcal{R}^{\phi_0}_{s,t}&=\{(\zeta,\xi,\eta): \zeta=-\xi, \eta=0\}.
\end{align*}

\underline{{\it Estimates of the leading cubic term.}}
By the stationary phrase method, the RHS of (\ref{J3}) can be further decomposed.  In fact,  similar to \cite{ip2}, by change of variables and Plancherel identity, we find
\begin{align*}
&\frac{i}{2\pi}\int^{t}_1 e^{i\tau\phi_0}(\eta-\xi)(\sigma-\eta)\widehat{\overline{f}}(\xi)\widehat{f}(\eta-\xi)\widehat{f} (\sigma-\eta)d\xi d\eta d\tau\\
&=\frac{i }{2\pi}\int^{t}_1e^{2i\tau\eta\zeta}(\sigma-\zeta)(\sigma-\eta)\overline{\widehat{{f}}(\sigma-\eta-\zeta)}\widehat{f}(\sigma-\eta)\widehat{f} (\sigma-\zeta)d\zeta d\eta d\tau \\
&= \frac{i }{2\pi}\int^t_1\int_{\Bbb R}\mathcal{F}_{\eta,\zeta}[e^{2i\tau\eta\zeta}]\mathcal{F}^{-1}_{\eta,\zeta}[E(\sigma,\eta,\zeta)]  d\tau d\widetilde{\eta} d\widetilde{\zeta} \nonumber\\
&= \frac{i }{2\pi}\int^t_1\int_{\Bbb R^2} (\frac{1}{2\tau} e^{-\frac{i}{2\tau}\widetilde{\eta}\widetilde{\zeta}})\mathcal{F}^{-1}_{\eta,\zeta}[E(\sigma,\eta,\zeta)]  d\widetilde{\eta} d\widetilde{\zeta} d\tau \\
&= \frac{i }{2\pi}\int^t_1\int_{\Bbb R^2}\frac{1}{2\tau}\mathcal{F}^{-1}_{\eta,\zeta}[E(\sigma,\eta,\zeta)]  d\widetilde{\eta} d\widetilde{\zeta} d\tau\\
&-\frac{i }{2\pi}\int^t_1\int_{\Bbb R^2}\frac{1}{2\tau} [ e^{-\frac{i}{2\tau}\widetilde{\eta}\widetilde{\zeta}}-1]\mathcal{F}^{-1}_{\eta,\zeta}[E(\sigma,\eta,\zeta)]  d\widetilde{\eta} d\widetilde{\zeta} d\tau,
\end{align*}
where we denote
\begin{align}\label{hh1}
E(\sigma,\eta,\zeta):=(\sigma-\zeta)(\sigma-\eta)\overline{\widehat{{f}}(\sigma-\eta-\zeta)}\widehat{f}(\sigma-\eta)\widehat{f} (\sigma-\zeta).
\end{align}
Observe that
\begin{align*}
\int_{\Bbb R^2}\mathcal{F}^{-1}_{\eta,\zeta}[E(\sigma,\eta,\zeta)]  d\widetilde{\eta} d\widetilde{\zeta}=2\pi E(\sigma,0,0)
=2\pi |\sigma\widehat{f}(\tau,\sigma)|^2   \widehat{f}(\tau,\sigma).
\end{align*}
Thus defining $\widehat{F}(t,\sigma)$ by
\begin{align}\label{PQ9}
\widehat{F}(t,\sigma)&=e^{ic\int^{t}_1\frac{1}{2\tau}\sigma^2|\widehat{f}(\tau,\sigma)|^2d\tau} \widehat{f}(t,\sigma),
\end{align}
we get
\begin{align}\label{Y7m}
\widehat{F}(t,\sigma)&=\widehat{f}(1,\sigma)+e^{ic\int^{t}_1\frac{1}{2\tau}\sigma^2|\widehat{f}(\tau,\sigma)|^2d\tau}\int^t_1{\mathcal{R}}d\tau\nonumber\\
&-e^{ic\int^{t}_1\frac{1}{2\tau}\sigma^2|\widehat{f}(\tau,\sigma)|^2d\tau}\frac{ic}{2\pi}\int^t_1\int_{\Bbb R^2}\frac{1}{2\tau} [ e^{-\frac{i}{2\tau}\widetilde{\eta}\widetilde{\zeta}}-1]\mathcal{F}^{-1}_{\eta,\zeta}[E(\sigma,\eta,\zeta)]  d\widetilde{\eta} d\widetilde{\zeta} d\tau,
\end{align}
where ${\mathcal R}$ denotes 5 order terms and higher order terms.

By computation, we have
\begin{align*}
\mathcal{F}^{-1}_{\eta,\zeta}[E(\sigma,\eta,\zeta)] =(2\pi)^{\frac{1}{2}}e^{i\sigma(\widetilde{\eta}+\widetilde{\zeta})}\int_{\Bbb R} e^{-ix\sigma}\bar{f}(x)\partial_xf(x-\widetilde{\eta})\partial_x f(x-\widetilde{\zeta})dx.
\end{align*}
Hence, for $0<\nu<\frac{1}{4}$, one has
\begin{align*}
&\langle \sigma\rangle^{2}  \int_{\Bbb R^2}|\widetilde{\eta}\widetilde{\zeta}|^{\nu}| \mathcal{F}^{-1}_{\eta,\zeta}[E(\sigma,\eta,\zeta)]|  d\widetilde{\eta} d\widetilde{\zeta}\\
&\lesssim  \sum_{i_1+i_2+i_3\le 2}\left| \int_{\Bbb R^3}\partial^{i_1}_{x}\bar{f}(x)(|x-\widetilde{\eta}|^{2\nu}+|x|^{2\nu})(|x-\widetilde{\zeta}|^{2\nu}+|x|^{2\nu})\partial^{i_2+1}_xf(x-\widetilde{\eta})\partial^{i_3+1}_x f(x-\widetilde{\zeta})dx d\widetilde{\eta} d\widetilde{\zeta}\right|\\
&\lesssim  \|f\|^3_{H^{3,1}_x}.
\end{align*}
Thus by {Lemma} \ref{B1}, we conclude that the last term in the RHS of (\ref{Y7m}) satisfies
\begin{align}
&\left\|\langle \sigma\rangle^2 \frac{ic}{2\pi}\int^t_1\int_{\Bbb R^2}\frac{1}{2\tau} [ e^{-\frac{i}{2\tau}\widetilde{\eta}\widetilde{\zeta}}-1]\mathcal{F}^{-1}_{\eta,\zeta}[E(\sigma,\eta,\zeta)]  d\widetilde{\eta} d\widetilde{\zeta} d\tau\right\|_{L^{\infty}_{\sigma}}\nonumber\\
&\lesssim  \int^{t}_1\tau^{-1-\nu} \|f\|^3_{H^{3,1}_x}d\tau\nonumber\\
&\lesssim   \varepsilon^3.\label{U1}
\end{align}

\underline{{\it Estimates of the high order terms.}} Since the high order term $\mathcal{R}$  in the RHS of (\ref{Y7m}) is at least 5 order,
 by {Lemma} \ref{B1}, we conclude that  $\mathcal{R}$ satisfies
\begin{align}\label{U2}
&\int^t_1\left\|\langle \sigma\rangle^2 \mathcal{R} \right\|_{L^{\infty}_{\sigma}}d\tau\lesssim  \int^{t}_1\|w\|^3_{W^{2,\infty}_x}\|w\|^2_{H^3_x}d\tau\lesssim   \varepsilon^3.
\end{align}

Therefore, we deduce that
\begin{Lemma}\label{Rtgo}
Under the assumptions (\ref{2A}) and (\ref{KEY}), for any $t\in[0,T]$, we have
\begin{align}\label{j678}
&\|\langle \sigma\rangle^2\widehat{w}(t)\|_{L^{\infty}_{\sigma}}\lesssim \epsilon_*+\varepsilon^3.
\end{align}
Moreover, there exists $\kappa>0$ such that for any $0\le t_1\le t_2\le T$ there holds
\begin{align}\label{Rtbnn}
\|\langle \sigma\rangle^2[ F(t_1,\sigma)-F(t_2,\sigma)]\|_{L^{\infty}_{\sigma}} \lesssim\langle   t\rangle^{-\kappa}_1.
\end{align}
\end{Lemma}
\begin{proof}
Observe that $F(t,\sigma)$ defined by
(\ref{PQ9}) satisfies $|\widehat{w}(t,\sigma)|=|F(t,\sigma)|$.
Then, for $t\in [0,1]$, Sobolev embeddings  and discussions in Section 2.1 give (\ref{j678}). For $t\in [1,T]$, (\ref{j678}) follows by (\ref{U1}) and (\ref{U2}).
(\ref{Rtbnn}) results from
 (\ref{Y7m}) and similar estimates as (\ref{U1})-(\ref{U2}) with integration interval replaced by $[t_1,t_2]$.
 \end{proof}

\section{Proof of Theorem \ref{XS2}: Decay estimates, modified scattering v.s. scattering}

Let $Q$ be an intrinsic vanishing point and  $w_0$ be an initial data satisfying Theorem \ref{XS2}.

Most parts of this section are standard, we present a detailed proof because some of them are useful in later sections.
Lemma \ref{agh1} shows
\begin{align*}
\| w(t)\|_{W^{2,\infty}_{x}}&\lesssim t^{-\frac{1}{2}}\|\langle\sigma \rangle^2 \widehat{w}\|_{L^{\infty}_{\sigma} }+t^{-\frac{5}{8}}\| f\|_{H^{2,1}_x}.
\end{align*}
So Lemma \ref{Rtgo}  gives
\begin{align*}
\| w(t)\|_{W^{2,\infty}_{x}} &\lesssim t^{-\frac{1}{2}}(\epsilon_*+\varepsilon^3).
\end{align*}
Therefore, by bootstrap assumption $T=\infty$. And thus one has the decay estimates
\begin{align*}
\|  w(t)\|_{W^{2,\infty}_{x}} \lesssim \langle t\rangle^{-\frac{1}{2}}, \mbox{ }\forall \mbox{ }t>0.
\end{align*}
And (\ref{Rtbnn}) implies that there exists $U\in \langle \sigma \rangle^{-2}L^{\infty}_{\sigma}$ such that
\begin{align}\label{Bc78}
\|\langle \sigma \rangle^2 [F(t,\sigma)-U(\sigma)]\|_{L^{\infty}_{\sigma}}\lesssim t^{-\kappa},
\end{align}
from which (\ref{asy}) follows.

Let
\begin{align*}
\Psi(t):=\int^t_{1}( |\sigma\widehat{f}(\tau,\sigma)|^2-  |\sigma\widehat{f}(t,\sigma)|^2)\frac{1}{2\tau} d\tau.
\end{align*}
By (\ref{Rtbnn}), for any $t_1<t_2<\infty$,
\begin{align*}
\|\Psi(t_1)-\Psi(t_2)\|_{\langle \sigma\rangle^{-1}L^{\infty}_{\sigma}}\lesssim t^{-\frac{1}{2}\kappa}_1.
\end{align*}
So there exists a real valued function $\Theta$ such that for $t$ large
\begin{align*}
\|\Psi(t,\sigma)-\Theta(\sigma) \|_{\langle \sigma\rangle^{-1}L^{\infty}_{\sigma}}\lesssim t^{-\frac{1}{2}\kappa},
\end{align*}
which further shows
\begin{align*}
 \int^t_{1} |\sigma\widehat{f}(\tau,\sigma)|^2\frac{1}{2\tau} d\tau=|\sigma\widehat{f}(t,\sigma)|^2 \ln t^{\frac{1}{2}}+\Theta(\sigma)+O_{\langle \sigma\rangle^{-1}L^{\infty}_{\sigma}}( t^{-\frac{1}{2}\kappa}).
\end{align*}
Thus by the definition of $F(t,\zeta)$ and  (\ref{Bc78}), we have  for $t$ large
\begin{align*}
 \int^t_{1} |\sigma\widehat{f}(\tau,\sigma)|^2\frac{1}{\tau} d\tau=\Theta(\sigma)+|\sigma U(\sigma)|^2 \ln t^{\frac{1}{2}}+O_{\langle \sigma\rangle^{-1}L^{\infty}_{\sigma}}(t^{-\frac{1}{4}\kappa}).
\end{align*}
Hence
\begin{align}\label{Kj00}
\widehat{f}(t,\sigma)= e^{i\int^t_1\frac{c}{\tau}|\sigma\widehat{f}(\sigma)|^2d\tau}F(t,\sigma)=U(\sigma)\exp(ci|\sigma U(\sigma)|^2 \ln t^{\frac{1}{2}}+ic\Theta(\sigma))+\widetilde{R}( t,\sigma),
\end{align}
where  the remainder $\widetilde{R}( t,\sigma)$ fulfills
\begin{align}\label{Kl00}
\|\widetilde{R}( t,\sigma)\|_{\langle \sigma\rangle^{-1}L^{\infty}_{\sigma}}\lesssim t^{-\frac{1}{4}\kappa}.
\end{align}
Hayashi-Naumkin \cite{hn} has proved that
\begin{align}\label{9AAA}
u(t,x)=\frac{e^{\frac{i|x|^2}{4t} } }{(4  it)^{\frac{1}{2}}} \mathcal{F}[e^{-it\Delta}u(t)]\left(2t,\frac{x}{2t}\right)+R(t,x),
\end{align}
where $R(t,x)$ satisfies
\begin{align}\label{9AAAs}
 R(t,x) =\frac{e^{\frac{i|x|^2}{4t} } }{(4\pi it)^{\frac{1}{2}}}\int_{\Bbb R} e^{-\frac{ixy}{2t}}(e^{\frac{iy^2}{4t}}-1)(e^{-it\Delta}u(t))dy.
\end{align}
It is easy to verify by Plancherel identity and change of variables  that
\begin{align}\label{K400}
\|R(x,t)\|_{L^2_x}&\lesssim t^{-\frac{1}{2}}\||x| e^{-it\Delta}u(t)\|_{L^2_x}.
\end{align}
And  \cite{hn} proved the point-wise estimate
\begin{align}\label{K500}
\|R(x,t)\|_{L^{\infty}_x}&\lesssim t^{-\frac{5}{8}}\|e^{-it\Delta}u(t)\|_{H^{0,1}}, \mbox{ }\forall \mbox{ }|t|\ge 1.
\end{align}
Therefore, (\ref{Kj00}) implies
\begin{align*}
w(t,x)&=\frac{e^{\frac{i|x|^2}{4t} } }{(2 it)^{\frac{1}{2}}} \mathcal{F}[e^{-it\Delta}w(t)]\left(2t,\frac{x}{2t}\right)+R(t,x)\\
&=\frac{e^{\frac{i|x|^2}{4t}} }{(2it)^{\frac{1}{2}}} \mathcal{F}[f]\left(2t,\frac{x}{2t}\right)+R(t,x)\\
&=\frac{e^{\frac{i|x|^2}{4t}} }{(2it)^{\frac{1}{2}}}  U(\frac{x}{2t})e^{ \frac{1}{2}ic|\frac{x}{2t}U(\frac{x}{2t})|^2\ln (2t)+ic\Theta(\frac{x}{2t})} +\frac{e^{\frac{i|x|^2}{4t}} }{(2it)^{\frac{1}{2}}}\widetilde{R}(2t,\frac{x}{2t})+R(t,x),
\end{align*}
where $\widetilde{R}(t,x)$ satisfies (\ref{Kl00}) and ${R}(t,x)$ fulfills (\ref{K400})-(\ref{K500}).
So if $\kappa>0$ is taken to be sufficiently small, for $t$ large there hold
\begin{align*}
 &\|w(t,x)-\frac{e^{\frac{i|x|^2}{4t}} }{(2it)^{\frac{1}{2}}}  U(\frac{x}{2t})e^{ \frac{1}{2}ic|\frac{x}{2t}U(\frac{x}{2t})|^2\ln (2t)+ic\Theta(\frac{x}{2t})}\|_{ L^{\infty}_x}
 \\&\lesssim t^{-\frac{1}{2}}\|\widetilde{R}(2t,\frac{x}{2t})\|_{  L^{\infty}_x} +\|R(t,x)\|_{ L^{\infty}_x}
 \lesssim t^{-\frac{1}{2}-\frac{1}{4}\kappa}\|f\|_{H^{0,1}_x}  \lesssim t^{-\frac{1}{2}-\frac{1}{8}\kappa}\|w_0\|_{H^{0,1}_x},
\end{align*}
and
\begin{align*}
 &\|w(t,x)-\frac{e^{\frac{i|x|^2}{4t}} }{(2it)^{\frac{1}{2}}}  U(\frac{x}{2t})e^{ \frac{1}{2}ic|\frac{x}{2t}U(\frac{x}{2t})|^2\ln (2t)+ic\Theta(\frac{x}{2t})}\|_{L^{2}_x}
\\&  \lesssim t^{-\frac{1}{2}}\|\widetilde{R}(2t,\frac{x}{2t})\|_{L^{2}_x} +\|R(t,x)\|_{L^{2}_x}
 \lesssim t^{-\frac{1}{4}\kappa}\|f\|_{H^{0,1}_x}  \lesssim t^{-\frac{1}{8}\kappa}\|w_0\|_{H^{0,1}_x}.
\end{align*}
And thus (\ref{GGG})   follows by letting $\varsigma=\frac{1}{8}\kappa$,
\begin{align*}
\psi(x)= U(x)e^{ic\Theta(x)},
\end{align*}
and noting that
\begin{align*}
\|w-z\|_{L^{\infty}_x}&\lesssim \|w^2\|_{L^{\infty}_x}\lesssim t^{-1}\\
\|w-z\|_{L^{2}_x}&\lesssim \|w^2\|_{L^{2}_x}\lesssim t^{-\frac{1}{2}}.
\end{align*}

\section{Proof of Theorem 1.2}
In this section, we prove
Theorem 1.2 by contradiction argument.

\subsection{Multi-linear estimates}

To do nonlinear estimates, we need the following lemma due to \cite{ip}, see also  Lemma 5.2 of \cite{ip2}.
\begin{Lemma}[\cite{ip}]\label{457}
Suppose that  $m(\xi,\eta,\zeta)\in L^{1}(\Bbb R\times \Bbb R\times \Bbb R)$ satisfies
\begin{align*}
\|\int_{\Bbb R^3} m(\xi,\eta,\zeta)e^{ix\xi}e^{iy\xi}e^{iz\zeta}d\xi d\eta d\zeta\|_{L^{1}_{x,y,z}}\le A,
\end{align*}
for some $0<A<\infty$.  Define
\begin{align*}
 \mathcal{F}{T_{m}(h_1,h_2,h_3,h_4)}(\sigma)=\int_{\Bbb R^3} m(\xi,\eta,\zeta)h_1(\xi)h_2(\eta-\xi)h_3(\zeta-\eta)h_4(\sigma-\zeta)d\xi d\eta d\zeta.
\end{align*}
Then for any $i_1,i_2,i_3,i_4$  which are rearrangement of $\{1,2,3,4\}$ and any $p_1,p_2,p_3,p_4\in [1,\infty]$ with $\sum^4_{j=1}\frac{1}{p_j}=1$, there holds
\begin{align*}
\|\mathcal{F}{T_{m}(h_1,h_2,h_3,h_4)}\|_{L^{\infty}_{\sigma}}  \lesssim A\| {h_{i_1}}\|_{L^{p_1}_x} \| {h_{i_2}}\|_{L^{p_2}_x} \| {h_{i_3}}\|_{L^{p_3}_x}\| {h_{i_4}}\|_{L^{p_4}_x}.
\end{align*}
\end{Lemma}

The following lemma  will be useful in nonlinear estimates of later subsections.
In Lemma \ref{Coifman},  $\phi$ denotes a  quadratic form of  $(\xi,\eta,\zeta,\sigma)\in \Bbb R^4$.
\begin{Lemma}\label{Coifman}
Assume that $g_i$ is a homogenous function such that $g_{i}(\lambda y)=\lambda^{\ell_i}g_{i}( y)$ for all $\lambda>0,y\in\Bbb R$ with $i=1,2,3,4$.
Let $\vartheta_i$ with $i=1,2,3,4$ be a smooth function such that $g_i(y)\vartheta_i(y)$ is smooth in $\Bbb R$ and
\begin{align}\label{yoyo}
 \| g_1(\phi)\vartheta_1(\phi) \|_{W^{2,\infty}(\Bbb R^3)}+\sum^{4}_{i=2}\|   g_i(y)\vartheta_i(y) \|_{W^{2,\infty}_y}\lesssim 1.
\end{align}
Define
$$m_{t}=g_1(\phi)g_2(\partial_{\xi}\phi) g_3(\partial_{\eta}\phi)  g_4(\partial_{\zeta}\phi)\vartheta_1(t^{2\mu}\phi)\vartheta_2(t^{\mu}\partial_{\xi}\phi)\vartheta_3(t^{\mu}\partial_{\eta}\phi)\vartheta_4(t^{\mu}\partial_{\zeta}\phi).$$
Then for  $\rho>\frac{3}{2}$, $\lambda>\frac{1}{2}$, there hold
\begin{align}
&|\int_{\Bbb R^3}m_t(\xi,\eta,\zeta)\widehat{ {f}_1}(\xi)\widehat{ {f}_2}(\eta-\xi)\widehat{{\partial_xf_3}}(\zeta-\eta)\widehat{\partial_xf_4}(\sigma-\zeta)d\xi d\eta d\zeta|\nonumber\\
&\lesssim t^{-2\mu ( \ell_1+\frac{1}{2}(\ell_2+\ell_3+\ell_4))} t^{\frac{3}{2}\mu}\|f_1\|_{L^{\infty}_x}\|f_2\|_{W^{\lambda,\infty}_x } \|f_3\|_{H^{\rho}_x}\|f_4\|_{H^{\rho}_x}\label{v1}\\
&|\int_{\Bbb R^3}m_t(\xi,\eta,\zeta)\widehat{ {f}_1}(\xi)\widehat{ {f}_2}(\eta-\xi)\widehat{{\partial_xf_3}}(\zeta-\eta)\widehat{f_4}(\sigma-\zeta)d\xi d\eta d\zeta|\nonumber\\
&\lesssim t^{-2\mu ( \ell_1+\frac{1}{2}(\ell_2+\ell_3+\ell_4))} t^{\frac{3}{2}\mu}\|f_1\|_{L^{\infty}_x}\|f_2\|_{W^{\lambda,\infty}_x }\|f_3\|_{H^{1+\lambda}_x}\|f_4\|_{H^{\lambda}_x},\label{v2}
\end{align}
and
\begin{align}
&|\int_{\Bbb R^3}m_t(\xi,\eta,\zeta))\widehat{ {f}_1}(\xi)\widehat{ {f}_2}(\eta-\xi)\widehat{{\partial_xf_3}}(\zeta-\eta)\widehat{\partial_xf_4}(\sigma-\zeta)d\xi d\eta d\zeta|\nonumber\\
&\lesssim t^{-2\mu ( \ell_1+\frac{1}{2}(\ell_2+\ell_3+\ell_4))}  t^{\frac{3}{2}\mu} \|f_1\|_{L^{2}_x}\|f_2\|_{W^{\lambda,\infty}_x }\|f_3\|_{H^{\rho}_x}\| f_4\|_{W^{\rho,\infty}_x} \label{v6}\\
&|\int_{\Bbb R^3}m_t(\xi,\eta,\zeta))\widehat{ {f}_1}(\xi)\widehat{ {f}_2}(\eta-\xi)\widehat{{\partial_xf_3}}(\zeta-\eta)\widehat{\partial_xf_4}(\sigma-\zeta)d\xi d\eta d\zeta|\nonumber\\
&\lesssim t^{-2\mu ( \ell_1+\frac{1}{2}(\ell_2+\ell_3+\ell_4))} t^{\frac{3}{2}\mu} \|f_1\|_{W^{\lambda,\infty}_x}\|f_2\|_{W^{\lambda,\infty}_x }\|f_3\|_{H^{1}_x}\| f_4\|_{H^{\rho}_x}.\label{v7}
\end{align}
\end{Lemma}
\begin{proof}
First, we verify (\ref{v1}). Recall the function $\varrho$ in the definition of Littlewood-Paley decomposition at the end of Section 1.
Define
\begin{align*}
\varphi_0(\xi)&=\sum_{j\le 0} \varrho (2^{-j}\xi)\\
\varphi_k(\xi)&= \varrho (2^{-k}\xi), \mbox{  }k\ge 1.
\end{align*}
Decompose $f_2,f_3,f_4$ into
\begin{align*}
&f_2=\sum_{j\ge 0}\mathcal{F}^{-1}\varphi_j( t^{\mu}\xi)\mathcal{F}f_2;\mbox{  }f_3=\sum_{k \ge 0}\mathcal{F}^{-1}\varphi_{k}( t^{\mu} \xi  )\mathcal{F}f_3;\mbox{  }
f_4=\sum_{l\ge 0}\mathcal{F}^{-1}\varphi_{l}(  t^{\mu} {\xi} )\mathcal{F}f_4.
\end{align*}
Let $\widetilde{\chi}$ be a $C^{\infty}_c$ cutoff function which equals one in the support of $\varrho$. Define
 \begin{align*}
\widetilde{\chi}_k(\xi)&=\widetilde{\chi} (2^{-k}\xi), \mbox{  }k\ge 1.
\end{align*}
Let   $\widetilde{\chi}_0$ be a $C^{\infty}_c$ cutoff function which equals one in the support of $\varphi_0$.

Consider the integral
\begin{align*}
I_{j,k,l}:= \int_{\Bbb R^2}&m_t(\xi,\eta,\zeta)\varphi_j(\frac{\eta-\xi}{t^{-\mu} })\varphi_k(\frac{\zeta-\eta}{t^{-\mu} })\varphi_l(\frac{\sigma-\zeta}{t^{-\mu} })
\widetilde{\chi}_j(t^{\mu} (\eta-\xi))\widetilde{\chi}_k( t^{\mu} (\zeta-\eta))
 \widetilde{\chi}_l( t^{\mu}(\sigma-\zeta))\\
& \cdot\widehat{ {f}_1}(\xi)\widehat{  {f}_2}(\eta-\xi)\widehat{{\partial_x  f_3}} (\zeta-\eta)\widehat{\partial_xf_4}(\sigma-\zeta)d\xi d\eta d\zeta.
\end{align*}
By scaling, we find the symbol
\begin{align*}
\widetilde{m}^{j,k,l}_{t}:=m_t(\xi,\eta,\zeta)\varphi_j(\frac{\eta-\xi}{t^{-\mu} })\varphi_k(\frac{\zeta-\eta}{t^{-\mu} })\varphi_l(\frac{\sigma-\zeta}{t^{-\mu} })
\end{align*}
satisfies the bound
\begin{align*}
\|\mathcal{F}_{\xi,\eta,\zeta}\widetilde{m}^{j,k,l}_{t}\|_{L^{1}(\Bbb R^3)}\lesssim t^{-\sigma ( \ell_1+\frac{1}{2}(\ell_2+\ell_3+\ell_4))} \|\mathcal{F}_{\xi,\eta,\zeta}\mathcal{M}^{j,k,l}_{t}\|_{L^{1}(\Bbb R^3)},
\end{align*}
where $\mathcal{M}^{j,k,l}_{t}$ is defined by
\begin{align*}
\mathcal{M}^{j,k,l}_{t}&:=g_1( \phi)\vartheta_1(\phi)g_2( \partial_{\xi}\phi)\vartheta_2(\partial_{\xi}\phi) g_3( \partial_{\eta}\phi)\vartheta_3(\partial_{\eta}\phi) g_4( \partial_{\zeta}\phi)\vartheta_4(\partial_{\zeta}\phi)\\
&\cdot\varphi_j({\eta-\xi})\varphi_k({\zeta-\eta})\varphi_l({\sigma-\zeta}).
\end{align*}
By H\"older inequality and Plancherel identity, we have
\begin{align*}
\|\mathcal{F}_{\xi,\eta,\zeta}\mathcal{M}^{j,k,l}_{t}\|_{L^{1}(\Bbb R^3)}\lesssim \| \mathcal{M}^{j,k,l}_{t}\|_{H^{2}(\Bbb R^3)}\lesssim 2^{\frac{j}{2}}2^{\frac{k}{2}}2^{\frac{l}{2}}.
 \end{align*}
Then by Lemma \ref{457}  and Bernstein inequality,
\begin{align*}
\sum_{j,k,l\ge 0} |I_{j,k,l}|&\lesssim
t^{-2\mu ( \ell_1+\frac{1}{2}(\ell_2+\ell_3+\ell_4))}\|f_1\|_{L^{\infty}_x}\sum_{k,l,j\ge 0} 2^{\frac{j}{2}}
\|\widetilde{\chi}_j(t^{\mu}D)f_2\|_{L^{\infty}_x} 2^{\frac{k}{2}}\|\widetilde{\chi}_k(t^{\mu}D)\partial_xf_3\|_{L^2_x}\\
&\cdot 2^{\frac{l}{2}}\|\widetilde{\chi}_l(t^{\mu}D)\partial_x f_4\|_{L^2_x}\\
&\lesssim t^{-2\mu( \ell_1+\frac{1}{2}(\ell_2+\ell_3+\ell_4))}t^{\frac{3}{2}\mu} \|f_1\|_{L^{\infty}_x} \|f_2\|_{W^{\lambda,\infty}_x} \|f_3\|_{H^{\rho}_x}\|f_4\|_{H^{\rho}_x}.
\end{align*}

Second, for (\ref{v2}),  by Lemma \ref{457}  and Bernstein inequality, we have
\begin{align*}
\sum_{j,k,l\ge 0} |I_{j,k,l}|&\lesssim  t^{-2\mu ( \ell_1+\frac{1}{2}(\ell_2+\ell_3+\ell_4))}\|f_1\|_{L^{\infty}_x}\sum_{k,l,j\ge 0}2^{\frac{1}{2}j} \|\widetilde{\chi}_j(t^{\mu} D)f_2\|_{L^{\infty}_x}
2^{\frac{1}{2}k} \|\widetilde{\chi}_k(t^{\mu}D)\partial_x f_3\|_{L^2_x}\\
&\cdot 2^{\frac{1}{2}l}\|\widetilde{\chi}_l(t^{\mu}D)f_4\|_{L^{2}_x}\\
&\lesssim t^{-2\mu( \ell_1+\frac{1}{2}(\ell_2+\ell_3+\ell_4))}t^{\frac{3}{2}\mu} \|f_1\|_{L^{\infty}_x} \|f_2\|_{W^{\lambda,\infty}_x} \|f_3\|_{H^{\rho}_x}\|f_4\|_{H^{\lambda}_x}.
\end{align*}

Third, for (\ref{v6}), by Lemma \ref{457}  and Bernstein inequality, we get
\begin{align*}
\sum_{j,k,l\ge 0} |I_{j,k,l}|&\lesssim  t^{-2\mu( \ell_1+\frac{1}{2}(\ell_2+\ell_3+\ell_4))}\|f_1\|_{L^{2}_x}  \sum_{j,k,l\ge 0}2^{\frac{j}{2}}\|\widetilde{\chi}_j(t^{\mu} D)f_2\|_{L^{\infty}_x}
  2^{\frac{k}{2}}\|\widetilde{\chi}_k(t^{\mu}D)\partial_x f_3\|_{L^{2}_x}\\
&\cdot 2^{\frac{l}{2}}\|\widetilde{\chi}_l(t^{\mu} D)\partial_x f_4\|_{L^{\infty}_x}\\
&\lesssim t^{-2\mu ( \ell_1+\frac{1}{2}(\ell_2+\ell_3+\ell_4))}t^{\frac{3}{2}\mu} \|f_1\|_{L^{2}_x}  \|f\|_{W^{\lambda,\infty} }\| f_3\|_{H^{\rho}_x}
\|f_4\|_{W^{\rho,\infty}_x}.
\end{align*}
Lastly, for (\ref{v7}),  decompose $f_1,f_2, f_4$ into
\begin{align*}
&f_1=\sum_{k\ge 0}\mathcal{F}^{-1}\varphi_k( {\xi}{t^{ \mu} })\mathcal{F}f_1;\mbox{ }\mbox{ } f_2=\sum_{j\ge  0}\mathcal{F}^{-1}\varphi_j( {\xi}{t^{ \mu}})\mathcal{F}f_2;\mbox{ }\mbox{ }
 f_4=\sum_{l\ge 0}\mathcal{F}^{-1}\varphi_l( {\xi}{t^{ \mu}})\mathcal{F}f_4.
\end{align*}
Consider the integral
\begin{align*}
\tilde{I}_{j,k,l}:= \int_{\Bbb R^2}&m_t(\xi,\eta,\zeta)\varphi_k( t^{\mu} {\xi}  )\varphi_l(t^{\mu} (\sigma-\zeta)   )\varphi_j(t^{\mu} (\eta-\xi) )
\widetilde{\chi}_j(t^{\mu} (\eta-\xi))\widetilde{\chi}_l( t^{\mu} (\sigma-\zeta) )\widetilde{\chi}_k( t^{\mu} \xi )\\
& \cdot  \widehat{ {f}_1}(\xi)\widehat{  {f}_2}(\eta-\xi)\widehat{{\partial_x  f_3}} (\zeta-\eta)\widehat{\partial_xf_4}(\sigma-\zeta)d\xi d\eta d\zeta.
\end{align*}
By scaling, we find the symbol
\begin{align*}
\widetilde{\tilde{m}}^{j,k,l}_{t}:=m_t(\xi,\eta,\zeta)\varphi_j( ({\eta-\xi}){t^{ \mu} })\varphi_l( ({\sigma-\zeta}){t^{ \mu} })\varphi_k( {\xi}{t^{\mu} })
\end{align*}
satisfies the bound
\begin{align*}
\|\mathcal{F}_{\xi,\eta,\zeta}\widetilde{\tilde{m}}^{j,k,l}_{t}\|_{L^{1}(\Bbb R^3)}\lesssim t^{-\sigma ( \ell_1+\frac{1}{2}(\ell_2+\ell_3+\ell_4))} 2^{\frac{j}{2}} 2^{\frac{k}{2}} 2^{\frac{l}{2}}.
\end{align*}
Then by Lemma \ref{457}  and Bernstein inequality,
\begin{align*}
\sum_{j,k,l\ge 0 } |\tilde{I}_{j,k,l}|&\lesssim  t^{-2\mu ( \ell_1+\frac{1}{2}(\ell_2+\ell_3+\ell_4))}\sum_{k,l,j\ge 0}2^{\frac{k}{2}} \|\widetilde{\chi}_k(t^{\mu} D)f_1\|_{L^{\infty}_x} 2^{\frac{j}{2}} \|\widetilde{\chi}_j(t^{\mu} D)f_2\|_{L^{\infty}_x} \|\partial_x f_3\|_{L^2_x}\\
&\cdot 2^{\frac{l}{2}} \|\widetilde{\chi}_l(t^{\mu} D)\partial_x f_4\|_{L^2_x}\\
&\lesssim t^{-2\mu( \ell_1+\frac{1}{2}(\ell_2+\ell_3+\ell_4))}t^{\frac{3}{2}\mu} \|f_1\|_{W^{\lambda,\infty}_x } \|f_2\|_{W^{\lambda,\infty}_x } \|f_3\|_{H^{1}_x}\|f_4\|_{H^{\rho}_x}.
\end{align*}
\end{proof}

The following simple lemma  will also be useful in  nonlinear estimates of later subsections.
\begin{Lemma}\label{4p}
We have
\begin{align*}
 \| \mathcal{F}_x  [e^{-is\Delta}  g(s)]  \|_{L^{ \infty}_{\sigma}} &\lesssim   \|g(s)\|_{L^1_x}  \\
 \| \mathcal{F}_x  [e^{-is\Delta} g(s)]   \|_{L^{2}_{\sigma}} &\lesssim \| g(s)\|_{L^2_x}  .
\end{align*}
\end{Lemma}

\subsection{The Proof of Theorem \ref{XS3}}

In this section, we assume $Q$ is not an intrinsic vanishing point, i.e.
\begin{align}\label{dKEY}
[\ln   {h}]_{z}(0)[\ln {h}]_{z\bar{z}}(0)-[\ln  {h}]_{z\bar{z}z}(0)\neq 0.
\end{align}

We prove Theorem 1.2 by contradiction. Assume that $u$ is a global solution to 1D SMF satisfying (\ref{suption2py})-(\ref{suption2pp}) under the local complex coordinate  $z$. Recall that $z$ satisfies
\begin{align}\label{Gifinal}
\begin{cases}
i\partial_t z+\Delta z=&c_0 \partial_xz\partial_xz+c_1 \bar{z}\partial_xz\partial_xz+c_2 {z}\partial_xz\partial_xz+  c_3 {z}^2\partial_xz\partial_xz+  c_4 {z}\bar{z}\partial_xz\partial_xz\\
&+   c_5 {\bar{z}}^2\partial_xz\partial_xz+O(|z|^3)\partial_xz\partial_xz\\
z\upharpoonright_{t=0}  = z_0.&
\end{cases}
\end{align}
Define  $f=e^{-it\Delta}z$.
Then $f$ fulfills
\begin{align}
&\widehat{f}(t,\sigma)=\widehat{f}(\upharpoonright_{t=1})- \frac{ ic_1}{2\pi}\int^{t}_1\int_{\Bbb R^2}  e^{i\tau\phi_0}(\eta-\xi)(\zeta-\eta)\widehat{\overline{f}}(\xi)\widehat{f}(\eta-\xi)\widehat{f} (\zeta-\eta)d\xi d\eta d\tau+\int^t_1\mathcal{R}d\tau \nonumber\\
&-\frac{ ic_0}{2\pi}\int^{t}_1\int_{\Bbb R}  e^{i\tau\phi_1}(\eta-\xi)(\zeta-\eta)\widehat{ {f}}(\xi)\widehat{f}(\eta-\xi) d\xi d\tau\label{a789}\\
&-\frac{ic_2}{2\pi}\int^{t}_1\int_{\Bbb R^2} e^{i\tau\phi_2}(\eta-\xi)(\zeta-\eta)\widehat{ {f}}(\xi)\widehat{f}(\eta-\xi)\widehat{f}(\zeta-\eta) d\xi d\eta d\tau  \label{b789}\\
&- \sum_{\imath_1,\imath_2\in\{+ ,- \}}\frac{i}{2\pi}\int^{t}_1\int_{\Bbb R^3}c_{\imath_1,\imath_2}e^{i\tau\phi_{\imath_1,\imath_2}}(\zeta-\eta)(\sigma-\zeta)\widehat{{f}^{\imath_1}}(\xi)\widehat{f^{\imath_2}}(\eta-\xi)\widehat{f}(\zeta-\eta)\widehat{f} (\sigma-\zeta)d\xi d\eta d\zeta d\tau,\label{ccJ3}
\end{align}
where we denote
\begin{align*}
\phi_0&=\sigma^2+\xi^2-(\eta-\xi)^2-(\sigma-\eta)^2 \\
\phi_1&=\sigma^2-\xi^2-(\sigma-\xi)^2 \\
\phi_2&=\sigma^2-\xi^2-(\eta-\xi)^2-(\sigma-\eta)^2 \\
\phi_{\imath_1,\imath_2}&=\sigma^2+\imath_1\xi^2+\imath_2(\xi-\eta)^2-(\zeta-\eta)^2-(\sigma-\zeta)^2, \mbox{  }\imath_1,\imath_2\in \{+,-\}\\
\mathcal{R}&=\mathcal{F}[e^{-i\tau \Delta} O(|z|^3) (\partial_x z)^2]\\
f^{\imath}&=f, \mbox{ }{\rm if} \mbox{ }\imath=+; \mbox{ }f^{\imath}=\overline{f}, \mbox{ }{\rm if} \mbox{ }\imath=-.
\end{align*}
Note that $z$ is now a given coordinate, we can not expect any vanishing of $c_0,c_2,c_3$. Moreover, we can not perform the holomorphic transformation $w:=z+\gamma_1z^2+\gamma_2 z^3+\gamma_3z^4$, since the condition  (\ref{suption2p}) breaks down for the new coordinate $w$ if $\gamma_1\neq 0$.
Therefore, compared with (\ref{sJ3}), (\ref{ccJ3}) has additional quadratic term (\ref{a789}), cubic term (\ref{b789}) and 4 order term $\phi_{--}$.

\begin{Lemma}\label{XY}
Under the assumptions (\ref{suption2p}), (\ref{suption2py}), (\ref{suption2pp}), we have
\begin{align}
 \| z(t) \|_{H^{1}_{x}}&\lesssim  \epsilon_*\nonumber\\
 \| z(t) \|_{H^{2}_{x}}&\lesssim  \epsilon_*\langle t\rangle^{\epsilon_*}\nonumber\\
 \|f(t)\|_{H^{2,1}_x}&\lesssim  \epsilon_*\langle t\rangle^{\beta}\nonumber\\
 \|z(t)\|_{L^2_x}\lesssim&\|z_0\|_{L^2_x}\lesssim \|z(t)\|_{L^2_x}.\label{RR5}
\end{align}
\end{Lemma}
\begin{proof}
The $H^1$ and $H^2$ bounds follow by Section 2. The $\|\partial^j_xf(t)\|_{H^{0,1}_x}$  bound with $j=0,1,2$ follows by the assumption $\|L(\partial^j_x z)\|_{L^2_x} \lesssim \epsilon_*\langle t\rangle^{\beta}$.
It reamins to prove the  last inequality on $\|z(t)\|_{L^2_x}$. Let $w=z+\gamma_1 z^2+\gamma_2 z^3+\gamma_3 z^4$, it  follows by Section 4 and assumption (1.9) that
\begin{align*}
\left| \frac{d}{dt}\|w(t) \|^2_{L^{2}_{x}}\right|&\lesssim  \epsilon^3_*\langle  t\rangle^{-\frac{3}{2}}\| w(t) \|^2_{L^{2}_{x}}.
\end{align*}
Thus, we get
\begin{align*}
 \|w(t)\|_{L^2_x}\lesssim&\|w_0\|_{L^2_x}\lesssim \|w(t)\|_{L^2_x}.
\end{align*}
Notice that by the assumption (1.9),
\begin{align*}
 \|w(t)\|_{L^2_x}\lesssim &\|z(t)\|_{L^2_x}\lesssim \|w(t)\|_{L^2_x},\mbox{ }\forall t\ge 0.
\end{align*}
So, (\ref{RR5}) follows.
\end{proof}

In the following, we call (\ref{RR5})  almost conservation of mass.

The proof of Theorem 1.2 will be divided into two cases according to whether $c_5$ in (\ref{Gifinal}) vanishes or not.
\begin{Proposition}\label{Gbbb}
Suppose that $c_5\neq  0$ in (\ref{Gifinal}).
Under the assumptions (\ref{dKEY}), (\ref{suption2py}), (\ref{suption2p}), (\ref{suption2pp}), and $U\neq 0$, we have  as $t\to \infty$
\begin{align*}
 \| \widehat{z}(t) \|_{L^{\infty}_{\sigma}} \ge C\ln (t),
\end{align*}
for some $C>0$.
\end{Proposition}
\begin{proof}
Let $\mu>0,\nu\in (0,\frac{1}{4})$ satisfy
\begin{align*}
-2\nu+2\mu\nu+4\beta&<0\\
-1+4\mu+\beta&<0\\
\nu>3\beta,\mbox{ }\mbox{ }\mu&>\beta>2\epsilon_*.
\end{align*}
Recall that $\widehat{f}  $ satisfies  (\ref{ccJ3}).

\noindent \underline{{\it Estimates of the leading cubic term.} } This part is the same as Section 8.
 Let
\begin{align}\label{xhh1}
E(\sigma,\eta,\zeta):=(\sigma-\zeta)(\sigma-\eta)\overline{\widehat{{f}}(\sigma-\eta-\zeta)}\widehat{f}(\sigma-\eta)\widehat{f} (\sigma-\zeta).
\end{align}
Defining $\widehat{F}(t,\sigma)$ by
\begin{align}\label{xPQ9}
\widehat{F}(t,\sigma)&=e^{ci\int^{t}_1\frac{1}{2\tau}\sigma^2|\widehat{f}(\tau,\sigma)|^2d\tau} \widehat{f}(t,\sigma),
\end{align}
we get
\begin{align}
&\widehat{F}(t,\sigma)=\widehat{f}(1,\sigma)\nonumber\\
&+e^{ci\int^{t}_1\frac{1}{2\tau}\sigma^2|\widehat{f}(\tau,\sigma)|^2d\tau}\left[ -\frac{i}{2\pi}\int^t_1{\bf{R}}d\tau
 - \frac{ic_1}{2\pi}\int^t_1\int_{\Bbb R^2}\frac{1}{2\tau} [ e^{-\frac{i}{2\tau}\widetilde{\eta}\widetilde{\zeta}}-1]\mathcal{F}^{-1}_{\eta,\zeta}[E(\sigma,\eta,\zeta)]  d\widetilde{\eta} d\widetilde{\zeta} d\tau\right].\label{xY7m}
\end{align}
Here, we denote
$${\bf R}:={\bf R}_{1} +{\bf R}_{2} +{\bf R}_{--}+{\bf R}_{++}+{\bf R}_{+-}+\mathcal{R},$$
 where ${\bf R}_1$, ${\bf R}_2$ denote  the terms associated  with (\ref{a789}) and  (\ref{b789}) respectively, ${\bf R}_{++}$, ${\bf R}_{--}$, ${\bf R}_{+-}$ denote 4 order terms associated with $\phi_{++}$, $\phi_{--}$ and  ${\phi}_{+-}$ respectively,  and $\mathcal{R}$ denotes higher order terms. (See (\ref{ccJ3}).)

Hence, for $\nu<\frac{1}{4}$, as Section 8.2 one has
\begin{align*}
&\int^t_1\tau^{-1-\nu}d\tau   \int_{\Bbb R^2}|\widetilde{\eta}\widetilde{\zeta}|^{\nu}| \mathcal{F}^{-1}_{\eta,\zeta}[E(\sigma,\eta,\zeta)]|  d\widetilde{\eta} d\widetilde{\zeta}
\lesssim \int^t_1\tau^{-1-\nu} \|f\|^3_{H^{1,1}_x}d\tau.
\end{align*}
Thus the last term in the RHS of (\ref{xY7m}) contributes to $|\widehat{F}(t,\sigma)|=| \widehat{f}(t,\sigma)|$ by $C\epsilon^3$.

\underline{{\it Estimates  of 4 order terms.}}

{\it Space-time resonance analysis of $\phi_{++}$. }
The time resonance set of $\phi_{++}$ is
\begin{align*}
\mathcal{R}^{++}_{t}=\{(\xi,\eta,\zeta,\sigma):\phi_{++}=0\}.
\end{align*}
The space resonance sets are
\begin{align*}
\mathcal{R}^{++}_{s,\xi}&=\{(\xi,\eta,\zeta,\sigma):\partial_{\xi}\phi_{+,+}=0\}=\{( \xi,\eta,\zeta,\sigma):2\xi-\eta=0\}\\
\mathcal{R}^{++}_{s,\eta}&=\{(\xi,\eta,\zeta,\sigma):\partial_{\eta}\phi_{+,+}=0\}=\{(  \xi,\eta,\zeta,\sigma): \xi-\zeta=0\}\\
\mathcal{R}^{++}_{s,\zeta}&=\{(\xi,\eta,\zeta,\sigma):\partial_{\zeta}\phi_{+,+}=0\}=\{(  \xi,\eta,\zeta,\sigma): 2\zeta-\sigma-\eta=0\}.
\end{align*}
The space-time resonance set is
\begin{align*}
\mathcal{R}^{++}_{s,t}&=\{(\xi,\eta,\zeta,\sigma): \eta=2\xi,\zeta= \xi, \sigma=0\}.
\end{align*}

\noindent\underline{{\it Estimates  of 4 order term associated with $\phi_{++}$.}}
Letting
\begin{align*}
\xi-\frac{\eta}{2}=\xi'; \mbox{ } \eta-2\zeta=2\eta',
\end{align*}
we find
\begin{align}
&\frac{i}{2\pi}\int^{t}_1 \int_{\Bbb R^3} e^{i\tau\phi_{+,+}}(\zeta-\eta)(\sigma-\zeta)\widehat{ \overline{{f}}}(\xi)\widehat{ \overline{{f}}}(\eta-\xi)\widehat{f}(\zeta-\eta)\widehat{f} (\sigma-\zeta)d\xi d\eta d\zeta d\tau \nonumber\\
&=\frac{i}{ \pi}\int^{t}_1  \int_{\Bbb R^3} e^{i\tau(2{\xi'}^2-2{\eta'}^2+2\zeta\sigma)}(-\zeta-2\eta')(\sigma-\zeta)\widehat{ \overline{{f}}}(\xi'+\zeta+ \eta')\nonumber\\
&\cdot\widehat{ \overline{{f}}}(\zeta+ \eta'-\xi')\widehat{f}(-\zeta-2\eta')\widehat{f} (\sigma-\zeta)d\xi' d\eta' d\zeta d\tau.\label{0hj}
\end{align}
Let $\chi$ be a smooth cutoff function adapted to the  interval $[-1,1]$. And define $\chi_{\mu}(\cdot)=\chi(\tau^{\mu}\cdot)$, $\gamma_{\mu}=1-\chi_{\mu}$.
We decompose (\ref{0hj}) as the following
\begin{align}
&\frac{i}{ \pi}\int^{t}_1  \int_{\Bbb R^3} e^{i\tau(2{\xi'}^2-2{\eta'}^2+2\zeta\sigma)}\chi_{\mu}(\xi')\chi_{\mu}(\eta')(-\zeta-2\eta')(\sigma-\zeta)\widehat{ \overline{{f}}}(\xi'+\zeta+ \eta')\widehat{ \overline{{f}}}(\zeta+ \eta'-\xi')\nonumber\\
&\cdot \widehat{f}(-\zeta-2\eta')\widehat{f} (\sigma-\zeta)d\xi' d\eta' d\zeta d\tau\label{ahj}\\
&+\frac{i}{ \pi}\int^{t}_1  \int_{\Bbb R^3} e^{i\tau(2{\xi'}^2-2{\eta'}^2+2\zeta\sigma)}\chi_{\mu}(\xi')\gamma_{\mu}(\eta')(-\zeta-2\eta')(\sigma-\zeta)\widehat{ \overline{{f}}}(\xi'+\zeta+ \eta')\widehat{ \overline{{f}}}(\zeta+ \eta'-\xi')\nonumber\\
&\cdot\widehat{f}(-\zeta-2\eta')\widehat{f} (\sigma-\zeta)d\xi' d\eta' d\zeta d\tau\label{bhj}\\
&+\frac{i}{ \pi}\int^{t}_1  \int_{\Bbb R^3} e^{i\tau(2{\xi'}^2-2{\eta'}^2+2\zeta\sigma)}\gamma_{\mu}(\xi')\chi_{\mu}(\eta')(-\zeta-2\eta')(\sigma-\zeta)\widehat{ \overline{{f}}}(\xi'+\zeta+ \eta')\widehat{ \overline{{f}}}(\zeta+ \eta'-\xi')\nonumber\\
&\cdot\widehat{f}(-\zeta-2\eta')\widehat{f} (\sigma-\zeta)d\xi' d\eta' d\zeta d\tau\label{chj}\\
&+\frac{i}{ \pi}\int^{t}_1  \int_{\Bbb R^3} e^{i\tau(2{\xi'}^2-2{\eta'}^2+2\zeta\sigma)}\gamma_{\mu}(\xi')\gamma_{\mu}(\eta')(-\zeta-2\eta')(\sigma-\zeta)\widehat{ \overline{{f}}}(\xi'+\zeta+ \eta')\widehat{ \overline{{f}}}(\zeta+ \eta'-\xi')\nonumber\\
&\cdot\widehat{f}(-\zeta-2\eta')\widehat{f} (\sigma-\zeta)d\xi' d\eta' d\zeta d\tau\label{dhj}.
\end{align}

{\it Estimates of (\ref{ahj}).}
To dominate (\ref{ahj}), we further decompose it as follows,
\begin{align}
&\frac{i}{ \pi}\int^{t}_1 \int_{\Bbb R^3}  e^{i\tau(2{\xi'}^2-2{\eta'}^2+2\zeta\sigma)} \chi_{\mu}(\sigma)\chi_{\mu}(\xi')\chi_{\mu}(\eta')(-\zeta-2\eta')(\sigma-\zeta)\widehat{ \overline{{f}}}(\xi'+\zeta+ \eta')\widehat{ \overline{{f}}}(\zeta+ \eta'-\xi')\nonumber\\
&\cdot\widehat{f}(-\zeta-2\eta')\widehat{f} (\sigma-\zeta)d\xi' d\eta' d\zeta d\tau \label{2ahj}\\
+&\frac{i}{ \pi}\int^{t}_1  \int_{\Bbb R^3} e^{i\tau(2{\xi'}^2-2{\eta'}^2+2\zeta\sigma)} \gamma_{\mu}(\sigma)\chi_{\mu}(\xi')\chi_{\mu}(\eta')(-\zeta-2\eta')(\sigma-\zeta)\widehat{ \overline{{f}}}(\xi'+\zeta+ \eta')\widehat{ \overline{{f}}}(\zeta+ \eta'-\xi')\nonumber\\
&\cdot\widehat{f}(-\zeta-2\eta')\widehat{f} (\sigma-\zeta)d\xi' d\eta' d\zeta d\tau \label{3ahj}.
\end{align}
Note that $|\sigma|\ge C\tau^{-\mu}$ in (\ref{3ahj}). And thus using the fact
\begin{align}
e^{i\tau(2{\xi'}^2-2{\eta'}^2+2\zeta\sigma)}=\frac{1}{2i\tau\sigma} \partial_{\zeta}e^{i\tau(2{\xi'}^2-2{\eta'}^2+2\zeta\sigma)}, \mbox{ }|2i\tau\sigma|\ge C |\tau|^{1-\mu},
\end{align}
and integration by parts in $\zeta$, one obtains
\begin{align}
\|(\ref{3ahj})\|_{L^{\infty}_{\sigma}}\lesssim \int^{t}_1 {\tau}^{-\mu-1}\|\widehat{f}\|^3_{L^{\infty}}\|\widehat{f}\|_{H^{1,1}} d\tau\lesssim 1.
\end{align}
 For (\ref{2ahj}), by change of variables and Plancherel identity, we have
\begin{align*}
 (\ref{2ahj})&=\frac{i}{ \pi}\int^{t}_1 \int_{\Bbb R^3} \mathcal{F}_{\xi',\eta'} [e^{i\tau(2{\xi'}^2-2{\eta'}^2)} ]\mathcal{F}^{-1}_{\xi',\eta'}[...] d\widetilde{\xi} d\widetilde{\eta}   d\zeta d\tau \\
&=\frac{i}{{(2\pi)}^{\frac{1}{2}}}\int^{t}_1 \int_{\Bbb R^3} \frac{1}{2\tau}  e^{\frac{\widetilde{\xi}^2-\widetilde{\eta}^2}{8it}} Y_{\sigma,\tau}(\widetilde{\xi},\widetilde{\eta},\zeta) d\widetilde{\xi} d{\eta}d\zeta d\tau,
\end{align*}
where $Y$ is defined by
\begin{align}
&Y_{\sigma,\tau}(\widetilde{\xi} ,\widetilde{\eta} ,\zeta)\nonumber\\
&=\mathcal{F}^{-1}_{\xi',\eta' }[e^{2i\tau \sigma\zeta} \chi_{\mu}(\xi')\chi_{\mu}(\eta')\chi_{\mu}(\sigma)\widehat{ \overline{{f}}}(\xi'+\zeta+ \eta')\widehat{ \overline{{f}}}(\zeta+ \eta'-\xi')\widehat{\partial_xf}(-\zeta-2\eta')\widehat{\partial_ x f} (\sigma-\zeta)].\label{hh2}
\end{align}
Further splitting  $e^{\frac{\widetilde{\xi}^2-\widetilde{\eta}^2}{8it}}$  yields
\begin{align}
&\frac{i}{2\pi}\int^{t}_1   \int_{\Bbb R^3}  e^{i\tau\phi_{++}}(\zeta-\eta)(\sigma-\zeta)\widehat{ \overline{{f}}}(\xi)\widehat{ \overline{{f}}}(\eta-\xi)\widehat{f}(\zeta-\eta)\widehat{f} (\sigma-\zeta)d\xi d\eta d\zeta d\tau \nonumber\\
&=\frac{i}{  2\pi }\int^{t}_1  \frac{1}{2\tau}  \int_{\Bbb R^3} [e^{\frac{\widetilde{\xi}^2-\widetilde{\eta}^2}{8i\tau}}-1] Y_{\sigma,\tau}(\widetilde{\xi},\widetilde{\eta},\zeta) d\widetilde{\xi} d{\widetilde{\eta}}d\zeta d\tau
 +\frac{i}{ 2\pi }\int^{t}_1   \frac{1}{2\tau}  \int_{\Bbb R^3} Y_{\sigma,\tau}(\widetilde{\xi},\widetilde{\eta},\zeta) d\widetilde{\xi} d{\widetilde{\eta}}d\zeta d\tau \nonumber\\
&=\frac{i}{ 2\pi }\int^{t}_1   \frac{1}{2\tau}  \int_{\Bbb R^3} [e^{\frac{\widetilde{\xi}^2-\widetilde{\eta}^2}{8i\tau}}-1] Y_{\sigma,\tau}(\widetilde{\xi},\widetilde{\eta},\zeta) d\widetilde{\xi} {\widetilde{\eta}} d\zeta d\tau\label{po1}\\
&+\frac{i}{{(2\pi)}^{\frac{1}{2}}}\int^{t}_1\frac{1}{2\tau}\int_{\Bbb R} e^{2i\tau \sigma\zeta} \widehat{ \overline{{f}}}( \zeta )\widehat{ \overline{{f}}}(\zeta  )\widehat{\partial_xf}(-\zeta )\widehat{\partial_ x f} (\sigma-\zeta) \chi_{\mu}(\sigma)d\zeta d\tau.\label{po2}
\end{align}
We pick up the main leading part of (\ref{po2}) by
\begin{align}
(\ref{po2})&=\frac{i}{{(2\pi)}^{\frac{1}{2}}}\chi_{\mu}(\sigma)\int^{t}_1\frac{1}{2\tau}\int_{\Bbb R} e^{2i\tau \sigma\zeta} \widehat{ \overline{{f}}}( \zeta )\widehat{ \overline{{f}}}(\zeta  )\widehat{\partial_xf}(-\zeta )\widehat{\partial_ x f} (-\zeta) d\zeta d\tau\label{po2a}\\
&+\frac{i}{{(2\pi)}^{\frac{1}{2}}}\int^{t}_1\frac{1}{2\tau}\int_{\Bbb R} e^{2i\tau \sigma\zeta} \widehat{ \overline{{f}}}( \zeta )\widehat{ \overline{{f}}}(\zeta  )\widehat{\partial_xf}(-\zeta )\chi_{\mu}(\sigma)[\int^{1}_0\sigma\partial_{\zeta}\widehat{\partial_ x f} (\theta\sigma-\zeta) d\theta] d\zeta d\tau. \label{po2b}
\end{align}
 (\ref{po2b}) is dominated by
\begin{align*}
&\| (\ref{po2b})\|_{L^{\infty}_{\sigma}}\lesssim \int^{t}_1 {\tau}^{-1-\mu}\|\partial_{\zeta}\widehat{f}\|_{L^2_x}\|\widehat{f}\|_{L^2_x}\|\widehat{f}\|^2_{L^{\infty}_x}d\tau \lesssim 1.
\end{align*}
Letting $\sigma=0$, we observe that  (\ref{po2a}) has a lower bound:
\begin{align*}
\|(\ref{po2a})\|_{L^{\infty}_{\sigma}}\ge \frac{1}{{(2\pi)}^{\frac{1}{2}}} \int^{t}_1\frac{1}{2\tau}\int_{\Bbb R} \zeta^2 | \widehat{  {{f}}}|^4 d\zeta d\tau.
\end{align*}
By the assumption   (\ref{suption2pp}),  we see as $t\to \infty$
\begin{align*}
\int_{\Bbb R} \zeta^2 | \widehat{ f}|^4 \sim \int_{\Bbb R} \zeta^2|U(\zeta)|^4d\zeta.
\end{align*}
Therefore, we conclude for (\ref{po2}) that
\begin{align}
\|(\ref{po2})\|_{L^{\infty}_{\zeta}}\ge C |c_5|\|\zeta^{\frac{1}{2}}|U(\zeta)| \|^4_{L^4} \ln (t)
\end{align}
for $t$ large and some $C> 0$.

Moreover, by computation we have
\begin{align*}
Y_{\sigma,\tau,\zeta}(\widetilde{\xi},\widetilde{\eta})&=[\Theta_{\sigma,\tau,\zeta}*\Upsilon_{\sigma,\tau}](\widetilde{\xi},\widetilde{\eta})\\
\Theta_{\sigma,\tau,\zeta}(\widetilde{\xi},\widetilde{\eta})&:=ce^{2i\tau\zeta\sigma} \int_{\Bbb R} e^{-iy\zeta-\frac{i}{2}\zeta\widetilde{\xi}-\frac{i}{2}\widetilde{\eta}\zeta} \overline{{f}}(y) \overline{{f}}(y-\widetilde{\xi})\partial_xf(y-\frac{1}{2}\widetilde{\eta}-\frac{1}{2}\widetilde{\xi})dy\widehat{\partial_xf}(\sigma-\zeta)\\
\Upsilon_{\sigma,\tau}(\widetilde{\xi},\widetilde{\eta})&:=\frac{1}{\tau^{2\mu}}\check{\chi}(\frac{\widetilde{\xi}}{\tau^{\mu}})\check{\chi}(\frac{\widetilde{\eta}}{\tau^{\mu}}),
\end{align*}
where $\check{\chi}$ denotes $\mathcal{F}^{-1}\chi$.
Hence, for $\nu\in (0,\frac{1}{4})$, using $|e^{is}-1|\lesssim |s|^{\nu}$,  (\ref{po1}) is dominated by
\begin{align*}
  \|(\ref{po1})\|_{L^{\infty}_{\sigma}}&\lesssim \int^{t}_1 \| [(\frac{ \widetilde{\xi}^2}{\tau})^{\nu}+(\frac{ \widetilde{\eta}^2}{\tau})^{\nu} ] \frac{1}{\tau^{2\mu}}\check{\chi}(\frac{\widetilde{\xi}}{\tau^{\mu}})  \check{\chi}(\frac{\widetilde{\eta}}{\tau^{\mu}})\|_{L^1_{\widetilde{\xi},\widetilde{\eta}}} {\tau }^{-\nu-1} \|\langle \xi  \rangle^{2\nu} f(\xi)\|^4_{L^1_{\xi}}d\tau\\
 &\lesssim \int^{t}_{1}\tau^{-1-2\nu+2\mu\nu}\| f\|^4_{H^{1,1}}d\tau\lesssim 1.
\end{align*}
So we obtain for (\ref{2ahj}) that
\begin{align*}
 &\|(\ref{2ahj})\|_{L^{\infty}_{\sigma}}   \ge  C\|\zeta^{\frac{1}{2}}|U(\zeta)| \|^4_{L^4} \ln (t)
\end{align*}
for $t$ large and some $C>0$. And combining bounds of  (\ref{2ahj}) and  (\ref{3ahj}), we conclude that  (\ref{ahj}) fulfills
\begin{align*}
 &\|(\ref{ahj})\|_{L^{\infty}_{\sigma}}   \ge C\|\zeta^{\frac{1}{2}}|U(\zeta)| \|^4_{L^4} \ln (t)
\end{align*}
for $t$ large and some $C> 0$.

{\it Estimates of (\ref{bhj}).}   To dominate (\ref{bhj}), we note that in the support of   (\ref{bhj}), $|\eta'|\ge Ct^{-\mu}$.  Thus integration by parts in $\eta'$ yields
\begin{align*}
 &(\ref{bhj})=\frac{i}{ \pi}\int^{t}_1 \int_{\Bbb R^3}  e^{i\tau(2{\xi'}^2-2{\eta'}^2+2\zeta\sigma)}\chi_{\mu}(\xi') (\sigma-\zeta)\widehat{f} (\sigma-\zeta)\\
 &\cdot\partial_{\eta'}[\frac{1}{4i\tau\eta'}\gamma_{\mu}(\eta')(-\zeta-2\eta')\widehat{ \overline{{f}}}(\xi'+\zeta+ \eta')\widehat{ \overline{{f}}}(\zeta+ \eta'-\xi')\widehat{f}(-\zeta-2\eta')]
 d\xi' d\eta' d\zeta d\tau.
\end{align*}
By Lemma \ref{Coifman}, we thus get
\begin{align*}
\| (\ref{bhj})\|_{L^{\infty}_{\sigma}}
 &\lesssim \int^t_{1}\left(\tau^{-1+2\mu}\|e^{it\Delta}f\|^2_{H^2_x}
 +\tau^{-1+2\mu}\|e^{it\Delta}(xf)\|_{H^2_x}\|e^{it\Delta}f\|_{H^2_x}\right) \|e^{it\Delta}f\|^2_{W^{1,\infty}_x}d\tau\\
&\lesssim \int^t_1 \tau^{-2+3.5\mu+\beta+2\epsilon_*} d\tau\lesssim 1.
\end{align*}
So for some $C>0$
\begin{align*}
\| (\ref{bhj})\|_{L^{\infty}_{\sigma}}\le C.
\end{align*}

 {\it Estimates of (\ref{chj}), (\ref{dhj}).} The estimates of these terms are the same as  (\ref{bhj}).

Therefore, we conclude for the 4 order term associated with $\phi_{++}$ that
\begin{align*}
\|\int^t_1  {\bf R}_{++} d\tau\|_{L^{\infty}_{\sigma}}  \ge C \|\zeta^{\frac{1}{2}}|U(\zeta)| \|^4_{L^4} \ln (t)
\end{align*}
for $t$ large and some $C>0$.

\noindent{\it\underline{ Estimates  of 4 order term associated with $\phi_{+-}$.} }

{\it Space-time resonance analysis of $\phi_{+-}$. }
The time resonance set of $\phi_{+-}$ is
\begin{align*}
\mathcal{R}^{+-}_{t}=\{(\xi,\eta,\zeta,\sigma):\phi_{+-}=0\}.
\end{align*}
The space resonance sets are
\begin{align*}
\mathcal{R}^{+-}_{s,\xi}&=\{(\xi,\eta,\zeta,\sigma):\partial_{\xi}\phi_{+,-}=0\}=\{( \xi,\eta,\zeta,\sigma):\eta=0\}\\
\mathcal{R}^{+-}_{s,\eta}&=\{(\xi,\eta,\zeta,\sigma):\partial_{\eta}\phi_{+,-}=0\}=\{(  \xi,\eta,\zeta,\sigma): 2\eta-\xi-\zeta=0\}\\
\mathcal{R}^{+-}_{s,\zeta}&=\{(\xi,\eta,\zeta,\sigma):\partial_{\zeta}\phi_{+,-}=0\}=\{(  \xi,\eta,\zeta,\sigma): 2\zeta-\sigma-\eta=0\}.
\end{align*}
The space-time resonance set is
\begin{align*}
\mathcal{R}^{+-}_{s,t}&=\{(\xi,\eta,\zeta,\sigma): \eta=\sigma=\xi=\zeta=0\}.
\end{align*}
The $\phi_{+-}$   part  is easier than  $\phi_{++}$. In fact, the phrase
$\phi_{+-}$   has a non-degenerate Hessian at critical points, while the Hessian of $\phi_{++}$ discussed above at critical points is degenerate. Due to the
non-degenerateness, the stationary phrase analysis indeed gives a $t^{-\frac{3}{2}}$ decay.

Let
\begin{align*}
\eta=\frac{1}{2}\xi'+\eta';\mbox{  }\xi=\xi'-\zeta'-\frac{\sigma}{2}; \mbox{  }\zeta=\frac{\sigma}{2}+\zeta'.
\end{align*}
Then the inhomogeneous term  corresponding to $\phi_{+-}$ reads as
\begin{align*}
& \int^t_1{\bf R}_{+-}d\tau := c_4\int^{t}_1\int_{\Bbb R^3} e^{i\tau\phi_{+,-}}(\zeta-\eta)(\sigma-\zeta)\widehat{ \overline{{f}}}(\xi)\widehat{  {{f}}}(\eta-\xi)\widehat{f}(\zeta-\eta)\widehat{f} (\sigma-\zeta)d\xi d\eta d\zeta  d\tau   \\
&=c_4\int^t_{1}\int_{\Bbb R^3}e^{\frac{i\tau}{2}\sigma^2} e^{i\tau(-2{\eta'}^2+\frac{1}{2}{\xi'}^2-2{\zeta'}^2)} \widehat{ \overline{{f}}}(\xi'-\zeta'-\frac{1}{2}\sigma)\widehat{ {{f}}}(\eta'+\zeta'-\frac{1}{2}\xi'+\frac{1}{2}\sigma)\widehat{\partial_xf}(\frac{\sigma}{2}+\zeta'-\frac{1}{2}\xi'-\eta')\\
&\cdot\widehat{\partial_x f} (\frac{1}{2}\sigma-\zeta')d\xi' d\eta' d\zeta'  d\tau.
\end{align*}
Thus, by Plancherel identity, we have
\begin{align*}
&  \int^{t}_1\int_{\Bbb R^3} e^{i\tau\phi_{+,-}}(\zeta-\eta)(\sigma-\zeta)\widehat{ \overline{{f}}}(\xi)\widehat{ \overline{{f}}}(\eta-\xi)\widehat{f}(\zeta-\eta)\widehat{f} (\sigma-\zeta)d\xi d\eta d\zeta  d\tau   \\
&=\int^t_{1}\int_{\Bbb R^3}e^{\frac{i\tau}{2}\sigma^2} \mathcal{F}[e^{i\tau(-2{\eta'}^2+\frac{1}{2}{\xi'}^2-2{\zeta'}^2)}]X_{\sigma,\tau}(\widetilde{\xi},\widetilde{\eta},\widetilde{\zeta}) d\widetilde{\xi} d\widetilde{\eta}  d\widetilde{\zeta}   d\tau\\
&=\int^t_{1}e^{\frac{i\tau}{2}\sigma^2} \frac{{\pi}^{\frac{3}{2}}}{(2i)^{\frac{1}{2}}}\int_{\Bbb R^3}\frac{1}{\tau^{\frac{3}{2}}} e^{ - \frac{\widetilde{\eta}^2}{8i\tau}+\frac{{\widetilde{\xi}}^2}{2i\tau}- \frac{{\widetilde{\zeta}}^2 }{8i\tau}}X_{\sigma,\tau}(\widetilde{\xi},\widetilde{\eta},\widetilde{\zeta}) d\widetilde{\xi} d\widetilde{\eta}  d\widetilde{\zeta}   d\tau,
\end{align*}
where we denote
\begin{align*}
&X_{\sigma,\tau}(\widetilde{\xi},\widetilde{\eta},\widetilde{\zeta})\\
&:=\mathcal{F}^{-1}_{\xi',\eta',\zeta'}[\widehat{ \overline{{f}}}(\xi'-\zeta'-\frac{1}{2}\sigma)\widehat{ {{f}}}(\eta'+\zeta'-\frac{1}{2}\xi'+\frac{1}{2}\sigma)\widehat{\partial_xf}(\frac{\sigma}{2}+\zeta'-\frac{1}{2}\xi'-\eta')
 \cdot\widehat{\partial_xf} (\frac{1}{2}\sigma-\zeta')].
\end{align*}
By computation,
\begin{align*}
&X_{\sigma,\tau}(\widetilde{\xi},\widetilde{\eta},\widetilde{\zeta})\\
&=\int_{\Bbb R} e^{\frac{i}{2}\sigma(2y+2\widetilde{\xi}+\widetilde{\zeta}-\widetilde{\eta})}f(y)\partial_xf(y-\widetilde{\eta})\bar{f}(-\frac{1}{2}\widetilde{\eta}+y+\widetilde{\xi})
\partial_xf(\frac{1}{2}\widetilde{\eta}-y-\widetilde{\xi}-\widetilde{\zeta})dy.
\end{align*}
Hence,  we have
\begin{align*}
\int^t_1\| {\bf R}_{+-}\|_{L^{\infty}_{\sigma}} d\tau   \le \int^{t}_1\tau^{-\frac{3}{2}}\|f\|^4_{H^{1,1}_x}d\tau \le C
\end{align*}
for some $C >0$.

\noindent\underline{{\it Estimates  of 4 order term associated with $\phi_{--}$.}}

{\it Space-time resonance analysis of $\phi_{--}$. }
The time resonance set of $\phi_{--}$ is
\begin{align*}
\mathcal{R}^{--}_{t}=\{(\xi,\eta,\zeta,\sigma):\phi_{--}=0\}.
\end{align*}
The space resonance sets are
\begin{align*}
\mathcal{R}^{--}_{s,\xi}&=\{(\xi,\eta,\zeta,\sigma):\partial_{\xi}\phi_{--}=0\}=\{( \xi,\eta,\zeta,\sigma):2\xi=\eta\}\\
\mathcal{R}^{--}_{s,\eta}&=\{(\xi,\eta,\zeta,\sigma):\partial_{\eta}\phi_{--}=0\}=\{(  \xi,\eta,\zeta,\sigma): 2\eta-\xi-\zeta=0\}\\
\mathcal{R}^{--}_{s,\zeta}&=\{(\xi,\eta,\zeta,\sigma):\partial_{\zeta}\phi_{--}=0\}=\{(  \xi,\eta,\zeta,\sigma): 2\zeta-\sigma-\eta=0\}.
\end{align*}
The space-time resonance set is
\begin{align*}
\mathcal{R}^{--}_{s,t}&=\{(\xi,\eta,\zeta,\sigma): \eta=\sigma=\xi=\zeta=0\}.
\end{align*}
We observe that  the phrase
$\phi_{--}$   has a non-degenerate Hessian at critical points. Due to the
non-degenerateness, the stationary phrase analysis indeed gives a $t^{-\frac{3}{2}}$ decay. And the same argument of  $\phi_{+-}$ gives
\begin{align*}
\int^t_{1} \| {\bf R}_{--}\|_{L^{\infty}_{\sigma}} d\tau  \le \int^{t}_1\tau^{-\frac{3}{2}}\|f\|^4_{H^{1,1}_x}d\tau \le C
\end{align*}
for some $C >0$.

\noindent\underline{{\it Estimates  of  cubic term associated with $\phi_2$.}}

Recall that ${\bf R}_2$ ( see (\ref{b789})) writes as
\begin{align*}
\int^t_1 {\bf R}_2d\tau =c_2\int^t_1e^{i\tau\phi_2}(\eta-\xi)(\sigma-\eta)\widehat{f}(\xi)\widehat{f}(\eta-\xi)\widehat{f}(\sigma-\eta)d\xi d\eta d\tau,
\end{align*}
and $\phi_2=\sigma^2-\xi^2-(\eta-\xi)^2-(\sigma-\eta)^2$.

The time resonance set of $\phi_{2}$ is
\begin{align*}
\mathcal{R}^{\phi_2}_{t}=\{(\xi,\eta,\zeta,\sigma):\phi_{2}=0\}.
\end{align*}
The space resonance sets are
\begin{align*}
\mathcal{R}^{\phi_2}_{s,\xi}&=\{(\xi,\eta,\sigma):\partial_{\xi}\phi_{2}=0\}=\{( \xi,\eta,\sigma):2\xi=\eta\}\\
\mathcal{R}^{\phi_2}_{s,\eta}&=\{(\xi,\eta,\sigma):\partial_{\eta}\phi_{2}=0\}=\{(  \xi,\eta,\sigma): 2\eta-\xi-\sigma=0\}.
\end{align*}
The space-time resonance set is
\begin{align*}
\mathcal{R}^{\phi_2}_{s,t}&=\{(\xi,\eta,\sigma): \eta=\sigma=\xi=0\}.
\end{align*}
We observe that  the phrase
$\phi_{2}$   has a non-degenerate Hessian at critical points, and at the critical points the $(\eta-\xi)(\sigma -\eta)$ term emerging from derivatives vanishes.
Let
\begin{align*}
\xi&=\xi'+\frac{1}{3}\sigma+\frac{1}{2}\eta'\\
\eta&=\frac{2}{3}\sigma+\eta'.
\end{align*}
Then by Plancherel identity, one has
\begin{align}
\int^t_1 {\bf R}_2d\tau &=c_2 \int^t_1 \int_{\Bbb R^2} e^{i\tau\frac{2}{3}\sigma^2} e^{-i\tau(2{\xi'}^2+\frac{3}{2}{\eta'}^2)}(\frac{1}{3}\sigma+\frac{1}{2}\eta'-\xi')(\frac{1}{3}\sigma-\eta')
 \widehat{f}(\xi'+\frac{1}{3}\sigma+\frac{1}{2}\eta')\nonumber\\
 &\cdot \widehat{f}(\frac{1}{3}\sigma+\frac{1}{2}\eta'-\xi')
\widehat{f}(\frac{1}{3}\sigma-\eta')d\xi' d\eta' d\tau\nonumber\\
&=\tilde{c}\int^t_1 \int_{\Bbb R^2}\frac{1}{\tau} e^{-\frac{\widetilde{\xi}^2}{8i\tau}} e^{-\frac{\widetilde{\eta}^2}{6i\tau}}  W_{\sigma,\tau}(\widetilde{\xi},\widetilde{\eta})d \widetilde{\xi} d\widetilde{\eta} d\tau\nonumber\\
&=\tilde{c}\int^t_1 \int_{\Bbb R^2}\frac{1}{\tau}e^{i\tau\frac{2}{3}\sigma^2}    W_{\sigma,\tau}(\widetilde{\xi},\widetilde{\eta})d \widetilde{\xi} d\widetilde{\eta} d\tau
 +\tilde{c}\int^t_1 \int_{\Bbb R^2}\frac{1}{\tau}e^{i\tau\frac{2}{3}\sigma^2}  [e^{-\frac{\widetilde{\xi}^2}{8i\tau}} e^{-\frac{\widetilde{\eta}^2}{6i\tau}}-1]  W_{\sigma,\tau}(\widetilde{\xi},\widetilde{\eta})d \widetilde{\xi} d\widetilde{\eta} d\tau,\label{Pib}
\end{align}
where
\begin{align*}
W_{\sigma,\tau}(\widetilde{\xi},\widetilde{\eta})&=\mathcal{F}^{-1}_{\xi',\eta'}[(\frac{1}{3}\sigma+\frac{1}{2}\eta'-\xi')(\frac{1}{3}\sigma-\eta')
  \widehat{f}(\xi'+\frac{1}{3}\sigma+\frac{1}{2}\eta')\widehat{f}(\frac{1}{3}\sigma+\frac{1}{2}\eta'-\xi')\widehat{f}(\frac{1}{3}\sigma-\eta')].
\end{align*}
The first term in the RHS of (\ref{Pib}) equals
\begin{align*}
 \frac{1}{9}\tilde{c}\int^t_1 \frac{1}{\tau}e^{i\tau\frac{2}{3}\sigma^2}  \sigma^2  [\widehat{f} ( \frac{1}{3}\sigma )]^3 d\tau,
\end{align*}
which by integration by  parts in $\tau$
reduces to
\begin{align}
  \int^t_1e^{i\tau\frac{2}{3}\sigma^2} \left(\tilde{c}_1 {\tau}^{-1} \widehat{f} (\frac{1}{3}\sigma )\widehat{f} (\frac{1}{3}\sigma )\partial_{\tau}\widehat{f}(\frac{1}{3}\sigma )+\tilde{c}_2\tau^{-2} [\widehat{f} ( \frac{1}{3}\sigma )]^3\right)d\tau+ O(\|\widehat{f}\|^3_{L^{\infty}_{\sigma}})
\label{BBxx}
\end{align}
We claim
\begin{align}\label{PX}
  \|\partial_{\tau}\widehat{f}\|_{L^{\infty}_{\sigma}}\lesssim \tau^{-\frac{1}{2}}\|f\|^2_{H^{1,1}}+\tau^{-\frac{1}{2}}\|z\|^2_{H^2_x}.
\end{align}
If (\ref{PX}) has been proved, then we see (\ref{BBxx}) is dominated by
\begin{align*}
\| (\ref{BBxx}) \|_{L^{\infty}_{\sigma}}\lesssim \int^t_1 \tau^{-\frac{3}{2}}(\|f\|^2_{H^{1,1}}+\|z\|^2_{H^2_x})d\tau+ O(\|\widehat{f}\|^3_{L^{\infty}_{\sigma}})
\lesssim 1.
\end{align*}
The second term in the RHS of (\ref{Pib}) can be estimated as before by using $|e^{is}-1|\lesssim |s|^{\nu}$. In fact, it is dominated by
\begin{align*}
 \int^t_1  \tau^{-1-\nu}\|f\|^3_{H^{1,1}}d\tau\lesssim 1.
\end{align*}

Therefore, we get
\begin{align*}
\| \int^t_{1}{\bf R}_{2}d\tau \|_{L^{\infty}_{\sigma}}   \lesssim 1,
\end{align*}
provided that (\ref{PX}) holds.

\underline{{\it Proof of (\ref{PX}).}}
From the equation of $\partial_{t}\widehat{f}$, we observe that
\begin{align*}
 \partial_{t}\widehat{f}={\rm quadratic\mbox{ }  term}+{\rm cubic  \mbox{ }terms}+{\rm 4 \mbox{ }order  \mbox{ }terms}+ {\rm higher \mbox{ }order  \mbox{ }terms}.
\end{align*}
By Lemma \ref{4p}, the  cubic, 4 order and higher order  terms can provide at least $t^{-\frac{1}{2}}$ decay. Thus it suffices to prove  that the quadratic  term fulfills
\begin{align}\label{xzPX}
  \|\int_{\Bbb R} e^{i\tau \phi_1}\xi( \sigma-\xi ) \widehat{f}( \sigma-\xi ) \widehat{f}( \xi )d\xi\|_{L^{\infty}_{\sigma}}\lesssim \tau^{-\frac{1}{2}} \|f\|^2_{H^{1,1}}.
\end{align}
By change of variables and Plancherel identity,
 \begin{align*}
&\int_{\Bbb R} e^{i\tau \phi_1}\xi( \sigma-\xi ) \widehat{f}( \sigma-\xi ) \widehat{f}( \xi )d\xi=\int_{\Bbb R} e^{\frac{i}{2}\tau\sigma^2}e^{-2i\tau {\xi'}^2}
 \widehat{\partial_x f}(\frac{1}{2} \sigma-\xi' ) \widehat{\partial_x f}( \xi'+\frac{1}{2}\sigma )d\xi'\\
&=\int_{\Bbb R}e^{\frac{i} {2}\tau\sigma^2}a  {\tau^{-\frac{1}{2}}} e^{- \frac{ {\widetilde{\xi}}^2}{8i\tau}}
 \mathcal{F}^{-1}_{\xi'}[\widehat{\partial_x f}(\frac{1}{2} \sigma-\xi' ) \widehat{\partial_x f}( \xi'+\frac{1}{2}\sigma )]d\widetilde{\xi}.
\end{align*}
Then (\ref{xzPX}) follows by  Hausdorff-Young inequality for convolution.

\noindent\underline{{\it Estimates  of  quadratic term associated with (\ref{a789}).}}

{\it Space-time resonance analysis of $\phi_{1}$. }
The corresponding quadratic term  ${\bf R}_1$ writes as
 \begin{align*}
\int^t_{1}  {\bf R}_1 d\tau=c_0\int^t_{1}\int_{\Bbb R} e^{i\tau\phi_1}\xi(\sigma-\xi)\widehat{f}(\xi)\widehat{f}(\sigma-\xi)d\xi d\tau.
\end{align*}
Observe that
\begin{align*}
 \partial_te^{it\phi_1}=2i\xi(\sigma-\xi).
\end{align*}
Thus by integration by parts in $\tau$, we have
 \begin{align}
\int^t_1{\bf R}_1 d\tau=-\frac{c_0}{2i}\int^t_{1}\int_{\Bbb R}e^{i\tau\phi_1}\partial_t[\widehat{f}(\xi)\widehat{f}(\sigma-\xi)]d\xi d\tau+O\left(\left\|\int_{\Bbb R} \widehat{f}(\xi)\widehat{f}(\sigma-\xi) d\xi\right\|_{L^{\infty}_{\sigma}}\right).\label{Ki0}
\end{align}
The last term in the RHS is easy to estimate, in fact by Lemma \ref{XY}
 \begin{align*}
\|\int_{\Bbb R} \widehat{f}(\xi)\widehat{f}(\sigma-\xi) d\xi\|_{L^{\infty}_{\sigma}}\lesssim \|f\|^2_{L^2_x}\lesssim \epsilon_*.
\end{align*}
To dominate the first term in the RHS of (\ref{Ki0}), by symmetry it suffices to consider
\begin{align}
-\frac{1}{2i}\int^t_{1}\int_{\Bbb R}e^{i\tau\phi_1} (\partial_t\widehat{f})(\xi)\widehat{f}(\sigma-\xi) d\xi d\tau.\label{Ki2}
\end{align}
Recall that
\begin{align*}
 &\partial_t\widehat{f} (\xi)\nonumber\\
 &=-\frac{ic_0}{2\pi}\int_{\Bbb R}e^{i\tau(\xi^2-\zeta^2-(\xi-\zeta)^2)} \zeta(\xi-\zeta)\widehat{f} (\zeta)\widehat{f} (\xi-\zeta)d \zeta  \\
&-\frac{ic_1}{2\pi}\int_{\Bbb R^2}e^{i\tau(\xi^2+\zeta^2-(\eta-\zeta)^2-(\eta-\xi)^2)}(\eta-\zeta)(\xi-\eta)\widehat{\overline{f}} (\zeta)\widehat{f} (\eta-\zeta)\widehat{f} (\xi-\eta)d\zeta d\eta \\
&-\frac{ic_2}{2\pi}\int_{\Bbb R^2}e^{i\tau(\xi^2-\zeta^2-(\eta-\zeta)^2-(\eta-\xi)^2)}(\eta-\zeta)(\xi-\eta)\widehat{{f}} (\zeta)\widehat{f} (\eta-\zeta)\widehat{f} (\xi-\eta)d\zeta d\eta \\
&+{\bf  R}_{\pm,\pm}+\mathcal{R},\nonumber
\end{align*}
where ${\bf R}_{\pm,\pm}$ denotes 4 order terms associated with $\phi_{\pm,\pm}$.
So to bound (\ref{Ki2}), it suffices to estimate
\begin{align*}
A_1&:=\int^t_{1}\int_{\Bbb R}e^{i\tau (2\sigma\xi-2\xi^2)}e^{i\tau(\xi^2-\zeta^2-(\xi-\zeta)^2)}\zeta(\xi-\zeta) \widehat{f} (\zeta)\widehat{f} (\xi-\zeta) \widehat{f}(\sigma-\xi) d \zeta d\xi d\tau \\
A_2&:=\int^t_{1}\int_{\Bbb R}e^{i\tau (2\sigma\xi-2\xi^2)}e^{i\tau(\xi^2+\zeta^2-(\eta-\zeta)^2-(\eta-\xi)^2)}(\eta-\zeta)(\xi-\eta)\widehat{\overline{f}} (\zeta)\widehat{f} (\eta-\zeta)\\
&\cdot\widehat{f} (\xi-\eta)\widehat{f}(\sigma-\xi) d \zeta d\eta d\xi d\tau \nonumber\\
A_3&:=\int^t_{1}\int_{\Bbb R}e^{i\tau (2\sigma\xi-2\xi^2)}e^{i\tau(\xi^2-\zeta^2-(\eta-\zeta)^2-(\eta-\xi)^2)}(\eta-\zeta)(\xi-\eta)\widehat{{f}} (\zeta)\widehat{f} (\eta-\zeta)\\
& \cdot\widehat{f} (\xi-\eta) \widehat{f}(\sigma-\xi) d \zeta d\eta d\xi d\tau \nonumber\\
A_4&:=\int^t_{1}\int_{\Bbb R}e^{i\tau (2\sigma\xi-2\xi^2)}({\bf  R}_{\pm,\pm}+\mathcal{R})(\xi)\widehat{f}(\sigma-\xi)   d\xi d\tau.
\end{align*}
Let $\varphi_1= 2\sigma\xi-2\xi^2+\xi^2-\zeta^2-(\xi-\zeta)^2$.
The space resonance sets of $\varphi_1$ are
\begin{align*}
\mathcal{R}^{\varphi_1}_{s,\xi}&=\{(\xi, \zeta,\sigma):\partial_{\xi}\varphi_{1}=0\}=\{( \xi,\zeta,\sigma):\sigma-2\xi+\zeta=0\}\\
\mathcal{R}^{\varphi_1}_{s,\zeta}&=\{(\xi, \zeta,\sigma):\partial_{\zeta}\varphi_{1}=0\}=\{(  \xi,\eta,\zeta,\sigma):\xi=2\zeta\}.
\end{align*}
The space-time resonance set is
\begin{align*}
\mathcal{R}^{\varphi_1}_{s,t}&=\{(\xi,\eta,\zeta,\sigma): \sigma=\xi=\zeta=0\}.
\end{align*}
We observe that the phrase
$\varphi_1$  has a non-degenerate Hessian at critical points, and at the critical points the $ \zeta(\xi-\zeta)$ term emerging from derivatives vanishes.
Thus  the same argument of  $\phi_{2}$ gives
\begin{align*}
\|A_1\|_{L^{\infty}_{\sigma}}   \lesssim  \int^{t}_1\tau^{-1-\nu}\|f\|^3_{H^{1,1}_x}d\tau \le C
\end{align*}
for some $C >0$.

Let $\varphi_2= 2\sigma\xi-2\xi^2+\xi^2+\zeta^2-(\eta-\zeta)^2-(\eta-\xi)^2$.
The space resonance sets of $\varphi_2$ are
\begin{align*}
\mathcal{R}^{\varphi_2}_{s,\xi}&=\{(\xi, \zeta,\eta,\sigma):\partial_{\xi}\varphi_{2}=0\}=\{( \xi,\zeta,\eta,\sigma):\sigma-2\xi+\eta=0\}\\
\mathcal{R}^{\varphi_2}_{s,\zeta}&=\{(\xi, \zeta,\eta,\sigma):\partial_{\zeta}\varphi_{2}=0\}=\{(  \xi,\zeta, \eta,\sigma): \eta=0\}\\
\mathcal{R}^{\varphi_2}_{s,\eta}&=\{(\xi, \zeta,\eta,\sigma):\partial_{\eta}\varphi_{2}=0\}=\{(  \xi,\zeta,\eta,\sigma):\xi+\zeta=2\eta\}.
\end{align*}
The space-time resonance set is
\begin{align*}
\mathcal{R}^{\varphi_2}_{s,t}&=\{(\xi,\zeta,\eta,\sigma): \sigma=\eta=\xi=\zeta=0\}.
\end{align*}
We observe that the phrase
$\varphi_2$  has a non-degenerate Hessian at critical points.
Then the same argument of  $\phi_{+-}$ gives
\begin{align*}
\|A_2\|_{L^{\infty}_{\sigma}}   \lesssim \int^{t}_1\tau^{-\frac{3}{2}}\|f\|^4_{H^{1,1}_x}d\tau \le C
\end{align*}
for some $C >0$.

Let $\varphi_3=2\sigma\xi-2\xi^2+\xi^2-\zeta^2-(\eta-\zeta)^2-(\eta-\xi)^2$. The space resonance sets of $\varphi_3$ are
\begin{align*}
\mathcal{R}^{\varphi_3}_{s,\xi}&=\{(\xi, \zeta,\eta,\sigma):\partial_{\xi}\varphi_{3}=0\}=\{( \xi,\zeta,\eta,\sigma):\sigma+\eta=\xi\}\\
\mathcal{R}^{\varphi_3}_{s,\zeta}&=\{(\xi, \zeta,\eta,\sigma):\partial_{\zeta}\varphi_{3}=0\}=\{(  \xi,\zeta, \eta,\sigma): 2\eta=\xi+\zeta\}\\
\mathcal{R}^{\varphi_3}_{s,\eta}&=\{(\xi, \zeta,\eta,\sigma):\partial_{\eta}\varphi_{3}=0\}=\{(  \xi,\zeta,\eta,\sigma):\eta=2\zeta\}.
\end{align*}
The space-time resonance set is
\begin{align*}
\mathcal{R}^{\varphi_3}_{s,t}&=\{(\xi,\zeta,\eta,\sigma): \sigma=\eta=\xi=\zeta=0\}.
\end{align*}
We observe that the phrase
$\varphi_3$  has a non-degenerate Hessian at critical points.
And the same argument of  $\phi_{+-}$ gives
\begin{align*}
\|A_3\|_{L^{\infty}_{\sigma}}   \lesssim \int^{t}_1\tau^{-\frac{3}{2}}\|f\|^4_{H^{1,1}_x}d\tau \le C
\end{align*}
for some $C >0$.

For $A_4$, Lemma  10.3 implies
\begin{align*}
\|A_4\|_{L^{\infty}_{\sigma}}   \lesssim   \int^{t}_1\tau^{-\frac{3}{2}}\|z\|^2_{H^{2}_x}d\tau \le C
\end{align*}
for some $C >0$.

Therefore, we conclude for $\mathbf{R}_{1}$ that
\begin{align*}
\|\int^t_{1} {\bf R}_1d\tau \|_{L^{\infty}_{\sigma}}   \le  C
\end{align*}
for some $C >0$.

\underline{Estimates of higher order terms  $\mathcal{R}$.}

The estimates of higher order terms  $\mathcal{R}$  are the same as Section 8, and in fact
\begin{align*}
\| \int^t_{1} {\mathcal R}d\tau \|_{L^{\infty}_{\sigma}}   \le C
\end{align*}
for some $C >0$.

Recall $\widehat{F}(t,\sigma)$ defined by (\ref{xPQ9}) and its equation given by (\ref{xY7m}). Now, we have proved
\begin{align*}
\| \widehat{f}(t,\sigma)\|_{L^{\infty}_{\sigma}}   \ge C \|\zeta^{\frac{1}{2}}|U(\zeta)| \|^4_{L^4}\ln (t)
\end{align*}
for $t$ large and some $C>0$. Then the desired result follows since $|\widehat{f}|=|\widehat{z}|$.
\end{proof}

\begin{Proposition}\label{V}
Suppose that $c_5= 0$ in (\ref{Gifinal}).
Under the assumptions (\ref{dKEY}), (\ref{suption2p}), (\ref{suption2py}), (\ref{suption2pp}), we have  as $t\to \infty$
\begin{align*}
 \| t z\partial_x z   \|_{L^{2}_{x}} \lesssim t^{\frac{1}{2}-\nu+3\beta}.
\end{align*}
for $3\beta<\nu<\frac{1}{4}$.
\end{Proposition}
\begin{proof}
Recall that $z  $ satisfies   (\ref{Gifinal}). Since
$$c_5= [\ln h]_{z\bar{z}\bar{z}}(0), \mbox{ } c_4=2[\ln h]_{z {z}\bar{z}}(0), \mbox{ }c_0=[\ln h]_{z}(0), \mbox{ }c_1=[\ln h]_{z\bar{z}}(0),$$
 $c_5=0$ and (\ref{dKEY}) together imply that
\begin{align*}
 c_4=0, c_0\neq 0, \mbox{ }c_1\neq 0.
\end{align*}
So $z$ indeed satisfies
\begin{align}\label{gg7}
\begin{cases}
i\partial_t z+\Delta z=&c_0 \partial_xz\partial_xz+c_1 \bar{z}\partial_xz\partial_xz+c_2 {z}\partial_xz\partial_xz+  c_3 {z}^2\partial_xz\partial_xz+
    O(|z|^3)\partial_xz\partial_xz\\
z\upharpoonright_{t=0}  = z_0.&
\end{cases}
\end{align}
Let $w=z+ \kappa_1 z^2+\kappa_2 z^3+\kappa_4 z^4$, we have
\begin{align*}
 i\partial_t w+\Delta w
 =&\left[(c_0 +2\kappa_1)+(c_2+2\kappa_1c_0+6\kappa_2)z+(c_3+2\kappa_1 c_2+3c_0\kappa_2+12\kappa_3)z^2 \right](\partial_xz)^2\nonumber\\
  & +(c_1\bar{z}+2\kappa_1c_2\bar{z}z)(\partial_xz)^2+ O(|z|^3)\partial_xz\partial_xz.
\end{align*}
Choose $\{\kappa_i\}^3_{i=1}$ to fulfill
\begin{align}
 c_0 +2\kappa_1 = c_2+2\kappa_1c_0+6\kappa_2= c_3+2\kappa_1 c_2+3c_0\kappa_2+12\kappa_3 =0.
\end{align}
Then
\begin{align}\label{fff}
 i\partial_t w+\Delta w=c_1\bar{z}(\partial_xz)^2+2\kappa_1c_2\bar{z}z (\partial_xz)^2+ O(|z|^3)\partial_xz\partial_xz.
\end{align}
Moreover, one has
\begin{align*}
\kappa_1\neq 0.
\end{align*}
Define  $g=e^{-it\Delta}w$, $f=e^{-it\Delta}z$.
Then $g$ fulfills
\begin{align}
\widehat{g}(t,\sigma)&=\widehat{g}(\upharpoonright_{t=1})- \frac{ ic_1}{2\pi}\int^{t}_1 e^{i\tau\phi_0}(\eta-\xi)(\sigma-\eta)\widehat{\overline{f}}(\xi)\widehat{f}(\eta-\xi)\widehat{f} (\sigma-\eta)d\xi d\eta d\tau+\int^t_1\mathcal{R}d\tau \nonumber\\
&- \frac{i}{2\pi}\int^{t}_1 2\kappa_1c_{2}e^{i\tau\phi_{+-}}(\zeta-\eta)(\sigma-\zeta)\widehat{{\overline{f}}}(\xi)\widehat{ {f}}(\eta-\xi)\widehat{f}(\zeta-\eta)\widehat{f} (\sigma-\zeta)d\xi d\eta d\zeta d\tau,\label{final1}
\end{align}
where we denote
\begin{align*}
\phi_0&=\sigma^2+\xi^2-(\eta-\xi)^2-(\sigma-\eta)^2 \\
\phi_{+-}&=\sigma^2+\xi^2-(\xi-\eta)^2-(\zeta-\eta)^2-(\sigma-\zeta)^2 \\
\mathcal{R}&=\mathcal{F}[e^{-i\tau \Delta} O(|z|^3) (\partial_x z)^2].
\end{align*}
We aim to bound $\|\partial_{\sigma}g\|_{L^2_{\sigma}}$.

\noindent {\it\underline{ Estimates of the leading cubic term.}}  By change of variables,
\begin{align*}
&\int^{t}_1 e^{i\tau\phi_0}(\eta-\xi)(\sigma-\eta)\widehat{\overline{f}}(\xi)\widehat{f}(\eta-\xi)\widehat{f} (\sigma-\eta)d\xi d\eta\\
&=\int^{t}_1 e^{2i\tau\xi'\eta }(\sigma-\xi')(\sigma-\eta )\widehat{\overline{f}}(\xi'+\eta-\sigma)\widehat{f}(-\xi'+\sigma)\widehat{f} (\sigma-\eta)d\xi' d\eta.
\end{align*}
Then $\partial_{\sigma}$ will not hit the phase function, and by Plancherel identity, one has
\begin{align*}
&\|\partial_{\sigma}\int^{t}_1 e^{i\tau\phi_0}(\eta-\xi)(\sigma-\eta)\widehat{\overline{f}}(\xi)\widehat{f}(\eta-\xi)\widehat{f} (\sigma-\eta)d\xi d\eta\|_{L^2_{\sigma}}\\
&\lesssim \int^{t}_1 \|z\|^2_{W^{1,\infty}}\| f\|_{H^{1,1}_x}d\tau\\
& \lesssim  \langle  t\rangle^{\beta}.
\end{align*}

\underline{{\it Estimates  of 4 order terms  ${\bf R}_{+-}$.}}

{\it Space-time resonance analysis of $\phi_{+-}$. }
We have seen in the proof of Proposition 10.2 that the phrase
$\phi_{+-}$   has a non-degenerate Hessian at critical points, and thus the stationary phrase analysis indeed gives a $t^{-\frac{3}{2}}$ decay.
The inhomogeneous term  corresponding to $\phi_{+-}$ reads as
\begin{align*}
& \partial_{\sigma}\int^{t}_1\int_{\Bbb R^3} e^{i\tau\phi_{+-}}(\zeta-\eta)(\sigma-\zeta)\widehat{ \overline{{f}}}(\xi)\widehat{ {{f}}}(\eta-\xi)\widehat{f}(\zeta-\eta)\widehat{f} (\sigma-\zeta)d\xi d\eta d\zeta  d\tau   \\
&=\int^{t}_1\int_{\Bbb R^3} e^{i\tau\phi_{+-}}2i\zeta\tau (\zeta-\eta)(\sigma-\zeta)\widehat{ \overline{{f}}}(\xi)\widehat{ {{f}}}(\eta-\xi)\widehat{f}(\zeta-\eta)\widehat{f} (\sigma-\zeta)d\xi d\eta d\zeta  d\tau   \\
&+\int^{t}_1\int_{\Bbb R^3} e^{i\tau\phi_{+-}} \partial_{\sigma}[(\zeta-\eta)(\sigma-\zeta)\widehat{ \overline{{f}}}(\xi)\widehat{ {{f}}}(\eta-\xi)\widehat{f}(\zeta-\eta)\widehat{f} (\sigma-\zeta)]d\xi d\eta d\zeta  d\tau   \\
&:={\bf I}+{\bf  II}.
\end{align*}
By  Plancherel identity, ${\bf II}$ is bounded by
\begin{align*}
\|{\bf  II}\|_{L^2_{\sigma}}\lesssim \int^t_1\|z\|^3_{W^{1,\infty}_x} \|z\|_{H^1_x}d\tau \lesssim \int^t_1 \tau^{-\frac{3}{2}} d\tau \lesssim 1.
\end{align*}

By the identity $\zeta=\xi+(\eta-\xi)+(\zeta-\eta)$, ${\bf I}$ further expands as
\begin{align*}
{\bf I}&=\int^{t}_1\int_{\Bbb R^3} e^{i\tau\phi_{+-}}2i \tau \xi(\zeta-\eta)(\sigma-\zeta)\widehat{ \overline{{f}}}(\xi)\widehat{ {{f}}}(\eta-\xi)\widehat{f}(\zeta-\eta)\widehat{f} (\sigma-\zeta)d\xi d\eta d\zeta  d\tau\\
&+\int^{t}_1\int_{\Bbb R^3} e^{i\tau\phi_{+-}}2i \tau (\eta-\xi) (\zeta-\eta) (\sigma-\zeta)\widehat{ \overline{{f}}}(\xi)\widehat{ {{f}}}(\eta-\xi)\widehat{f}(\zeta-\eta)\widehat{f} (\sigma-\zeta)d\xi d\eta d\zeta  d\tau\\
&+\int^{t}_1\int_{\Bbb R^3} e^{i\tau\phi_{+-}}2i \tau  (\zeta-\eta)^2 (\sigma-\zeta)\widehat{ \overline{{f}}}(\xi)\widehat{ {{f}}}(\eta-\xi)\widehat{f}(\zeta-\eta)\widehat{f} (\sigma-\zeta)d\xi d\eta d\zeta  d\tau\\
&:={\bf I}_1+{\bf I}_2+{\bf  I}_3.
\end{align*}
Let
\begin{align*}
\eta=\frac{1}{2}\xi'+\eta';\mbox{  }\xi=\xi'-\zeta'-\frac{\sigma}{2}; \mbox{  }\zeta=\frac{\sigma}{2}+\zeta'.
\end{align*}
Then by Plancherel identity, we obtain
\begin{align*}
{\bf I}_1=\int^t_{1}e^{\frac{i\tau}{2}\sigma^2}  {{\pi}^{\frac{3}{2}}}{(2i)^{\frac{1}{2}}}\int_{\Bbb R^3} {\tau^{-\frac{1}{2}}} e^{ - \frac{\widetilde{\eta}^2}{8i\tau}+\frac{{\widetilde{\xi}}^2}{2i\tau}- \frac{{\widetilde{\zeta}}^2 }{8i\tau}}\widetilde{X}_{\sigma,\tau}(\widetilde{\xi},\widetilde{\eta},\widetilde{\zeta}) d\widetilde{\xi} d\widetilde{\eta}  d\widetilde{\zeta}   d\tau,
\end{align*}
where $\widetilde{X}_{\sigma,\tau}(\widetilde{\xi},\widetilde{\eta},\widetilde{\zeta})$ is given by
\begin{align*}
&\widetilde{X}_{\sigma,\tau}(\widetilde{\xi},\widetilde{\eta},\widetilde{\zeta})\\
&=\mathcal{F}^{-1}_{\xi',\zeta',\eta'}[\xi (\zeta-\eta)(\sigma-\zeta)\widehat{ \overline{{f}}}(\xi)\widehat{ {{f}}}(\eta-\xi)\widehat{f}(\zeta-\eta)\widehat{f} (\sigma-\zeta)]\\
&=\int_{\Bbb R} e^{\frac{i}{2}\sigma(-\widetilde{\eta}+2y+2\widetilde{\xi}+\widetilde{\zeta})}\partial_xf(-y)f(\widetilde{\eta}-y)\partial_x\bar{f}(\frac{1}{2}\widetilde{\eta}-y+\widetilde{\xi})
\partial_xf(\frac{1}{2}\widetilde{\eta}-y-\widetilde{\xi}-\widetilde{\zeta})dy.
\end{align*}
As before, expanding $e^{ - \frac{\widetilde{\eta}^2}{8i\tau}+\frac{{\widetilde{\xi}}^2}{2i\tau}- \frac{{\widetilde{\zeta}}^2 }{8i\tau}}$ to $1$ and the difference, one has
\begin{align*}
{\bf I}_1&=b_1\int^t_{1}  e^{\frac{i\tau}{2}\sigma^2}  {\tau^{-\frac{1}{2}}}  \sigma^3 \widehat{\bar{f}}( -\frac{\sigma}{2}) \widehat{ {f}}(  \frac{\sigma}{2}) \widehat{ {f}}(  \frac{\sigma}{2})d\tau\\
&+b_2\int^t_{1}e^{\frac{i\tau}{2}\sigma^2} {\tau^{-\frac{1}{2}}}\int_{\Bbb R^3} [ e^{ - \frac{\widetilde{\eta}^2}{8i\tau}+\frac{{\widetilde{\xi}}^2}{2i\tau}- \frac{{\widetilde{\zeta}}^2 }{8i\tau}}-1]\widetilde{X}_{\sigma,\tau}(\widetilde{\xi},\widetilde{\eta},\widetilde{\zeta}) d\widetilde{\xi} d\widetilde{\eta}  d\widetilde{\zeta}   d\tau\\
&:={\bf I}_{11}+{\bf I}_{12},
\end{align*}
where $b_1,b_2$ are some universal  constants.  By integration by parts in $\tau$, ${\bf I}_{11}$ is dominated by
\begin{align*}
\|{\bf I}_{11}\|_{L^2_{\sigma}}&\lesssim\| \int^t_{1}  e^{\frac{i\tau}{2}\sigma^2}  {\tau^{-\frac{1}{2}}}  \sigma  \partial_{\tau}[\widehat{\bar{f}}( -\frac{\sigma}{2}) \widehat{ {f}}(  \frac{\sigma}{2}) \widehat{ {f}}(  \frac{\sigma}{2})]d\tau\|_{L^{\sigma}_2}+  \| \sigma  \widehat{{f}}( -\frac{\sigma}{2}) \|_{L^{\infty}_{t}L^2_{\sigma}} \| \widehat{{f}}\|^3_{L^{\infty}_{t}L^{\infty}_{\sigma}} \\
&\lesssim \int^t_{1}    {\tau^{-1} }\| f \|^2_{H^{1,1}_{\sigma}}\|f\|_{H^1_x} \|\widehat{{f}}\|_{L^{\infty}_{\sigma}} d\tau+\epsilon^4_*\\
&\lesssim \langle t\rangle^{2\beta},
\end{align*}
where we applied  (\ref{PX}) in the second inequality.

Using $|e^{is}-1|\lesssim |s|^{\nu}$, by Plancherel identity and change of variables, we find  ${\bf I}_{12}$ is controlled by
\begin{align*}
&\|{\bf I}_{12}\|_{L^2_{\sigma}}
 \lesssim \int^t_{1} \tau^{-\frac{1}{2}-\nu} \|f\|_{L^2_x}\|\langle x\rangle^{2\nu} \partial_x f\|^3_{L^1_x} d\tau \lesssim \langle t\rangle^{ \frac{1}{2}-\nu+3\beta}.
\end{align*}

The estimate of
${\bf I}_2$ is the same as ${\bf I}_1$.
For ${\bf I}_3$, similarly  we have
\begin{align*}
 \|{\bf I}_{3}\|_{L^2_{\sigma}}
 &\lesssim \int^t_{1} \tau^{-\frac{1}{2}-\nu} \|f\|_{H^2_x}\|f\|^3_{H^{1,1}_x} d\tau + \int^t_{1} \tau^{-1} \|f\|_{H^1_x}\|f\|^2_{H^{1,1}_x}\|\widehat{f}\|_{L^{\infty}_{\sigma}} d\tau
+\|f\|_{{L^{\infty}_ t H^{1}_x}}\|\widehat{f}\|^3_{L^{\infty}_{t,\sigma}}\\
&\lesssim \langle t\rangle^{ \frac{1}{2}-\nu+3\beta}.
\end{align*}

Hence, we obtain
\begin{align*}
 \| \partial_{\sigma}\int^t_1{\bf R}_{+-}d\tau\|_{L^{2}_{\sigma}} d\tau   \lesssim \langle t \rangle^{\frac{1}{2}-\nu+3\beta}.
\end{align*}
for some $3\beta<\nu<\frac{1}{4}$.

And for the high order term $\mathcal{R}$, we have
\begin{align*}
 \| \partial_{\sigma}\int^t_1{\mathcal R} d\tau\|_{L^{2}_{\sigma}} d\tau &  \lesssim \int^t_{1}\|L\left(O(|z|^3)(\partial_x z)^2\right)\|_{L^2_x}d\tau\lesssim \epsilon^4\int^{t}_1\langle \tau \rangle^{-1+\epsilon}+\langle \tau \rangle^{-2+\beta}d\tau\nonumber \\
 &\lesssim \epsilon^4 \langle t \rangle^{\epsilon_*}.
\end{align*}

Combining the above three  results on cubic, 4 order and higher order terms, we infer from  (\ref{final1}) that
\begin{align*}
\|\partial_{\sigma}\widehat{g}(t,\sigma)\|_{L^{2}_{\sigma}}\lesssim \langle t \rangle^{\frac{1}{2}-\nu+3\beta},
\end{align*}
which further gives
\begin{align}\label{RxRT}
\|L(z+\kappa_1z^2+\kappa_2z^3+\kappa_3 z^4)\|_{L^{2}_{x}}\lesssim \langle t \rangle^{\frac{1}{2}-\nu+3\beta}.
\end{align}
Since  (\ref{suption2p}) implies
\begin{align*}
\|L(z^3) \|_{L^{2}_{x}}&\lesssim \|Lz\|_{L^2_x}\|z\|^2_{L^{\infty}_x}+t\|z\|^2_{L^{\infty}_{x}}\|z\|_{H^1_x}\lesssim 1\\
\|L(z^4) \|_{L^{2}_{x}}&\lesssim \|Lz\|_{L^2_x}\|z\|^3_{L^{\infty}_x}+t\|z\|^3_{L^{\infty}_{x}}\|z\|_{H^1_x}\lesssim 1,
\end{align*}
we get from (\ref{RxRT}) that
\begin{align*}
|\kappa_1|\|L(z^2) \|_{L^{2}_{x}}&\lesssim  \langle t \rangle^{\frac{1}{2}-\nu+3\beta}.
\end{align*}
And since
\begin{align*}
 &L(z^2) =z Lz-2t(\partial_x z)z\\
 &\|Lz\|_{L^2_{x}}\lesssim \langle t\rangle^{\beta},\mbox{ }\| z\|_{L^{\infty}_{x}}\lesssim \langle t\rangle^{-\frac{1}{2}},
\end{align*}
we finally obtain that by $\kappa_1\neq 0$ that
\begin{align*}
\|2t(\partial_x z)z\|_{L^2_x}\lesssim \langle t \rangle^{\frac{1}{2}-\nu+3\beta}.
\end{align*}
\end{proof}

\subsection{End of Proof to Theorem 1.2}

Assume (\ref{dKEY}),   (\ref{suption2py}), (\ref{suption2p}), (\ref{suption2pp}) and $U\neq 0$.
We consider two cases. \\
{\bf Case 1.} Suppose that $c_5\neq  0$ in (\ref{Gifinal}).
Then  Proposition \ref{Gbbb} shows for $t$ large and some $C> 0$
\begin{align*}
 \| \widehat{w}(t) \|_{L^{\infty}_{\sigma}}\ge C \|\zeta^{\frac{1}{2}}|U(\zeta)| \|^4_{L^4_{\zeta}}\ln (t).
\end{align*}
This implies  $\| \widehat{w}(t) \|_{L^{\infty}_{\sigma}}$ grows at least as fast as $\ln (t)$. But (\ref{suption2p}) shows
$\| \widehat{w}(t) \|_{L^{\infty}_{t,\sigma}}\lesssim 1$, thus yielding contradiction. In other words, $U$ must be constantly zero, and thus $w\equiv 0$ by  almost conservation of mass.

{\bf Case 2.} Suppose that $c_5= 0$ in (\ref{Gifinal}). Then Proposition \ref{V} shows for $t\ge 1$
\begin{align}\label{Yuu}
 \|  z\partial_x z   \|_{L^{2}_{x}} \lesssim t^{-\frac{1}{2}-\nu+3\beta}
\end{align}
with $3\beta<\nu<\frac{1}{4}$.
By (\ref{9AAA}), we have
\begin{align*}
 z(t,x)=\frac{e^{\frac{i|x|^2}{4t}}}{(4it)^{\frac{1}{2}}}\mathcal{F}[e^{-it\Delta}z(t)]\left(2t,\frac{x}{2t}\right)+R(t,x)\\
  \partial_xz(t,x)=\frac{e^{\frac{i|x|^2}{4t}}}{(4it)^{\frac{1}{2}}}\mathcal{F}[e^{-it\Delta}\partial_xz(t)]\left(2t,\frac{x}{2t}\right)+ \underline{{R}}(t,x)
\end{align*}
 where $R(t,x)$ and ${\underline{{R}}}(t,x)$ satisfy (9.6), (9.7) as well, i.e.,
 \begin{align*}
&t^{\frac{1}{2}}\|R(t,x)\|_{L^2_x} +t^{\frac{5}{8}}\|R(t,x)\|_{L^{\infty}_x}\lesssim \|e^{-it\Delta} z\|_{H^{0,1}}\lesssim \|Lz\|_{L^2_x}+\|z\|_{L^2_x}\lesssim \langle t\rangle^{\beta}\\
& t^{\frac{1}{2}}\|\underline{R}(t,x)\|_{L^2_x} +t^{\frac{5}{8}}\|\underline{R}(t,x)\|_{L^{\infty}_x}\lesssim \|e^{-it\Delta} z\|_{H^{1,1}}\lesssim \sum_{j=0,1}\|L\partial^{j}_xz\|_{L^2_x}+\|z\|_{H^1_x}\lesssim \langle t\rangle^{\beta}.
\end{align*}
 So  (\ref{suption2p}) and (\ref{suption2pp}) imply
\begin{align*}
\|z\partial_x z  \|_{L^2_x} \sim \| \frac{1}{(4it)} U(\frac{x}{2t}) \frac{ix}{2t} U(\frac{x}{2t})\|_{L^2_x}\sim t^{- \frac{1}{2}}\|x^{\frac{1}{2}} U \|^2_{L^4_x},
\end{align*}
which contradicts with  (\ref{Yuu}). Thus $U=0$, and $z\equiv0$ by almost conservation of mass.

\section{Proof of Theorem 1.3}

Since Theorem 1.3 does not assume (\ref{KEY}), the new function $w:=z+\gamma_1z^2+\gamma_3 z^3+\gamma_3 z^4$ now solves
\begin{align}\label{bbss0}
 i\partial_t w+\Delta w=c\bar{w}(\partial_x w)^2+\nu_2\bar{w}w(\partial_x w  )^2+\nu_3\overline{w^2}(\partial_x w  )^2+O(|w|^3)(\partial_x w  )^2,
\end{align}
where $c=-\frac{1}{2} K(Q)h_0$.

First, we prove an abstract result which ensures the existence of wave operators under the assumption that there exists  a good approximate solution.
\begin{Lemma}\label{JJJI}
Let $\mathcal{N}$ be a  Riemannian surface. Given $Q\in \mathcal{N}$, denote $c=-\frac{1}{2} K(Q)h_0$. Assume that for some sufficiently large $m$, sufficiently small $\epsilon_*$ and  some $\nu>0$, there exists  a function $v$ satisfying
\begin{align}
\|i\partial_t v+\Delta v-(\partial_x v)^2(c\bar{v}+\nu_2\bar{v}v+\nu_3\overline{v^2})\|_{L^2_x\cap L^{\infty}_x}&\lesssim t^{-\frac{3}{2}-\nu}\label{Jkll} \\
\|v(t)\|_{ H^{m}_x}&\lesssim \epsilon_*\label{8ij}\\
\|v(t)\|_{ W^{1,\infty}_x}&\lesssim t^{-\frac{1}{2}}\epsilon_* \label{9ij}
\end{align}
 for all $t\ge N_0\ge 1$. Then there exist a constant $\theta\in (\frac{1}{2},1)$ and  an initial data $u_0$  evolving  to a global solution of (\ref{hia3})  so that
 \begin{align}\label{ooasy3}
\sup_{t\ge N_0} \langle t\rangle^{\theta} \|w  (t)-v(t,x)\|_{L^2_x} \lesssim 1,
\end{align}
where $w$ is the well chosen local complex coordinate near $Q$ such that $w(Q)=0$ and (\ref{bbss0}) holds.
\end{Lemma}
\begin{proof}
{\bf Step 1.} For simplicity, write the $O(|w|^3)(\partial_x w  )^2$ in (\ref{bbss0}) as $K(w)$.
Given $N\in \Bbb Z_+$, consider the equation
\begin{align}\label{bbb0}
w_{N}(t) =v(t,x)+i\int^{N}_{t}e^{i(t-\tau)\Delta}[(c\overline{w_{N}}+\nu_2\overline{w_{N}}w_{N}+\nu_3\overline{w^2_{N}})(\partial_xw_{N})^2+ K(w_{N})-(i\partial_{\tau}v+\Delta v)]d\tau.
\end{align}
It is easy to check $w_{N}(t)$ solves the local equation (\ref{equation1}) with initial data $w_{N}(N,x)=v(N,x)$.
By the assumption (\ref{8ij}),  for $N\ge N_0$
$$\|w_{N}(N,x)\|_{H^{m}_x}= \|v(N,x)\|_{H^{m}_x}\le \epsilon_*.$$
Thus by Section 2,  $w_{N}(t)$ is a global solution of 1D SMF in $t\in \Bbb R$.

{\bf Step 2.}
For simplicity of notations, we drop $N$ and write $w$ instead of $ w_{N}$.
Let
$$R(v)=i\partial_t v+\Delta v-(c \bar{v}+\nu_2\bar{v}v+\nu_3\overline{v^2})(\partial_x v)^2.$$
Rewrite (\ref{bbb0}) as
\begin{align*}
w(t)&=v(t,x)+i\int^{N}_{t}e^{i(t-\tau)\Delta}[K(w)+R(v) +  G (\tau)]d\tau\\
G &= 2c \bar{v}\partial_x(w-v)\partial_xv+  c \overline{w-v}\partial_xv\partial_xv+c  \overline{z-v}\partial_x(w-v)\partial_xv\\
&+2\nu_2 \bar{v}v\partial_x(w-v)\partial_xv+  \nu_2 \overline{w-v}v\partial_xv\partial_xv+\nu_2 \bar{v} {(w-v)}\partial_x w\partial_xv\\
&+2\nu_3  \overline{v^2}\partial_x(w-v)\partial_xv+  2\nu_3 \overline{w-v}\bar{v}\partial_xv\partial_xv+\widetilde{\mathbf{R}}_2  +\widetilde{\mathbf{R}}_3+\widetilde{\mathbf{R}}_4,
\end{align*}
where $\widetilde{{\bf R}}_2$, $\widetilde{{\bf R}}_3$, $\widetilde{{\bf R}}_4$ denote quadratic, cubic and respectively 4 order terms in $w-v$.
Write $G$ as
\begin{align*}
G&=2c \bar{v}\partial_x(z-v)\partial_xv+2\nu_2 \bar{v}v\partial_x(w-v)\partial_xv+2\nu_3  \overline{v^2}\partial_x(w-v)\partial_xv+\mathcal{G}.
\end{align*}
Note that the key troublesome derivative loss occurs in
$$2c_0\bar{v}\partial_x(z-v)\partial_xv+2\nu_2 \bar{v}v\partial_x(w-v)\partial_xv+2\nu_3  \overline{v^2}\partial_x(w-v)\partial_xv,
$$
and we will see the derivative loss in $\mathcal{G}$ is harmless since it is at least quadratic in $w-v$. Let
\begin{align*}
\mathcal{W}&=\int_{\Bbb R} W(v,\bar{v})|w-v|^2dx.
\end{align*}
Then we find
\begin{align*}
\frac{d}{dt}\mathcal{W}&=\int_{\Bbb R}(W_v \partial_t v+W_{\bar{v}}\partial_t\bar{v})|w-v|^2dx+\int_{\Bbb R} W [\overline{w-v}\partial_t (w-v)+(w-v)\partial_t \overline{{w-v}}]d x,
\end{align*}
which further expands as
\begin{align}
&\int_{\Bbb R} (iW_{v}\Delta v-iW_{\bar{v}}\Delta\bar{v})|w-v|^2dx+\int_{\Bbb R} W  [\overline{w-v}i\Delta (w-v)-i(w-v)\Delta \overline{{w-v}}]dx\label{S1}\\
&+\int_{\Bbb R} W(-iR(v)+i\overline{R(v)})|w-v|^2dx\nonumber\\
&+\int_{\Bbb R} W [-i\overline{w-v}(K+R(v)+\mathcal{G}) +i(w-v)  \overline{(K+R(v)+\mathcal{G})}]dx\nonumber\\
&+\int_{\Bbb R} W    [-2ic\overline{w-v}( \bar{v}\partial_xv\partial_x (w-v)) +2ic(w-v)   {  {v}\partial_x\overline{v}\partial_x{ (\overline{w-v})}}]dx\label{S2}\\
&+\int_{\Bbb R} -iW \overline{(w-v)}  \left [2\nu_2|v|^2\partial_x({w-v}) \partial_xv +2\nu_3 \overline{v^2}\partial_xv\partial_x({w-v})\right]dx\label{S3}\\
&+\int_{\Bbb R} iW {(w-v)}  \left [2\overline{\nu_2}|v|^2\partial_x(\overline{w-v}) \partial_x\bar{v} +2\overline{\nu_3} {v^2}\partial_x\bar{v}\partial_x(\overline{w-v})\right]dx.\label{S4}
\end{align}
By integration by parts,
\begin{align*}
 &\int_{\Bbb R} (W_{v}i\bar{v}\Delta v-iW_{\bar{v}}\Delta\bar{v})|w-v|^2dx+\int_{\Bbb R} W  [\overline{w-v}i\Delta (w-v)-i(w-v)\Delta \overline{{w-v}}]dx\\
&=-\int_{\Bbb R}(iW_v\bar{v}\partial_x v-iW_{\bar{v}}\partial_x\bar{v})[(w-v)\partial_x(\overline{w-v})+(\overline{w-v})\partial_x( {w-v})]dx\\
&-\int_{\Bbb R} [(iW_{v}\partial_x v+iW_{\bar{v}})\partial_x\bar{v})\overline{(w-v)} \partial_x (w-v)-i(W_{v}\partial_x v+iW_{\bar{v}})(w-v)\partial_x \overline{{(w-v)}}]  dx\\
&-\int_{\Bbb R} [i(\partial_xW_{v})\partial_x v-i(\partial_xW_{\bar{v}})\partial_x\bar{v}]|w-v|^2dx.
\end{align*}
Letting
\begin{align*}
W(z,\bar{z})=1-c|z|^2-\frac{1}{2}\nu_2z^2\bar{z}-\nu_3\bar{z}^2z,
\end{align*}
from Lemma 3.1, we see
\begin{align*}
W(z,\bar{z})(-2c\bar{z}-2\nu_2\bar{z}z-2\nu_{3}\bar{z}^2)-2W_{z}=O(|z|^3),
\end{align*}
and find that
 \begin{align*}
&|(\ref{S1})+(\ref{S2})+(\ref{S3})+(\ref{S4})|\\
&\lesssim \int_{\Bbb R} |\partial_x(w-v)||w-v||v|^3|\partial_x v|dx+\int_{\Bbb R} |(\partial_x W_v)\partial_x v||w-v|^2dx.
\end{align*}
Hence we arrive at
\begin{align*}
&\frac{d}{dt}\mathcal{W}\lesssim \int_{\Bbb R}  |-iW_{v}R(v)+iW_{\bar{v}}\overline{R(v)}||w-v|^2dx+\int_{\Bbb R} W  \left| (w-v) (K+R(v)+\mathcal{G}) \right|dx\\
&+\int_{\Bbb R} |\partial_x (w-v)||w-v||v|^3| \partial_x v|dx+\int_{\Bbb R} |\partial_x W_{v} \partial_x v||w-v|^2dx.
\end{align*}
Since $\|v\|_{L^{\infty}_{t,x}}\lesssim 1$, we see
 \begin{align*}
W\sim 1.
\end{align*}
Thus we conclude
\begin{align*}
 |\frac{d}{dt}\mathcal{W}|\lesssim &  \|R(v)\|^2_{L^{\infty}_x} \mathcal{{W}}
 +\|{w-v}\|_{L^2_x}[ \|K\|_{L^2_x}+\|R(v)\|_{L^2_x}+\|\mathcal{G}\|_{L^2_x}]
 + \mathcal{{W}} \| {v}\|^2_{L^{\infty}_x}\|\partial_x v \|^2_{L^{\infty}_x}\\
&+\|w-v\|_{L^2_x}\|w-v\|_{\dot{H}^1_x}\|v\|^4_{W^{1,\infty}_x}+\mathcal{W}\|v\|^2_{W^{1,\infty}_x}.
\end{align*}
Noting that $\|w-v\|_{{\dot H}^1}\lesssim \epsilon$ by energy conservation, we further have
\begin{align*}
 |\frac{d}{dt}\mathcal{W}|\lesssim &  \|R(v)\|^2_{L^{\infty}_x} \mathcal{{W}}
 +\sqrt{\mathcal{{W}}}[ \|K\|_{L^2_x}+\|R(v)\|_{L^2_x}+\|\mathcal{G}\|_{L^2_x}]
 + \mathcal{{W}} \| {v}\|^2_{L^{\infty}_x}\|\partial_x v \|^2_{L^{\infty}_x}\\
&+\sqrt{\mathcal{{W}}} \|v\|^4_{W^{1,\infty}_x}+\mathcal{W}\|v\|^2_{W^{1,\infty}_x}.
\end{align*}

{\bf Step 3.}
Take $\theta\in (\frac{1}{2},\frac{1}{2}+\nu)$ and $m$ large to fulfill
\begin{align}\label{WWP}
- \theta+\frac{9}{2m}(\theta+1)+\frac{1}{2}<0.
\end{align}
Assume that $S\in [N_0,N]$ is the smallest time such that
\begin{align*}
\sup_{t\in [S, N]} [\langle t\rangle^{\theta} \|w(t)-v(t)\|_{L^2_x}+ \langle t\rangle^{-1} \|w(t)-v(t)\|_{\dot{H}^1_x\cap {\dot H}^m_x} ]\le \epsilon.
\end{align*}
Since $w(N)=v(N)$, $S$ is well defined.

{\bf Step 3.1. Slow growth of Sobolev norms.}
By Gagliardo-Nirenberg inequality, we have for $t\in [S,N]$
\begin{align}
\|  w(t)-v(t) \|_{L^{\infty}_x}&\lesssim \|w(t)- v(t)\|^{1-\frac{1}{2m}}_{L^2_x}\|\partial^m_x[w(t)- v(t)]\|^{\frac{1}{2m}}_{L^2_x} \lesssim \langle t\rangle^{-\theta+\frac{1}{2m}(\theta+1)}\epsilon\lesssim \epsilon\label{Ty}\\
\|\partial_{x}[w(t)-v(t)]\|_{L^{\infty}_x}  &\lesssim  \|w(t)- v(t)\|^{1-\frac{3}{2m}}_{L^2_x}\|\partial^m_x[w(t)- v(t)]\|^{\frac{3}{2m}}_{L^2_x} \lesssim \langle t\rangle^{-\theta+\frac{3}{2m}(\theta+1)}\epsilon\lesssim \langle t\rangle^{-\frac{1}{2}} \epsilon. \label{Uy}
\end{align}
Since for $t\ge N_0$, $v(t)$ also satisfies (\ref{9ij}),
we obtain  by (\ref{Uy}) and (\ref{Ty}) that for $t\in [S,N]$
\begin{align*}
\|\partial_x w(t) \|_{L^{\infty}_x} &\lesssim \epsilon t^{-\frac{1}{2}} \\
\| w(t) \|_{L^{\infty}_x} &\lesssim  \epsilon.
\end{align*}
Transferring this bound to $u$, one has
\begin{align*}
  \sup_{t\in [S,N]} t^{\frac{1}{2}}\|\partial_x u(t) \|_{L^{\infty}_x} &\lesssim \epsilon.
\end{align*}
Thus Section 2.1 implies for all $t\in [S,N]$, $k\in[1,m]$,
\begin{align*}
 \| u(t) \|_{W^{k,2}(\Bbb R;\mathcal{N})}  \lesssim \epsilon_* (1+t)^{\epsilon},
\end{align*}
which further gives
\begin{align*}
 \sup_{t\in [S,N]}(1+t)^{-\epsilon} \| w(t) \|_{\dot{H}^1_x\cap \dot{H}^m_x}  \lesssim \epsilon_*.
\end{align*}
Hence, we conclude for Step 3.1 that
\begin{align*}
 \sup_{t\in [S,N]}(1+t)^{-1} \| wz(t) -v(t)\|_{\dot{H}^1_x\cap \dot{H}^m_x}  \lesssim
 \sup_{t\in [S,N]}(1+t)^{-\epsilon} \| w(t) \|_{\dot{H}^1_x\cap \dot{H}^m_x}  \lesssim
 \epsilon_*\ll \epsilon.
\end{align*}
{\bf Step 3.2.}
Let's bound the inhomogeneous term. For $t\in [S,N]$, (\ref{Ty}), (\ref{Uy}) and bootstrap assumption imply
\begin{align*}
 \|K(w)\|_{L^2_x}  &\lesssim \|w\|^4_{W^{1,\infty}_x}\|\partial_x w\|_{L^2_x}\lesssim \epsilon^5 t^{-2}\\
  \|\mathcal{G} \|_{L^2_x}   &\lesssim   \|w-v\|_{W^{1,\infty}_x}\|w-v\|_{L^{2}_x}\|v\|_{W^{1,\infty}_x} + \|v\|^2_{W^{1,\infty}_x}\|w-v\|_{L^2_x} +\|w-v\|^2_{W^{1,\infty}_x}\|w-v\|_{L^2_x}\\
& + \|v\|^2_{W^{1,\infty}_x}\|w-v\|^2_{L^2_x}+
 \|v\|^2_{W^{1,\infty}_x}\|w-v\|_{L^2_x}\|w-v\|_{W^{1,\infty}_x}\\
&+\|w-v\|^2_{W^{1,\infty}_x}\|w-v\|_{L^{2}_x}\|v\|_{W^{1,\infty}_x}
+\|w-v\|^3_{W^{1,\infty}_x}\|w-v\|_{L^2_x}\\
 &\lesssim  \epsilon^3 t^{-\frac{1}{2}-2\theta+\frac{3}{2m}(\theta+1)} + \epsilon^3 t^{- \theta-1} +\epsilon^3t^{-3\theta+\frac{9}{2m}(\theta+1)},
\end{align*}
 where we also used the conservation of energy.
Then for $t\in [S,N]$, we get
\begin{align*}
 &\|w(t)-v(t,x)\|^2_{L^2_x}\sim \mathcal{W}(t)\\
 &\lesssim \int^{N}_{t}[\|R(v)\|_{L^2_x}+\|K(w)\|_{L^2_x}+   \|\mathcal{G} \|_{L^2_x}]\|w-v\|_{L^2_x}d\tau\\
& +\int^{N}_t\|R(w)\|^2_{L^{\infty}_x}\|w-v\|^2_{L^2_x}d\tau+\int^{N}_{t}  \| {v}\|^2_{L^{\infty}_x}\|\partial_x v \|^2_{L^{\infty}_x}\|w-v\|^2_{L^{2}_x}d\tau\\
&\lesssim \int^{N}_{t} [  \tau^{-\theta-\frac{3}{2}-\nu}\epsilon_* + \epsilon^6{\tau}^{-2-\theta}+\epsilon^4 {\tau}^{-\frac{1}{2}-3\theta+\frac{3}{2m}(\theta+1)} + \epsilon^4 {\tau}^{- 2\theta-1} +\epsilon^4{\tau}^{-4\theta+\frac{9}{2m}(\theta+1)}]d\tau\\
&+\int^{N}_{t} [ \epsilon^4 {\tau}^{- 2\theta-3-2\nu} +\epsilon^6{\tau}^{-2-2\theta}]d\tau\\
&\lesssim \epsilon_*t^{-\theta-\frac{1}{2}-\nu}  +\epsilon^6  t^{-\theta-1}+ \epsilon^4  t^{-3\theta+\frac{3}{2m}(\theta+1)+\frac{1}{2}}+\epsilon^4  t^{-4\theta+\frac{9}{2m}(\theta+1)+1}+\epsilon^4 t^{-2\theta}.
\end{align*}
Therefore, taking $\theta\in (\frac{1}{2},\frac{1}{2}+\nu)$ and letting $m$ be sufficiently large  to satisfy
(\ref{WWP}), we arrive at
\begin{align*}
\sup_{t\in [S,N]}\langle t\rangle^{\theta} \|w(t)-v(t,x)\|_{L^2_x}&\le  \epsilon_*+ C\epsilon^2.
\end{align*}
Then by conclusion of Step 3.1 and  bootstrap, $S=N_0$ if $\epsilon$ is taken to be sufficiently small such that $C\epsilon<0.1$.

{\bf Step 4.} In Step 3, for each given $N\ge N_0$, we have constructed $w_{N}(t)$ such that
\begin{align*}
\sup_{t\in [N_0,N]}\langle t\rangle^{\theta} \|w_{N}(t)-v(t,x)\|_{L^2_x}&\le   C\epsilon^3\\
\|w_{N}(t)\|_{H^{m}_x}&\lesssim (1+t)^{\epsilon}.
\end{align*}
By compactness, $w_{N}(t)$ converges to some function $w(t) \in C([0,T];H^{m})$ strongly in $C([0,T];H^{m-1})$ for any given $T\in \Bbb R^+$.  By Sobolev embedding, $w$ is at least $C^2$ and solves the SMF equation point-wisely. Moreover, $w(t)$ satisfies
 \begin{align*}
\sup_{t \ge N_0}\langle t\rangle^{\theta} \|w  (t)-v(t,x)\|_{L^2_x}&\le   C\epsilon^3,
\end{align*}
from which  (\ref{ooasy3})  follows.
\end{proof}

The following four lemmas together prove  the existence of approximate solution $w_{ap}$. We begin with the first time correction.
\begin{Lemma}\label{a8i}
Assume that
\begin{align}\label{xxx}
\sum_{0\le j\le 2}\|\langle x\rangle  \psi^{(j)} (x)\|_{L^{2}_x\cap L^{\infty}_x} \lesssim \epsilon.
\end{align}
Let
$$v_1(t,x):=\frac{1}{(2it)^{\frac{1}{2}}}\psi(\frac{x}{2t})e^{i\frac{|x|^2}{4t}}e^{ \frac{ic}{2}| \frac{x}{2t} \psi(\frac{x}{2t})|^2\ln (2t)}.$$
Then  we have
\begin{align}
\|i\partial_t v_1+\Delta v_1-(\partial_x v)^2 c\bar{v_1} \|_{L^2_x\cap L^{\infty}_x}&\lesssim t^{-2}(\ln t)^2.\label{8b}
\end{align}
\end{Lemma}
\begin{proof}
The proof is a straightforward calculation. To avoid  the formulas covering too much space, we introduce the notation
\begin{align*}
\Psi(x,t):= \psi'(\frac{x}{2t})\bar{\psi}(\frac{x}{2t})+ \bar{\psi}'(\frac{x}{2t})\psi(\frac{x}{2t}).
\end{align*}
For $\partial_t v$, one has
\begin{align*}
&\partial_t[\frac{1}{(2it)^{\frac{1}{2}}}\psi(\frac{x}{2t})e^{i\frac{|x|^2}{4t}}e^{ \frac{ic}{2}| \frac{x}{2t} \psi(\frac{x}{2t})|^2\ln (2t)}]\\
&=\frac{1}{(2it)^{\frac{1}{2}}}e^{i\frac{|x|^2}{4t}}e^{ \frac{ic}{2}| \frac{x}{2t} \psi(\frac{x}{2t})|^2\ln (2t)}[-\frac{1}{2}t^{-1}\psi(\frac{x}{2t})-\frac{x}{2t^2}\psi'(\frac{x}{2t})-i\frac{|x|^2}{4t^2}\psi(\frac{x}{2t})]\nonumber\\
&+\frac{1}{(2it)^{\frac{1}{2}}}e^{i\frac{|x|^2}{4t}}e^{ \frac{ic}{2} | \frac{x}{2t} \psi(\frac{x}{2t})|^2\ln (2t)}\psi(\frac{x}{2t})[-\frac{ic}{2} \frac{x^2}{2t^3} | \psi(\frac{x}{2t})|^2\ln (2t)
-\frac{ic}{2}  \frac{x^3}{8t^4}\Psi(x,t)\ln(2t) ]\\
&+\frac{1}{(2it)^{\frac{1}{2}}}e^{i\frac{|x|^2}{4t}}e^{ \frac{ic}{2} | \frac{x}{2t} \psi(\frac{x}{2t})|^2\ln (2t)}\psi(\frac{x}{2t})\frac{ic}{2t} |\frac{x}{2t}|^2| \psi(\frac{x}{2t})|^2.
 \end{align*}
And $\partial^2_xv$ is given by
 \begin{align*}
&\partial^2_x[\frac{1}{(2it)^{\frac{1}{2}}}\psi(\frac{x}{2t})e^{i\frac{|x|^2}{4t}}e^{ \frac{ic}{2}| \frac{x}{2t} \psi(\frac{x}{2t})|^2\ln (2t)}]\\
&=\frac{1}{(2it)^{\frac{1}{2}}}e^{i\frac{|x|^2}{4t}}e^{ \frac{ic}{2}| \frac{x}{2t} \psi(\frac{x}{2t})|^2\ln (2t)} \frac{1}{4t^2}\psi''(\frac{x}{2t}) \nonumber\\
&+\frac{1}{(2it)^{\frac{1}{2}}}e^{i\frac{|x|^2}{4t}}e^{ \frac{ic}{2} | \frac{x}{2t} \psi(\frac{x}{2t})|^2\ln (2t)}   \frac{1}{ t}\psi'(\frac{x}{2t})[\frac{ix}{2t}+ \frac{ic}{2}  \frac{x}{2t^2} |\psi(\frac{x}{2t})|^2\ln (2t)]\\
 &+\frac{1}{(2it)^{\frac{1}{2}}}e^{i\frac{|x|^2}{4t}}e^{ \frac{ic}{2} | \frac{x}{2t} \psi(\frac{x}{2t})|^2\ln (2t)}  \frac{1}{ t}\psi'(\frac{x}{2t})[\frac{ic}{2}\frac{x^2}{8t^3}\ln (2t)\Psi(x,t)   ]\nonumber\\
&+\frac{e^{i\frac{|x|^2}{4t}+\frac{ic}{2} | \frac{x}{2t} \psi(\frac{x}{2t})|^2\ln (2t)} }{(2it)^{\frac{1}{2}}} \psi(\frac{x}{2t})\left[\frac{ix}{2t}+ \frac{ic}{2}  \frac{x}{2t^2} |\psi(\frac{x}{2t})|^2\ln (2t)
 +\frac{ic}{2}\frac{x^2}{8t^3}\ln (2t) \Psi(x,t)  \right]^2\nonumber\\
 &+\frac{e^{i\frac{|x|^2}{4t}+ \frac{ic}{2} | \frac{x}{2t} \psi(\frac{x}{2t})|^2\ln (2t)}}{(2it)^{\frac{1}{2}}}  \psi(\frac{x}{2t})\left[\frac{i }{2t}+\partial_x\left(\frac{ic}{2}  \frac{x}{2t^2} |\psi(\frac{x}{2t})|^2\ln (2t) +\frac{ic}{2}\frac{x^2}{8t^3}\ln (2t)  \Psi(x,t)\right)\right ]. \nonumber
 \end{align*}
By computation, $\bar{v}(\partial_xv)^2$ is
 \begin{align*}
&c\bar{v}(\partial_xv)^2=\frac{1}{2t}\frac{c}{(2it)^{\frac{1}{2}}}{e^{ i\frac{|x|^2}{4t}+ \frac{ic}{2} | \frac{x}{2t} \psi(\frac{x}{2t})|^2\ln (2t)}}\psi(\frac{x}{2t})|\psi(\frac{x}{2t})|^2 [\frac{ix}{2t}+ {O}_{L^{\infty}_x} (\frac{\ln t}{t} )]^2.
 \end{align*}
Therefore,
 \begin{align*}
&|i\partial_tv_1+\partial^2_xv_1-c\overline{v_1}(\partial_xv_1)^2|\\
&\lesssim t^{-\frac{5}{2}}\sum_{l=1,2,3} \sum_{j=0,1,2}  (1+|\ln t|)^2(1+(\frac{x}{2t})^{l})|\psi^{(j)}(\frac{x}{2t}))|^{l}.
 \end{align*}
By (\ref{xxx}),  $v_1$ satisfies (\ref{8b}).
\end{proof}

The following lemma is for constructing the second time correction of approximate solution.
\begin{Lemma}\label{b8i}
Let
\begin{align*}
 v_2(x,t)=\frac{i\nu_2 }{8t^2}|\psi(\frac{x}{2t})|^2 \psi^2(\frac{x}{2t})e^{i\frac{ x^2}{2t}+i\frac{ c x^2}{4t^2}|\psi(\frac{x}{2t})|^2\ln (2t)}.
\end{align*}
Then we have
\begin{align*}
 (i\partial_t+\Delta)(v_1+v_2)-c\overline{ v_1 }(\partial_x v_1)^2-\nu_2\overline{ v_1 }v_1(\partial_x v_1)^2=O_{L^{\infty}_x}( {t^{-\frac{5}{2}}(\ln t)^2})\cap O_{L^{2}_x}({t^{ -2}(\ln t)^2}).
\end{align*}
\end{Lemma}
\begin{proof}
Let
$$
v_2=\frac{1}{t^{k}}\tilde{U}(\frac{x}{2t})e^{i\frac{\gamma x^2}{4t}+i\frac{\gamma c x^2}{8t^2}|\psi(\frac{x}{2t})|^2\ln (2t)}.
$$
By Lemma \ref{a8i}, it suffices  to   find $k,\gamma$ and $\tilde{U}$ such that
\begin{align*}
 (i\partial_t+\Delta)v_2-\nu_2\overline{ v_1 }v_1(\partial_x v_1)^2=O_{L^{\infty}_x}({t^{-\frac{5}{2}}}(\ln t)^2)\cap O_{L^{2}_x}({t^{ -2}}(\ln t)^2).
\end{align*}
For the sake of simplicity, we denote
\begin{align*}
\Phi:=&e^{i\frac{\gamma x^2}{4t}+i\frac{\gamma c x^2}{8t^2}|\psi(\frac{x}{2t})|^2\ln (2t)}\\
\Psi:=&\psi'(\frac{x}{2t})\overline{\psi(\frac{x}{2t})}+\psi(\frac{x}{2t})\overline{\psi'(\frac{x}{2t})}.
\end{align*}
By computation,
\begin{align*}
&i\partial_t v_2=-i\frac{k}{t^{k+1}}\tilde{U}(\frac{x}{2t})\Phi+\frac{1}{t^{k}}\tilde{U}'(\frac{x}{2t})(-\frac{x}{2t^2})\Phi\\
&+\frac{i}{t^k}\tilde{U}(\frac{x}{2t})\Phi\left(-\frac{i\gamma x^2}{4t^2}-\frac{i\gamma c x^2}{ 4t^3}|\psi(\frac{x}{2t})|^2\ln (2t)-\frac{ic\gamma x^3}{16t^4}\ln (2t)\Psi+\frac{ic\gamma x^2}{8t^3}|\psi(\frac{x}{2t})|^2\right).
\end{align*}
And we have
\begin{align*}
&\partial^2_x v_2=\frac{1}{4t^{k+2}}\tilde{U}''(\frac{x}{2t})\Phi+\frac{1}{t^{k+1}}\tilde{U}'(\frac{x}{2t})\left(\frac{i\gamma x}{2t}+\frac{ic\gamma x}{4t^2}|\psi(\frac{x}{2t})|^2\ln( 2t)+\frac{ic\gamma x^2}{16t^3}\Psi\ln(2t)\right)\Phi\\
&+\frac{1}{t^{k}}\tilde{U}(\frac{x}{2t})\left[\frac{i\gamma}{2t}+\left(\frac{ic\gamma}{4t^2}|\psi(\frac{x}{2t})|^2+\frac{ic\gamma x}{4t^3}\Psi+\frac{ic\gamma x^2}{16t^3}\partial_x\Psi\right)\ln(2t)\right]\Phi.
\end{align*}
Note that
\begin{align*}
&i\partial_tv_2+\partial^2_x v_2=\frac{\gamma x^2}{4t^{k+2}}\tilde{U}(\frac{x}{2t})\Phi+O_{L^{\infty}_x} ({t^{-(k+1)}}(\ln t)^2)\cap O_{L^{2}_x}({t^{-(k+\frac{1}{2})}}(\ln t)^2).
\end{align*}
Recall also that
\begin{align*}
&\nu_2v_1\overline{v_1}(\partial_x v_1)^2=e^{ \frac{ ix^2}{2t}+i\frac{ ic x^2}{4t^2}|\psi(\frac{x}{2t})|^2\ln (2t)}\frac{\nu_2}{4it^2}|\psi(\frac{x}{2t})|^2\psi^2(\frac{x}{2t})(i\frac{x}{2t})^2+O_{L^{\infty}_x}({t^{-\frac{5}{2}}}(\ln t)^2)\cap O_{L^{2}_x}( {t^{-2}}(\ln t)^2).
\end{align*}
Let $\gamma=2$,   $k=2$, and $\tilde{U}$ be
\begin{align*}
\tilde{U}(y)= \frac{i\nu_2}{8}|\psi(y)|^2 \psi^2(y).
\end{align*}
Then we get
\begin{align*}
 (i\partial_t+\Delta)(v_1+v_2)-c\overline{ v_1 }(\partial_x v_1)^2-\nu_2\overline{ v_1 }v_1(\partial_x v_1)^2=O_{L^{\infty}_x}({t^{-\frac{5}{2}}}(\ln t)^2)\cap O_{L^{2}_x}( {t^{-2}}(\ln t)^2).
\end{align*}
\end{proof}

The following lemma is the third time correction of approximate solution.
\begin{Lemma}\label{c8i}
Let
$$
v_3=\frac{1}{t}\tilde{Q}(\frac{x}{2t}),
$$
where
\begin{align}\label{ii}
\tilde{Q}(y)=-\frac{1}{y}\int^y_0 \frac{i\nu_3}{4}|\psi(\tilde{y})|^4\tilde{y}^2d\tilde{y}.
\end{align}
Then we have
\begin{align*}
& (i\partial_t+\Delta)(v_1+v_2+v_3)-c\overline{ v_1 }(\partial_x v_1)^2-\nu_2\overline{ v_1 }v_1(\partial_x v_1)^2-\nu_3\overline{ v^2_1 }(\partial_x v_1)^2\\
&=O_{L^{\infty}_x}({t^{-\frac{5}{2}}}(\ln t)^2)\cap O_{L^{2}_x}( {t^{-2}}(\ln t)^2).
\end{align*}
\end{Lemma}
\begin{proof}
By Lemma \ref{b8i}, it suffices to prove
\begin{align*}
(i\partial_t+\Delta) v_3 -\nu_3 \overline{v^2_1}(\partial^2_x v_1)^2= O_{L^{\infty}_x}({t^{-\frac{5}{2}}}(\ln t)^2)\cap O_{L^{2}_x}( {t^{-2}}(\ln t)^2).
\end{align*}
Recall also that
\begin{align*}
&\nu_3 \overline{v^2_1}(\partial^2_x v_1)^2= -\frac{\nu_3 x^2}{16t^4}|\psi(\frac{x}{2t})|^4 +O_{L^{\infty}_x}({t^{-\frac{5}{2}}}(\ln t)^2)\cap O_{L^{2}_x}( {t^{-2}}(\ln t)^2).
\end{align*}
By computation,
\begin{align*}
&i\partial_t v_3+\Delta v_3=-\frac{i}{t^2}\tilde{Q}(\frac{x}{2t})-\frac{ix}{2t^3} \tilde{Q}'(\frac{x}{2t})+O_{L^{\infty}_x}({t^{-\frac{5}{2}}}(\ln t)^2)\cap O_{L^{2}_x}( {t^{-2}}(\ln t)^2).
\end{align*}
Observe that $\tilde{Q}$ solves
\begin{align*}
y\tilde{Q}'(y)+\tilde{Q}=- \frac{\nu_3i}{4}|\psi(y)|^4y^2.
\end{align*}
Then we get the desired result.
\end{proof}

The following lemma is the forth time correction of approximate solution.
\begin{Lemma}\label{d8i}
Let
$$
v_4=\frac{1}{t^2}P(\frac{x}{2t})e^{ \frac{ ix^2}{2t}+i\frac{ ic x^2}{4t^2}|\psi(\frac{x}{2t})|^2\ln (2t)}
$$
where
\begin{align}\label{ii2}
P(y)=\frac{c}{4}  \overline{ \tilde{Q} (y)}\psi^2(y),
\end{align}
and $\tilde{Q}$ is defined by (\ref{ii}).
Then we have
\begin{align*}
& (i\partial_t+\Delta)(v_1+v_2+v_3+v_4)-c\overline{v_1 }(\partial_x v_1)^2+-c\overline{ v_3 }(\partial_x v_1)^2-\nu_2\overline{ v_1 }v_1(\partial_x v_1)^2-\nu_3\overline{ v^2_1 }(\partial_x v_1)^2\\
&=O_{L^{\infty}_x}({t^{-\frac{5}{2}}}(\ln t)^2)\cap O_{L^{2}_x}( {t^{-2}}(\ln t)^2).
\end{align*}
\end{Lemma}
\begin{proof}
Assume generally  that
\begin{align*}
v_4=\frac{1}{t^a}P(\frac{x}{2t})e^{ \frac{ ix^2}{4t}+i\frac{ ic x^2}{4t^2}|\psi(\frac{x}{2t})|^2\ln (2t)}
\end{align*}
By Lemma \ref{c8i}, it suffices to prove that $a=2 $ and $P$ defined by (\ref{ii2}) lead  to
\begin{align*}
(i\partial_t+\Delta) v_4 -c \overline{v_3}(\partial^2_x v_1)^2= O_{L^{\infty}_x}({t^{-\frac{5}{2}}}(\ln t)^2)\cap O_{L^{2}_x}( {t^{-2}}(\ln t)^2).
\end{align*}
Recall  that
\begin{align*}
c \overline{v_3}(\partial_x v_1)^2= \frac{c x^2}{8t^{4}} \psi^2(\frac{x}{2t}) \overline{ \tilde{Q} (\frac{x}{2t})}e^{ \frac{ ix^2}{2t}+i\frac{ ic x^2}{4t^2}|\psi(\frac{x}{2t})|^2\ln (2t)} +O_{L^{\infty}_x}({t^{-\frac{5}{2}}}(\ln t)^2)\cap O_{L^{2}_x}( {t^{-2}}(\ln t)^2).
\end{align*}
By computation,
\begin{align*}
&i\partial_tv_4+\partial^2_x v_4=\frac{ x^2}{2t^{a+2}}P(\frac{x}{2t})e^{ \frac{ ix^2}{2t}+i\frac{ ic x^2}{4t^2}|\psi(\frac{x}{2t})|^2\ln (2t)} +O_{L^{\infty}_x}({t^{-\frac{5}{2}}}(\ln t)^2)\cap O_{L^{2}_x}( {t^{-2}}(\ln t)^2).
\end{align*}
Observe that for $a=2$ and $P$ defined above, there holds
\begin{align*}
 \frac{ x^2}{2t^{a+2}}P(\frac{x}{2t})- \frac{c x^2}{8t^{4}} \psi^2(\frac{x}{2t}) \overline{ \tilde{Q} (\frac{x}{2t})}=0.
\end{align*}
Then we get the desired result.
\end{proof}

\begin{Lemma}\label{HH}
Assume that $\psi$ satisfies
(\ref{assx}) for $m$ large enough.
Let
\begin{align*}
v:=v_1+v_2+v_3+v_4.
\end{align*}
Then $v$ fulfills (\ref{Jkll}), (\ref{8ij}), (\ref{9ij}).
\begin{align}
\|i\partial_t v+\Delta v-(c\bar{v}+\nu_2\overline{v}v+\nu_3 \overline{v^2})(\partial_x v)^2\|_{L^2_x\cap L^{\infty}_x}&\lesssim t^{-\frac{5}{2}}(\ln t)^2\label{xJkll} \\
\|v(t)\|_{ H^{m}_x}&\le \epsilon_*\label{x8ij}\\
\|v(t)\|_{ W^{1,\infty}_x}&\le t^{-\frac{1}{2}}\epsilon_*. \label{x9ij}
\end{align}
\end{Lemma}
\begin{proof}
To prove (\ref{x9ij}) and (\ref{x8ij}), it suffices to check the explicit formulas of $\{v_i\}^4_{i=1}$. We remark that $\tilde{Q}$ defined by (\ref{ii}) indeed has no singularity at $y=0$ up to $H^{m}$ by expanding $\psi$ at zero.

Now, let's prove (\ref{xJkll}). By {Lemma} \ref{c8i}, it suffices to prove
\begin{align}
\sum_{j=2,3,4}  \|\overline{v_1}\partial_x v_1\partial_x v_j\|_{L^2_x\cap L^{\infty}_x} & \lesssim t^{-\frac{5}{2}}(\ln t)^2 \label{aq8ij}\\
  \|\overline{v_4}\partial_x v_1\partial_x v_1\|_{L^2_x\cap L^{\infty}_x}  &\lesssim t^{-\frac{5}{2}} (\ln t)^2\label{bq8ij}\\
 \sum_{{\rm med}(i,j,l)\ge 2} \|\overline{v_i}\partial_x v_j\partial_x v_l\|_{L^2_x\cap L^{\infty}_x}  &\lesssim t^{-\frac{5}{2}}(\ln t)^2 \label{cq8ij}\\
\sum_{{\rm max}(i,j,l,k)\ge 2} \| | v_i v_{k} \partial_x v_j\partial_x v_l\|_{L^2_x\cap L^{\infty}_x}  &\lesssim t^{-\frac{5}{2}}(\ln t)^2 \label{dq8ij}.
\end{align}
(\ref{cq8ij}), (\ref{bq8ij}) (\ref{dq8ij}) are easy to verify by directly applying the decay of $\{v_i\}^4_{i=1}$. The relatively troublesome term in the LHS of  (\ref{aq8ij}) is
\begin{align*}
 \|\overline{v_1}\partial_x v_1\partial_x v_3\|_{L^2_x\cap L^{\infty}_x}.
\end{align*}
But notice that $\partial_x v_3$ decays as $t^{-2}$ in $L^{\infty}_x$ which is faster than $v_3$. So
\begin{align*}
 \|\overline{v_1}\partial_x v_1\partial_x v_3\|_{L^2_x\cap L^{\infty}_x}\lesssim t^{-\frac{5}{2}}.
\end{align*}

\end{proof}

Now, we are ready to prove Theorem 1.3.
Choose the approximate solution $w_{ap}$ to be $w_{ap}=v_1+v_2+v_3+v_4$.
 By Lemma \ref{HH} and  Lemma \ref{JJJI}, there exists an initial data $u_0$ evolving to a global solution of 1D SMF for which
\begin{align}\label{1po}
 \|w(t)-w_{ap}\|_{L^2_x}\lesssim t^{-\theta},
\end{align}
where $\theta>\frac{1}{2}$.
Meanwhile, Step 4 of Lemma \ref{JJJI} also shows
\begin{align*}
 \|w(t)\|_{H^m_x}\lesssim t^{\epsilon}.
\end{align*}
Then by (\ref{8ij}),
\begin{align}\label{2po}
 \|w(t)-v(t)\|_{H^m_x}\lesssim t^{\epsilon}.
\end{align}
By interpolation (\ref{2po}) with (\ref{1po}), one further obtains
\begin{align*}
 \lim_{t\to \infty} \|w(t)-v(t)\|_{H^1_x}=0.
\end{align*}
Therefore,  using
$$\|w-z\|_{H^1_x}\lesssim \|w\|^2_{H^1_x}\lesssim t^{-\frac{1}{2}},$$
and
$$\sum_{i=2,3,4}\|v_i\|_{H^1_x}\lesssim t^{-\frac{1}{2}},$$
 we get (\ref{asy3}).

\section*{Appendix A}

We concretely explain the  inconvenience of applying caloric gauge in our problem. As discussed in Section 1.1, the differential fields $\psi_x$ solves (1.18). There are two choices to set up the bootstrap of time decay estimates. One is to assume
$$(a) \mbox{ } \langle t\rangle^{\frac{1}{2}}\|\phi_x \|_{W^{k,\infty}_x} \le \varepsilon,$$
the other is to assume
$$(b) \mbox{ } \langle t\rangle^{\frac{1}{2}}[\|\phi_x \|_{W^{k,\infty}_x} +\|z \|_{W^{k,\infty}_x}]\le \varepsilon.$$
Since the equation of $z$ always suffers from the troublesome 4 order terms, in order to bound $\|z \|_{W^{k,\infty}_x}$, we can only try to use $\psi_x$. Observe that  $z$ and $\psi_x$ are related by
\begin{align}\label{Ghuk}
\|z(t,x)\|_{L^{\infty}_x}\lesssim \|u(t,x)-Q\|_{L^{\infty}_x}\lesssim \int^{\infty}_0 \|\psi_s\|_{L^{\infty}_x}ds\lesssim \int^{\infty}_0\|D_x\psi_x\|_{L^{\infty}_x}ds.
\end{align}
By the decay estimates of heat equation, one can only expect
$$
\int^{\infty}_0\|\partial_x\psi_x\|_{L^{\infty}_x}ds\lesssim \|z(t)\|_{L^{p}_x},
$$
for $p\in[1,\infty)$.
Note that $p=\infty$ is not possible since it will lead to an integral $\int^{\infty}_1\frac{1}{s}ds$.

Thus, if one assumes (b), the possible time decay of the RHS of (\ref{Ghuk}) shall be
\begin{align*}
\|z(t,x)\|_{L^{\infty}_x}\lesssim  t^{\frac{1}{p}-\frac{1}{2}},
\end{align*}
where $p$ could not take infinity. So the bootstrap of  $\langle t\rangle^{\frac{1}{2}}\|z \|_{W^{k,\infty}_x}$ could not  be closed,

Now, assume that we use assumption (a). Then, the weakest decay one needs, though far from being enough, of $\|A_x\partial_x \psi_x\|_{L^2_x}$ shall be $t^{-1}$. But  the integral
$$
\|(\int^{\infty}_0 |\psi_x||\partial_x\psi_x| ds)\partial_x \psi_x\|_{L^2_x}
$$
does not have enough decay in $s$ to provide $t^{-1}$ decay under the assumption (a). For example, let $p_1,p_2,p_3\in [2,\infty]$ with $\sum_{i}\frac{1}{p_i}=\frac{1}{2}$. Then the decay of heat equations suggests that
\begin{align*}
&\|(\int^{\infty}_1 |\psi_x||\partial_x\psi_x| ds)\partial_x \psi_x\|_{L^2_x}\lesssim \left(\int^{\infty}_1 \|\psi_x\|_{L^{p_1}_x} \|\partial_x \psi_x\|_{L^{p_2}_x}ds \right)\|\partial_x \psi_x\|_{L^{p_3}_x}\\
&\lesssim \left(\int^{\infty}_1 s^{-\frac{1}{2}}\|\psi_x\|_{L^{p_1}_x}  \| \psi_x\|_{L^{p_2}_x}ds \right)t^{-\frac{1}{2}(1-\frac{1}{p_3})}
\end{align*}
In order to reach $t^{-1}$ decay, we need to obtain $t^{-\frac{3}{4}+\frac{1}{2p_1}+\frac{1}{2p_2}}$ decay from
\begin{align*}
\|\psi_x\|_{L^{p_1}_x}  \| \psi_x\|_{L^{p_2}_x}.
\end{align*}
The decay estimate of heat equations is likely to give
\begin{align*}
\|\psi_x\|_{L^{p}_x} \lesssim s^{-\frac{1}{2}+\frac{1}{p}-\frac{1}{q}}\|z\|_{L^{q}}, \mbox{ }1\le q\le p\le \infty.
\end{align*}
And under the assumption (a) we see
 \begin{align*}
\|z\|_{L^{q}_x} \lesssim \|z\|^{ \frac{1}{2}+\frac{1}{q}}_{L^2_x}\|\psi_x\|^{\frac{1}{2}-\frac{1}{q}}_{L^{\infty}_x} \lesssim t^{-\frac{1}{4}+\frac{1}{2q}}.
\end{align*}
So we get
\begin{align*}
\|\psi_x\|_{L^{p}_x} \lesssim s^{-\frac{1}{2}+\frac{1}{p}-\frac{1}{q}}t^{-\frac{1}{4}+\frac{1}{2q}}, \mbox{ }1\le q\le p\le \infty.
\end{align*}
On the other hand directly applying assumption (a) also gives
\begin{align*}
\|\psi_x\|_{L^{p}_x}  \lesssim t^{-\frac{1}{4}+\frac{1}{2p}}.
\end{align*}
Hence,
\begin{align*}
\int^{\infty}_1 s^{-\frac{1}{2}}\|\psi_x\|_{L^{p_1}_x}  \| \psi_x\|_{L^{p_2}_x}ds\lesssim t^{-\frac{1}{2}+\frac{1}{2q_1}+\frac{1}{2q_2}}  \int^{\infty}_1 s^{-\frac{3}{2}+\frac{1}{p_1}-\frac{1}{q_1}+\frac{1}{p_2}-\frac{1}{q_2}}ds.
\end{align*}
In order to reach $t^{-\frac{3}{4}+\frac{1}{2p_1}+\frac{1}{2p_2}}$ decay and make the $s$ integral converge, it requires to set
\begin{align*}
 &-\frac{1}{2}+\frac{1}{2q_1}+\frac{1}{2q_2} \le -\frac{3}{4}+\frac{1}{2p_1}+\frac{1}{2p_2}\\
& -\frac{3}{2}+\frac{1}{p_1}-\frac{1}{q_1}+\frac{1}{p_2}-\frac{1}{q_2}<-1.
\end{align*}
However, these two inequalities are exactly opposite to each other. So it seems impossible to make  $\|A_x\partial_x \psi_x\|_{L^2_x}$ have  $t^{-1}$ decay.

From the above discussions, it seems that caloric gauge would not bring additional benefit  in our problem.


\begin{thebibliography}{999}



{\scriptsize
\bibitem{bv} V. Banica, L. Vega,  \emph{ Stability of the Self-similar Dynamics of a Vortex Filament}. Arch. Rational Mech. Anal.,  210(3), 673-712, 2013.


\bibitem{b} I. Bejenaru,  \emph{On Schr\"odinger maps}, Amer. J. Math., {\bf 130} (2008), 1033-1065.

\bibitem{bik} I. Bejenaru, A. Ionescu, and C. Kenig, \emph{Global existence and uniqueness of Schr\"odinger maps in dimensions $d\ge4$}. Adv.  Math., {\bf215}(1), 263-291, 2007.

\bibitem{bikt} I. Bejenaru, A. Ionescu, C. Kenig, D. Tataru, \emph{Global Schr\"odinger maps in dimensions $d\ge2$: Small data in the critical Sobolev spaces}.  Ann. of Math.,  {\bf173}(3), 1443-1506, 2011.

\bibitem{bikt1} I. Bejenaru, A. Ionescu, C. Kenig, D. Tataru, \emph{Equivariant Schr\"odinger maps in two spatial dimensions}. Duke Math. J., {\bf162}(11), 1967-2025, 2013.

\bibitem{bt} I. Bejenaru, T. Tataru,\emph{ Near soliotn evolution for equivariant Schr\"odigner maps  in two spatial dimensions}. Memories of American Mathematical Society, 2010.

\bibitem{csu} N.H. Chang, J. Shatah, and K. Uhlenbeck, \emph{Schr\"odinger maps}. Comm. Pure Appl. Math., {\bf53}(5), 590-602, 2000.

\bibitem{dtz} W. Ding, H. Tang, C. Zeng,  \emph{Self-similar solutions of Schr\"odinger flows},  Calc. Var.,  34(2), 267-277, 2009.


\bibitem{dw} W.Y. Ding, and Y.D. Wang, \emph{Local Schr\"odinger flow into K\"ahler manifolds}. Sci. China Math., Ser. A
{\bf44}, 1446-1464, 2001.


\bibitem{gms1} P. Germain, N. Masmoudi, J. Shatah, \emph{Global solutions for 3D quadratic Schr\"odinger equations}, Int. Math. Res. Not. IMRN 3,  414-432, 2009.

\bibitem{gms2} P. Germain, N. Masmoudi, J. Shatah, \emph{Global solutions for the gravity water wave equations in dimension 3}, Ann. of Math., 175(2), 691-754, 2012.

\bibitem{gms3} P. Germain, N. Masmoudi, J. Shatah, \emph{ Global solutions for 2D quadratic Schr\"odinger equations}, J. Math.Pure. Appl.  97(5),  505-543, 2012.

\bibitem{gsz} P. Germain, J. Shatah, and C. Zeng,  \emph{ Self-similar solutions for the Schr\"odinger map equation}. Math.
Z., 264(3), 697-707, 2010.


\bibitem{go}  J. Ginibre and T. Ozawa,  \emph{Long range scattering for nonlinear Schr\"odinger and Hartree equations in space dimension $n\ge 2$}, Comm. Math. Phys., 151,
619-645, 1993.

\bibitem{gkt} S. Gustafson, K. Kang, T.P. Tsai, \emph{Asymptotic stability of harmonic maps under the Schr\"odinger flow}. Duke Math. J.,  {\bf145}(3), 537-583, 2008.


\bibitem{gnt} S. Gustafson, K. Nakanishi, T.P. Tsai, \emph{Asymptotic stability, concentration, and oscillation
in harmonic map heat-flow, Landau-ifshitz, and Schr\"oinger maps on $R^2$}. Comm. Math. Phys. 300:1 (2010), 205-242.


\bibitem{hn} N. Hayashi, P. Naumkin, \emph{Asymptotics for large time of solutions to the nonlinear Schr\"odinger and Hartree equations}, Amer. J. Math. 120 (2), 369-389, 1998.
\bibitem{IT} M. Ifrim,  D. Tataru, \emph{Global bounds for the cubic nonlinear Schr\"odinger equation (NLS) in one space dimension}, Nonlinearity 28(8), 2661-2675, 2015.

\bibitem{ik1} A. D. Ionescu and C. E. Kenig, \emph{Low-regularity Schrodinger maps}, Differential Integral Equations 19 (2006), 1271-1300.

\bibitem{ik2} A. D. Ionescu and C. E. Kenig, \emph{ Low-regularity Schrodinger maps. II. Global well-posedness in dimen sions $d\ge 3$}, Comm. Math. Phys. {  271},  523-559, 2007.

\bibitem{ip} A. D. Ionescu, F. Pusateri, \emph{  Global solutions for the gravity water waves
system in 2d}, Invent. Math. 199, 653-804, 2015.

\bibitem{ip2} J. Kato, F. Pusateri,  \emph{   A new proof of long-range scattering for critical nonlinear Schr\"odinger equations
 }, Differential and Integral Equations, 24(10),   923-940, 2011.

\bibitem{K} S. Klainerman,  \emph{Long-time behavior of solutions to nonlinear evolution equations}, Arch. Rational Mech. Anal. 78,
73-98, 1982.

 \bibitem{ll} L. Landau, and E. Lifshitz, \emph{On the theory of the dispersion of magnetic permeability in
ferromagnetic bodies}. Phys. Z. Sovietunion, {\bf8}, 153-169, 1935.

\bibitem{LS}  H, Lindblad,  A. Soffer, \emph{Scattering and small data completeness for the critical nonlinear Schr\"odiger equation},  Nonlinearity,  19(2), 345-353, 2006.

\bibitem{li1} Z. Li,  \emph{Global Schr\"odinger map flows to K\"ahler manifolds with small data in critical Sobolev spaces: Energy critical case.} arxiv preprint.

\bibitem{li2} Z. Li, \emph{Global Schr\"odinger map flows to K\"ahler manifolds with small data in critical Sobolev spaces: High dimensions.} J. Funct. Anal., 281(6), 109093, 2021.


\bibitem{m}  H. McGahagan, \emph{An approximation scheme for Schr\"odinger maps}. Comm. Partial Differential Equations , { 32}(3), 375-400, 2007.

\bibitem{mp} J. Murphy, F. Pusateri,\emph{  Almost global existence for cubic nonlinear schr\"odinger equations in one space dimension}. Discrete Cont. Dyn.-A, 37(4),  2077-2102, 2017.

\bibitem{mrr} F. Merle, P. Raphael, I. Rodnianski, \emph{Blowup dynamics for smooth data equivariant solutions to the critical Schr\"odinger map problem.} Invent. Math., {\bf193}(2), 249-365, 2013.

\bibitem{mtt} K. Moriyama,  S. Tonegawa, Y.  Tsutsumi,  \emph{ Wave operators for the nonlinear schr\"odinger equation with a nonlinearity of low degree in one or two space dimensions}. Commun. Contemp. Math.,   5 (6),  983-996, 2003.



\bibitem{nsu} A. Nahmod, A. Stefanov, and K. Uhlenbeck, \emph{On Schr\"odinger maps}, Comm. Pure Appl. Math. {  56} (2003), 114-151.


\bibitem{nsvz} A. Nahmod, J. Shatah, L. Vega, and C. Zeng,\emph{ Schrodinger maps and their associated frame systems}. Int. Math. Res. Notices, {2007}(21) article ID rnm088, 2007.

\bibitem{o} T. Ozawa,  Long range scattering for nonlinear Schr\"odinger equations in one space dimension, Comm. Math. Phys., 139,  479-493, 1991.



\bibitem{p} G. Perelman, \emph{Blow up dynamics for equivariant critical Schr\"odinger maps}. Comm. Math. Phys., 330(1), 69-105, 2014.


\bibitem{rrs} I. Rodnianski, Y. Rubinstein, G. Staffilani, \emph{On the global well-posedness of the one-dimensional Schr\"odinger map flow}. Anal. PDE, {\bf2}(2), 187-209, 2009.


\bibitem{sh} J. Shatah, \emph{ Global existence of small solutions to nonlinear evolution equations}, J. Differential Equations 46, 409-
425, 1982.

\bibitem{shi} A. Shimomura, \emph{ Modified Wave Operators for Maxwell-Schr\"odinger Equations in Three Space Dimensions},  Annales Henri Poincar\'{e}, 4, 661-683, 2003.

\bibitem{smith} P. Smith,   \emph{Conditional global  regularity of  Schr\"odinger maps: sub-threshold dispersed energy}. Anal. PDE,   {\bf 6}(3), 601-686, 2013.

\bibitem{ss} P.L. Sulem, C. Sulem, C. Bardos, \emph{On the continuous limit for a system of classical spins}. Comm. Math. Phys., {\bf107}(3), 431-454, 1986.

\bibitem{tao} T. Tao, \emph{Geometric renormalization of large energy wave maps}. Journees equations aux derivees partielles, 1-32, 2004.

\bibitem{tao1}  T. Tao, \emph{Gauges for the schr\"odinger map}, unpublished. {\tiny http://www. math. ucla.edu/~tao/preprints/Expository}.

\bibitem{tao2}  T. Tao, \emph{Nonlinear Dispersive Equations. Local and Global Analysis}, CBMS Regional Conference Series in Mathematics, vol. 106,
American Mathematical Society, Providence, RI, 2006. Published for the Conference Board of the Mathematical Sciences, Washington, DC.


\bibitem{UT} C. Terng and K. Uhlenbeck, {\it Schr\"odinger flows on Grassmannians, in Integrable
Systems, Geometry, and Topology}, AMS/IP Studies in Advanced Mathematics,
Vol. 36, American Mathematical Society, Providence, RI, 2006,  235-256.

\bibitem{zgt} Y.  Zhou,   B.   Guo, S. Tan, \emph{ Existence and uniqueness of smooth solution for system of  ferromagnetic chain}. Science in China, Ser.A,  34(3), 157-166, 1991.

}
\end{thebibliography}
\end{document}